\pdfoutput=1
\documentclass [11pt,reqno]{amsart}
\usepackage[english]{babel}
\usepackage[utf8]{inputenc}
\usepackage[T1]{fontenc}
\usepackage{lmodern}
\usepackage{amssymb,amsmath,amsthm,amsfonts}
\usepackage{thmtools,mathtools}
\usepackage{microtype}
\allowdisplaybreaks
\usepackage{tikz,tikz-cd}
\usepackage[breaklinks,pdfencoding=auto,psdextra]{hyperref}
\usepackage{geometry}
\usepackage{tikz-cd}
\usepackage{comment} 
\def\bC{{\mathbb{C}}}

\def\bK{{\mathbb{K}}}

\def\bN{{\mathbb{N}}}

\def\bP{{\mathbb{P}}}
\def\bQ{{\mathbb{Q}}}
\def\bR{{\mathbb{R}}}

\def\bZ{{\mathbb{Z}}}

\def\cA{{\mathcal{A}}}

\def\cC{{\mathcal{C}}}

\def\cE{{\mathcal{E}}}
\def\cF{{\mathcal{F}}}

\def\cL{{\mathcal{L}}}
\def\cM{{\mathcal{M}}}

\def\cO{{\mathcal{O}}}

\def\cX{{\mathcal{X}}}

\def\fo{{\mathfrak{o}}}

\def\m{\mathfrak m}

\def\an{\operatorname{an}}
\def\ar{\operatorname{ar}}
\def\asy{\operatorname{asy}}

\def\Big{\operatorname{Big}}

\def\diff{\mathrm{d}}
\def\dc{\mathrm{d}^c}
\def\DFS{\operatorname{DFS}}
\def\Div{\operatorname{Div}}
\def\div{\operatorname{div}}

\def\FS{\operatorname{FS}}

\def\gen{\operatorname{gen}}

\def\gr{\operatorname{gr}}

\def\LSC{\operatorname{LSC}}

\def\MA{\operatorname{MA}}

\def\PDIV{\operatorname{PDIV}}
\def\Pic{\operatorname{Pic}}

\DeclareMathOperator{\PSH}{PSH}

\def\qm{\operatorname{qm}}

\def\red{\operatorname{red}}
\def\reg{\operatorname{reg}}

\def\Spec{\operatorname{Spec}}

\def\triv{\operatorname{triv}}

\def\USC{\operatorname{USC}}
\def\um{\operatorname{um}}

\def\vol{\operatorname{vol}}

\def\tilde{\widetilde}
\def\setminus{\smallsetminus}
\def\emptyset{\varnothing}

\addto\extrasenglish{}
\addto\extrasenglish{}
\addto\extrasenglish{\def\equationautorefname~#1\null{\textnormal{(#1)}\null}}
\declaretheorem[name=Theorem,refname={Theorem},style=plain,numberwithin=section]{theorem}
\declaretheorem[name=Theorem,refname={Theorem},style=plain,numbered=no]{theorem*}
\declaretheorem[name=Proposition,refname={Proposition},style=plain,sibling=theorem]{proposition}
\declaretheorem[name=Proposition,refname={Proposition},style=plain,numbered=no]{proposition*}
\declaretheorem[name=Lemma,refname={Lemma},style=plain,sibling=theorem]{lemma}
\declaretheorem[name=Lemma,refname={Lemma},style=plain,numbered=no]{lemma*}
\declaretheorem[name=Definition,refname={Definition},style=definition,sibling=theorem]{definition}
\declaretheorem[name=Definition,refname={Definition},style=definition,numbered=no]{definition*}
\declaretheorem[name=Remark,refname={Remark},style=definition,sibling=theorem]{remark}
\declaretheorem[name=Remark,refname={Remark},style=remark,numbered=no]{remark*}
\declaretheorem[name=Corollary,refname={Corollary},style=plain,sibling=theorem]{corollary}
\declaretheorem[name=Corollary,refname={Corollary},style=plain,numbered=no]{corollary*}
\declaretheorem[name=Example,refname={Example},style=definition,sibling=theorem]{example}

\declaretheorem[name=Conjecture,refname={Conjecture},style=plain,sibling=theorem]{conjecture}

\newcommand{\va}{|\cdot|}

\newcommand{\fonction}[5]{\begin{array}{l|rcl}
#1: & #2 & \longrightarrow & #3 \\
    & #4 & \longmapsto & #5 \end{array}}

\begin{document}

\title{Differentiability of the $\chi$-volume function over an adelic curve}
\date{\today}
\author{Antoine Sédillot}
\address{Universit\'e Paris Cité, Sorbonne Universit\'e, CNRS, Institut de Math\'ematiques de Jussieu-Paris Rive Gauche, F-75013 Paris, France}
\email{antoine.sedillot@imj-prg.fr}
\begin{abstract}
In this article, we show a differentiability property of the $\chi$-volume function on the ample cone of adelic line bundles over an adelic curve. This result is deduced from a non-Archimedean counterpart of the differentiability result of Witt Nyström in \cite{Nystrom13}. As an application, we give an adelic curve counterpart to the logarithmic equidistribution result of \cite{CLT09}.
\end{abstract}

\maketitle

\tableofcontents

\section{Introduction}
\subsection{Motivation}

In algebraic geometry, a fundamental problem is the study of graded linear series of a line bundle on a projective algebraic variety. Let $k$ be a field, $X$ be a projective variety over $k$, and $L$ be an invertible $\cO_X$-module. For all $n\in\bN$, we denote 
\begin{align*}
h^0(X,nL) := \dim_{k}(H^0(X,nL)). 
\end{align*}

In this context, (\cite{LazarsfeldI04}, Example 1.2.20) yields a bound $h^0(X,nL) = O(n^d)$, where $d=\dim(X)$. In the ample case, the asymptotic Riemann-Roch theorem gives the asymptotic estimate
\begin{align*}
h^0(X,nL) = \frac{\deg_L(X)}{d!}n^d + O(n^{d-1}).
\end{align*}
It is thus natural to study the class of invertible $\cO_{X}$-modules $L$ such that
\begin{align*}
\vol(L) := \displaystyle\limsup_{n\to +\infty}\frac{ h^0(X,nL)}{n^d/d!} >0. 
\end{align*}
These line bundles are called \emph{big}.  

The theory of Okounkov bodies, introduced by Okounkov in \cite{Okounkov96} and then systematically developed, independently, by Lazarsfeld and Musta\textcommabelow{t}\u{a} in \cite{LazarsfeldMustata09} and by Kaveh and Khovanskii in \cite{KK12}, gives an interpretation of the algebra of sections of a line bundle on a projective variety from the point of view of convex geometry. Let $L$ be a big line bundle on a projective variety $X$ of dimension $d$. Then one can construct a convex body in $\bR^{d}$, denoted by $\Delta(L)$, that encodes a lot of information about $L$. Theorem A of \cite{LazarsfeldMustata09} gives the equality
\begin{align*}
\vol(L) = d!\vol_{\bR^d}(\Delta(L)).
\end{align*}
This volume function is continuous and even differentiable on the big cone of $X$ (\cite{BFJ09}, \cite{LazarsfeldMustata09}). 
In this article, we study variants of this volume function in the case where the base field is endowed with additional metric data.

In Arakelov geometry, the analogue of a projective variety is an arithmetic projective variety, namely an integral scheme $\cX$, which is regular, projective and flat over $\Spec(\bZ)$, with $\dim(X) =: d+1$ The analogue of a line bundle is a hermitian line bundle, namely the data $\overline{\cL} =(\cL,\varphi)$, where $\cL$ is an invertible $\cO_{\cX}$-module and $\varphi$ is a smooth (hermitian) metric on the complex analytification of $\cL_{\bC}$. For all $n\in \bN$, for all $s\in H^{0}(\cX,n\cL) \otimes_{\bZ} \bR \subseteq H^{0}(\cX_{\bC},n\cL_{\bC})$, let
\begin{align*}
\|s\|_{n\varphi} := \displaystyle\sup_{x\in \cX_{\bC}} |s|_{n\varphi}(x).
\end{align*}
Then $\|\cdot\|_{n\varphi}$ is a norm on $H^{0}(\cX,n\cL) \otimes_{\bZ} \bR$ and $(H^{0}(\cX,n\cL),\|\cdot\|_{n\varphi})$ is an euclidean lattice. If we fix a Haar measure $\vol$ on $H^{0}(\cX,n\cL) \otimes_{\bZ} \bR$, the arithmetic Euler-Poincaré characteristic of $(H^{0}(\cX,n\cL),\|\cdot\|_{n\varphi})$ is given by 
\begin{align*}
\chi(H^{0}(\cX,n\cL),\|\cdot\|_{n\varphi}) := \ln\frac{\vol\{s \in H^{0}(\cX,n\cL)\otimes_{\bZ} \bR : \|s\|_{n\varphi} \leq 1\}}{\mathrm{covol}(H^{0}(\cX,n\cL),\|\cdot\|_{n\varphi})}.
\end{align*}
The arithmetic Hilbert-Samuel theorem states that if $\cL$ is relatively ample and if $\varphi$ is a  semi-positive metric, then the sequence
\begin{align*}
\left(\frac{\chi(H^{0}(\cX,n\cL),\|\cdot\|_{n\varphi})}{n^{d+1}/(d+1)!}\right)_{n\in\bN}
\end{align*}
converges to the arithmetic intersection product $(\overline{\cL}^{d+1})$, as defined in \cite{GilletSoule90}. In general, the quantity
\begin{align*}
\widehat{\vol}_{\chi}(\overline{\cL}) := \displaystyle\limsup_{n\to +\infty} \frac{\chi(H^{0}(\cX,n\cL),\|\cdot\|_{n\varphi})}{n^{d+1}/(d+1)!}
\end{align*}
is called the $\chi$-volume of $\overline{\cL}$. In the litterature, the $\chi$-volume is also called "sectional capacity" \cite{Rumely00}. It is linked to the arithmetic volume of $\overline{\cL}$, which is defined by
\begin{align*}
\widehat{\vol}(\overline{\cL}) := \displaystyle\limsup_{n\to +\infty} \frac{\widehat{h}^0(n\overline{\cL})}{n^{d+1}/(d+1)!},
\end{align*}
where $\widehat{h}^0(n\overline{\cL}) := \ln \#\left\{s \in H^{0}(\cX,n\cL) : \|s\|_{n\varphi} \leq 1\right\}$. The continuity of the arithmetic volume was proven by Moriwaki in \cite{Moriwaki06} and the continuity of the $\chi$-volume was shown by Ikoma in \cite{Ikoma13}. Using an arithmetic version of Siu's inequlality of Yuan \cite{Yuan08}, Chen gave the differentiability of $\widehat{\vol}$ in \cite{Chen11}. We mention as well that the study of the differentiability of a variant of the arithmetic volume was done in \cite{Ikoma18}. Equidistribution problems are known to be a remarkable application of the differentiability of various arithmetic volume functions  \cite{SUZ97}, \cite{Yuan08}, \cite{CLT09}.

More recently, an Arakelov theory over an arbitrary countable field was developed by Chen and Moriwaki in \cite{ChenMori}. Let $K$ be a field. An adelic structure over $K$ is the data $((\Omega,\cA,\nu),\phi)$ where $(\Omega,\cA,\nu)$ is a measure space and $\phi = (\va_{\omega})_{\omega\in \Omega}$ is a family of absolute values on $K$ satisfying the condition
\begin{align*}
\forall a \in K^{\times}, \quad (\omega\in\Omega) \mapsto \ln|a|_{\omega}
\end{align*} 
is $\cA$-measurable and $\nu$-integrable. In this case, we say that $S:=(K,(\Omega,\cA,\nu),\phi)$ is an adelic curve which is called proper if the product formula
\begin{align*}
\forall a\in K^{\times},\quad \int_{\Omega}\ln|a|_{\omega} \;\nu(\diff\omega) = 0 
\end{align*}
is satisfied. In this context, the avatar of a vector bundle is a so-called adelic vector bundle on $S$ : it is the data $\overline{E} := (E,\xi)$ where $E$ is vector space of finite rank over $K$ and $\xi=(\|\cdot\|_{\omega})_{\omega\in \Omega}$ is a family of norm on the $K_\omega$-vector spaces $E_{\omega} := E\otimes_K K_{\omega}$, where, for all $\omega \in \Omega$, $K_{\omega}$ denotes the completion of $K$ with respect to the absolute value $\va_{\omega}$. Moreover, the norm family $\xi$ satisfies suitable measurability and dominancy conditions (cf. \S \ref{subsub:adelic_vector_bundle}). In the framework of adelic curves, there is no obvious analogue of the arithmetic Euler-Poincaré Characteristic. A suitable replacement of the latter is the Arakelov degree, defined by
\begin{align*}
\widehat{\deg}(\overline{E}) := -\int_{\Omega} \ln\|s_1\wedge \cdots\wedge s_r\|_{\omega,\det}\;\nu(\diff\omega), 
\end{align*}
where $(s_1,...,s_r)$ is an arbitrary basis of $E$ and, for all $\omega\in\Omega$, $\|\cdot\|_{\omega,\det}$ denotes the determinant norm on  $\det(E_{\omega})$ induced by $\|\cdot\|_{\omega}$. 

We now consider the higher dimensional setting  where $\pi : X \to \Spec(K)$ is a projective $K$-scheme of dimension $d$. An adelic line bundle on $X$ is the data $\overline{L}:=(L,\varphi)$, where $L$ is an invertible $\cO_X$-module and $\varphi=(\varphi_{\omega})_{\omega\in\Omega}$ is family of (continuous) metrics on each $L_{\omega} := L \otimes_{\cO_X} \cO_{X_\omega}$, where metrics are understood in the context of Berkovich spaces. Further, the metric family $\varphi$ satisfies additional dominancy and measurability conditions (cf. \S \ref{subsub:adelic_line_bundle}). In that case, $\varphi$ is said to be an adelic metric family for $L$ and we denote by $\cM(L)$ the set of adelic metric families for $L$. If $X$ is geometrically reduced, then
\begin{align*}
\pi_{\ast}(\overline{L}) := (H^0(X,nL),(\|\cdot\|_{\varphi_{\omega}})_{\omega\in\Omega})
\end{align*}
is an adelic vector bundle over $S$, where, for all $\omega\in\Omega$, we denoted
\begin{align*}
\|\cdot\|_{\varphi_{\omega}} := \displaystyle\sup_{x\in X^{\an}_{\omega}} |\cdot|_{\varphi_{\omega}}(x).
\end{align*}
The $\chi$-volume of $\overline{L}$ is the defined as
\begin{align*}
\widehat{\vol}_{\chi}(\overline{L}) := \displaystyle\limsup_{n\to+\infty} \frac{\widehat{\deg}(\pi_{\ast}(n\overline{L}))}{n^{d+1}/(d+1)!}.
\end{align*}
We assume that either $K$ is perfect or $X$ is geometrically integral. In the case where $L$ is big and the $K$-algebra of sections $\bigoplus_{n\geq 0} H^{0}(X,nL)$ is of finite type, (\cite{ChenMoriHSP} Theorem-Definition 3.2.1) gives the convergence of the sequence 
\begin{align*}
\left(\frac{\widehat{\deg}(\pi_{\ast}(n\overline{L}))}{n^{d+1}/(d+1)!}\right)_{n\in \bN}.
\end{align*}
An arithmetic intersection theory was developed in \cite{ChenMori21}. It allows to make sense of expressions of the form
\begin{align*}
(\overline{L_0}\cdots\overline{L_d})_S,
\end{align*}
where $\overline{L_0},...,\overline{L_d}$ are so-called integrable adelic line bundles over $X$. The arithmetic Hilbert-Samuel theorem (\cite{ChenMoriHSP} Theorem 4.5.1) states that if $\overline{L}=(L,\varphi)$ is an adelic line bundle over $X$, with $L$ ample and $\varphi$ semi-positive (i.e. for all $\omega\in\Omega$, $\varphi_{\omega}$ is semi-positive), then we have the equality
\begin{align*}
\widehat{\vol}_{\chi}(\overline{L}) = (\overline{L}^{d+1})_S.
\end{align*}

Let us introduce the differentiability problem of the function $\widehat{\vol}_{\chi}$. More precisely, we denote by $\widehat{\Pic}_A(X)$ the set of adelic line bundles $\overline{L}=(L,\varphi)$ on $X$ such that $L$ is big and semiample. Then Proposition 2.14 of \cite{CM22_equidistribution} ensures that, for all $\overline{L}=(L,\varphi)\in \widehat{\Pic}_A(X)$ such that $\varphi$ is plurisubharmonic (cf. \S \ref{subsub:psh_metrics}), the function $\widehat{\vol}_{\chi}$ is Gâteaux differentiable in $\overline{L}$ along the directions of $\cM(\cO_X)$. Namely, for all $f\in \cM(\cO_X)$, the function 
\begin{align*}
(t\in \bR)\mapsto \widehat{\vol}_{\chi}(\overline{L}(tf)) 
\end{align*}
is differentiable in $0$, where $\overline{L}(f) := \overline{L} + (\cO_{X},f)$. This differentiability result is obtained thanks to a "local-global priciple". The local analogues of this result are (\cite{BermanBoucksom08}, Theorem B) for the Archimedean case and (\cite{Boucksom21}, Theorem 1.2) for the non-Archimedean one. We also mention Theorem B of the seminal work \cite{BJ20} which gives the non-Archimedean case when the base field is discretely valued.

\subsection{Results}

In this article, we extend the previous differentiability result by allowing abritrary directions in the Gâteaux differentiability: roughly speaking, we allow the underlying line bundle of the direction to be different from $\cO_{X}$. For this purpose, we should obtain a local version of such a result. Let $(k,\va)$ be a complete valued field and $X$ be a projective $k$-scheme asssumed to be normal and geometrically integral. For all big and semiample line bundle $L$ on $X$, endowed with two continuous metrics $\varphi,\psi$, the relative $\chi$-volume $\vol_{\chi}(L,\varphi,\psi)$ is defined (cf. Definition \ref{def:local_volume}). In the case where $k=\bC$ with the usual absolute value, Witt Nyström \cite{Nystrom13} obtains a (continuous) differentiability property of the relative $\chi$-volume (cf. Theorem 1.5 of \emph{loc. cit.}). The first original result of this article is a non-Archimedean analogue. The proof is essentially an adaptation of \cite{Nystrom13} in our framework. 

The proof of Theorem 1.5 of \cite{Nystrom13} makes use of an interpretation of the relative $\chi$-volume in terms of the Monge-Ampère energy and Okounkov bodies. More precisely, on $\Delta(L)$, one can construct concave functions $G_{\varphi},G_{\psi} : \Delta(L) \to \bR \cup\{-\infty\}$ such that
\begin{align}
\label{eq:volum_transformée_concave}
\vol_{\chi}(L,\varphi,\psi) = (d+1)!\int_{\Delta(L)} \left(G_{\varphi}-G_{\psi}\right)\diff\lambda,
\end{align}
where $d=\dim(X)$ and $\lambda$ denotes the Lebesgue measure on $\bR^d$ (cf. \S \ref{subsub:transformée_concave_R_diviseur}). 
We additionally assume that $L$ is ample and that the metrics $\varphi,\psi$ are plurisubharmonics (cf. Definition \ref{def:psh_metric_unified}). In \cite{BoucksomEriksson}, Boucksom and Eriksson proved the equality
\begin{align}
\label{eq:énergie_et_volume}
\vol_{\chi}(L,\varphi,\psi) = (d+1)\cE(\varphi,\psi),
\end{align}
where
\begin{align*}
\cE(\varphi,\psi) := \frac{1}{d+1}\displaystyle\sum_{i=0}^{d}\int_{X^{\an}}(\varphi-\psi)(\diff\dc \varphi)^{\wedge i}(\diff\dc \psi)^{\wedge d-i}
\end{align*}
denotes the Monge-Ampère energy associated to the metrics $\varphi,\psi$ (cf. \emph{loc. cit.}, Theorem A). When $k$ is non-Archimedean, the expressions of the form $(\diff\dc \varphi)^{\wedge i}(\diff\dc \psi)^{\wedge d-i}$, for $i=0,...,d$, are understood in  the sense of the theory intiated by Chambert-Loir and Ducros in \cite{CLD12} (in the framework of general Berkovich spaces) and further extended by Gubler and Künnemann in \cite{GK17} (in the case of analytifications of algebraic varieties). More generally, the equality (\ref{eq:énergie_et_volume}) is true when $L$ is only assumed to be semiample and big, for an arbitrary complete valued field $k$ (cf. \cite{Boucksom21}, Theorem 4.6 for the non-Archimedean case and \cite{BermanBoucksom08}, Theorem A in the Archimedean case). These ingredients allow us to prove the following result (cf. Theorem \ref{th:local_differentiability_general} for a more general statement).

\begin{theorem}
\label{th:local_differentiability_intro}
Let $k$ be a complete valued field, $X$ be a projective, geometrically integral and normal $k$-scheme, $L$ be an ample invertible $\cO_{X}$-module, and $\varphi,\psi$ two continuous plurisubharmonic metrics on $L$. Then for any invertible $\cO_{X}$-module $M$ endowed with two continuous metrics $\Phi,\Psi$, the limit
\begin{align*}
\displaystyle\lim_{\bQ \ni t \to 0} \frac{\vol_{\chi}(L+tM,\varphi+t\Phi,\psi+t\Psi)-\vol_{\chi}(L,\varphi,\psi)}{t}
\end{align*}
exists.
\end{theorem}
The positivity condition on the metrics $\varphi,\psi$ can be relaxed, under a so-called "continuity of envelope" property, which is known to hold in various cases of application (cf. \S \ref{subsub:envelopes}). 

We now move on to the differentiability problem over an adelic curve. We fix a proper adelic curve $S=(K,(\Omega,\cA,\nu),\phi)$, with $K$ a  countable field, and a projective, geometrically integral and normal $k$-scheme $X$. For any ample invertible $\cO_{X}$-module $L$, for any $\varphi,\psi\in\cM(L)$, Proposition 2.14 of \cite{CM22_equidistribution} gives the equality
\begin{align*}
\widehat{\vol}_{\chi}(L,\varphi,\psi) := \widehat{\vol}_{\chi}(L,\varphi)-\widehat{\vol}_{\chi}(L,\psi) = \int_{\Omega} \vol_{\chi}(L_{\omega},\varphi_{\omega},\psi_{\omega})\nu(\diff\omega).
\end{align*}
This "local-global principle", or rather a stronger version of this fact (Propoisition \ref{prop:complement_chi_volume}), allows to prove the main result of this article (cf. Theorem \ref{th:global_differentiability} for a more general statement).
\begin{theorem}
\label{th:global_differentiability_intro}
Let $S=(K,(\Omega,\cA,\nu),\phi)$ be a proper adelic curve, with $K$ a perfect countable field. Let $X$ be a projective, geometrically integral and normal $k$-scheme. Let $\overline{L}=(L,\varphi)$ be relatively ample adelic line bundle on $X$, namely $L$ is an ample invertible $\cO_X$-module and $\varphi$ is a plurisubharmonic adelic metric family on $L$. Let $\overline{M}=(M,\Phi)$ be any adelic line bundle on $X$. Then
\begin{align*}
\displaystyle\lim_{\bQ\ni t\to 0} \frac{\widehat{\vol}_{\chi}(\overline{L}+t\overline{M})-\widehat{\vol}_{\chi}(\overline{L})}{t} = (d+1)\left(\overline{L}^{d}\cdot \overline{M}\right)_{S}.
\end{align*}
\end{theorem}

As an application, we obtain the following logarithmic equidistribution result, which is a counterpart of Théorème 1.2 of \cite{CLT09} in the framework of adelic curves. Let $\overline{L}=(L,\varphi)$ be a relatively ample adelic line bundle on $X$. Let $(x_n)_{n\geq 0}$ be a sequence of closed points of $X$ which is assumed to be \emph{generic} and \emph{small} (see \S \ref{sub:logarithmic_equidistribution_points} for the definition). For any $\omega\in\Omega$, for any $n\in \bN$, let $\delta_{\overline{L},x_n,\omega}$ denote the probability counting measure on the finite set $\{x_n\}^{\an}_{\omega}$. We also denote by $\delta_{\overline{L},X,\omega}$ the probability measure $(1/\vol(L))(\diff\dc\varphi)^{\wedge d}$.

\begin{theorem}
\label{th:logarithmic_equidistribution_intro}
We assume that $M$ admits a non-zero global section $s$. Denote $D:=\div(s)$.
\begin{itemize}
	\item[(1)] We have
	\begin{align*}
	h_{\overline{L}}(D) \geq h_{\overline{L}}(X).
	\end{align*}
	\item[(2)] We further assume that $h_{\overline{L}}(D) = h_{\overline{L}}(X)$. Then we have 
	\begin{align*}	
	\displaystyle\lim_{n\to+\infty} \int_{\Omega}\left(\int_{X_{\omega}^{\an}} (-\ln|s|_{\Phi_{\omega}})\delta_{\overline{L},x_n,\omega}\right)\nu(\diff\omega) = \int_{\Omega}\left(\int_{X_{\omega}^{\an}} (-\ln|s|_{\Phi_{\omega}})\delta_{\overline{L},X,\omega}\right)\nu(\diff\omega).
	\end{align*}
\end{itemize}
\end{theorem}

\subsection*{Content} This article is divided in three parts. In \S \ref{sec:preliminaries}, we recall several facts about non-Archimedean pluripotential theory, Okounkov bodies. In \S \ref{sec:local_differentiability} we prove Theorem \ref{th:local_differentiability_intro} and in \S \ref{sec:differentiability_adelic_curve} we give a proof of Theorem \ref{th:global_differentiability_intro}.

\subsection*{Acknowledgements} The author would like to thank Huayi Chen for proposing this subject and for his support throughout the redaction of this text. We also thank Mathieu Daylies and Nuno Hultberg for useful remarks.

\subsection*{Conventions and notation}

For any subgroup $\Gamma \subset \bR_{>0}$, we denote $\sqrt{\Gamma} := \{r^{1/n} : r \in \Gamma, n\in\bN\}$.

We identify line bundle and invertible modules for which we use additive notation. 

If $X$ is an integral scheme, we denote by $k(X)$ its field of rational functions.

Let $S=(K,(\Omega,\cA,\nu),(\va_{\omega})_{\omega\in\Omega})$ be an adelic curve. For any $\omega \in \Omega$, we denote by $K_{\omega}$ the completion of $K$ with respect to the absolute value $\va_{\omega}$. A function $f: \Omega\to \bR$ is said to be $\nu$\emph{-dominated} if there exists a $\nu$-integrable function $C : \Omega\to \bR_{\geq 0}$ such that $|f|\leq C$ $\nu$-almost everywhere.

\section{Preliminaries}
\label{sec:preliminaries}
\subsection{Nets}

Let $X$ be a topological space. Let $(A,\leq)$ be a directed set (namely $(A,\leq)$ is an ordered set such that $x,y\in A$, there exists $z\in A$ such that $x\leq z$ et $y\leq z$). A \emph{net} in $X$ is any family $x=(x_{\alpha})_{\alpha\in A}\in X^{A}$. A net $x=(x_{\alpha})_{\alpha\in A}$ is called
\begin{itemize}
	\item[(1)] \emph{convergent} in $X$ if there exists $l\in X$ such that, for any open neighbourhood $U$ of $l$ in $X$, there exists $\alpha_0\in A$ such that $x_{\alpha} \in U$ for all $\alpha\geq \alpha_0$;
	\item[(2)] \emph{increasing}, resp. \emph{decreasing}, if $X= \bR$ and for any $\alpha,\beta\in A$ such that $\alpha \leq \beta$ the inequality $x_{\alpha} \leq x_{\beta}$, resp. $x_{\beta} \leq x_{\alpha}$, holds.
	\end{itemize}

\subsection{Upper semi-continuous envelope}

Throughout this paragraph, we fix a topological space $X$. A function $f : X \to \bR \cup \{-\infty,+\infty\}$ is called \emph{upper semi-continuous}, resp. \emph{lower semi-continuous}, \emph{usc} for short, resp. \emph{lsc} for short , if for any $\alpha\in\bR$, the set $\{x\in X : f(x)\geq \alpha\}$, resp. $\{x\in X : f(x)\leq \alpha\}$, is closed. Let $\USC(X)$, resp. $\LSC(X)$, be the class of usc, resp. lsc, functions on $X$.

\begin{lemma}[\cite{BouTG}, Ch.IV \S 6 Proposition 2 et Théorèmes 3 et 4]
\label{lemma:usc_lsc_functions}
\begin{itemize}
	\item[(1)] $\USC(X)$ and $\LSC(X)$ are lattices. 
	\item[(2)] For any family $(f_i)_{i\in I}$ in $\USC(X)$, resp. in $\LSC(X)$, the function $f : (x\in X) \mapsto \inf\{f_i(x) : i\in I\}$ is usc, resp. the function $f : (x\in X) \mapsto \sup\{f_i(x) : i\in I\}$ is lsc.
	\item[(3)] If $X$ is compact Hausdorff, any $f\in \USC(X)$ is bounded from above and attains its maximum. Likewise, $f \in \LSC(X)$ is bounded from below and attains its minimum.
\end{itemize}
\end{lemma}

If $f : X \to \bR \cup \{-\infty,+\infty\}$ is an arbitrary function, its \emph{upper semi-continuous envelope}, \emph{usc envelope} for short, is defined by
\begin{align*}
\forall x \in X,\quad f^{\star} = \inf\{g : g\in \USC(X) \text{ et } g\geq f\} \in \bR\cup \{-\infty,+\infty\}.
\end{align*}
By Lemma \ref{lemma:usc_lsc_functions} $(1)$, $f^{\star} \in \USC(X)$.
 
\begin{lemma}
\label{lemma:psh_envelope_non_infinite}
Let $f:X\to \bR\cup\{+\infty\}$ and assume that $f$ is locally bounded from above on $U := \{x\in X : f(x)<+\infty\}$. Then $f^{\star}$ takes real values on $U$, in particular $f^{\star}\not\equiv +\infty$.
\end{lemma}

\begin{proof}
Let $x\in U$, there exist an open neighbourhood $x\in V \subset U$ and a constant $C\in \bR$ such that $f_{|V} \leq C$. By definition of $f^{\star}$, we have $f^{\star}(x) \leq \sup_{V} f \leq C$.
\end{proof}

\subsection{Pluripotential theory over Berkovich spaces}
\label{sub:pluripotential}

Throughout this paragraph, we fix a complete valued field $(k,|\cdot|)$. 

\subsubsection{Analytification in the sense of Berkovich}

Let $X$ be a $k$-scheme. Its \emph{analytification in the sense of Berkovich}, denoted by $X^{\an}$, is defined in the following way: a point $x\in X^{\an}$ is the data $(p,\va_{x})$ where $p \in X$ and $\va_{x}$ is an absolute value on $\kappa(x)$ extending the absolute value on $k$. $X^{\an}$ can be endowed with the Zariski topology: namely the coarsest topology on $X^{\an}$ such that the first projection $j : X^{\an} \to X$ is continuous. There exists a finer topology on $X^{\an}$, called the Berkovich topology: it is the initial topology on $X^{\an}$ with respect to the family defined by $j : X^{\an} \to X$ and the applications
\begin{align*}
\fonction{|f|_{\cdot}}{U^{\an}}{\bR_{\geq 0},}{x}{|f|_x,}
\end{align*}
where $U^{\an}$ is of the form $U^{\an}:=j^{-1}(U)$, with $U$ a Zariski open subset of $X$, and $f\in \cO_{X}(U)$. Endowed with this topology, $X^{\an}$ is a locally compact topological space. There are GAGA type results: namely, $X$ is separated, resp. proper iff $X^{\an}$ is Hausdorff, resp. compact Hausdorff. If $X$ is a scheme of finite type, $X^{\an}$ can be endowed with a sheaf of analytic functions.

%Dans le cas où $k$ est non-archimédien, on peut voir $X^{\an}$ comme l'ensemble des couples $x=(p,v_x)$ où $p\in X$ et $v$ est une semi-valuation $v_x$ sur $X$ de centre $p$, i.e. $v_x$ définit une valuation sur $\kappa(x)$ qui est positive sur $\cO_{X,p}$. 
We now assume that $k$ is non-Archimedean, we denote by $\fo_k$ its valuation ring, $\m_k \subset \fo_k$ the maximal ideal and $\tilde{k}$ the residue field. Let $X$ be an integral projective $k$-scheme.  We call \emph{model} of $X$ any  integral projective $\fo_k$-scheme $\cX$ whose generic fibre is $X$. If $L$ is an invertible $\cO_X$-module, we say that $(\cX,\cL)$ is a \emph{model} of $(X,L)$ if $\cX$ is a model of $X$ and the restriction of $\cL$ to $X$ is $L$. A model $\cX$ of $X$ is said to be \emph{coherent}, resp. \emph{flat}, if $\cX \to \Spec(\fo_k)$ is of finite presentation, resp. flat. Let $\cX$ be a model of $X$, the valuative criterion of properness yields a \emph{reduction} morphism which is anticontinuous (i.e. the inverse image of an open set is closed)
\begin{align*}
\red_{\cX} : X^{\an} \to \cX_s.
\end{align*}

A point $x\in X^{\an}$ is said to be \emph{divisorial}, or \emph{Shilov}, if there exists a model $\cX$ of $X$ such that $\red_{\cX}(x)$ is a generic point of the special fibre $\cX_s$. We denote by $X^{\div}$ the set of divisorial points of $X^{\an}$. Note that in the trivially valued case, $X^{\div} = \{x_{\triv}\}$, where $x_{\triv}$ is the point corresponding to the trivial absolute value on $k(X)$.

For any $x\in X^{\an}$, the completed residue field $\widehat{\kappa}(x)$ is non-Archimedean and let $\tilde{\widehat{\kappa}(x)}$ be its residue field. Let
\begin{align*}
s(x) := \mathrm{tr.deg}(\tilde{\widehat{\kappa}(x)}/\tilde{k}), \quad t(x) := \dim_{\bQ}(\sqrt{|\widehat{\kappa}(x)^{\times}|_{x}}/\sqrt{|k^{\times}|}),\quad d(x) := s(x) + t(x).
\end{align*}
Abhyankar inequality yields $d(x) \leq \dim(X)$. The point $x$ is said to be \emph{quasimonomial} if $d(x)=\dim(X)$ and we denote by $X^{\qm}$ the set of quasimonomial points of $X$. We have the inclusion $X^{\div}\subset X^{\qm}$ and $X^{\qm}$ is dense in $X^{\an}$ (\cite{Poineau13b}, Proposition 4.7).   

\subsubsection{Metrics}
\label{subsub:metrics}

Throughout this paragraph, we fix a $k$-scheme $X$.

Let $L$ be a line bundle on $X$. A \emph{metric} on $L$ is a family $\varphi :=(\va_{\varphi}(x))_{x\in X^{\an}}$, where 
\begin{align*}
\forall x\in X^{\an},\quad \va_{\varphi}(x) : L(x) := L \otimes_{\cO_X} \widehat{\kappa}(x) \to \bR_{\geq 0}
\end{align*}
is a norm on the $\widehat{\kappa}(x)$-vector space $L(x)$. The metric $\varphi$ is called \emph{continuous} if for all $U\subset X$ open and for all $s\in H^{0}(U,L)$, the map $|s|_{\varphi} : U^{\an} \to \bR_{\geq 0}$ is continuous with respect to the Berkovich topology. 

We use the following notation for metrics on line bundles.
\begin{itemize}
	\item[(i)] If $L$ is a line bundle on $X$, we denote by $C^0(L)$ the class of continuous metrics on $L$.
	\item[(ii)] If $L,M$ are two line bundles on $X$, respectively endowed with metrics $\varphi,\psi$, the family $\varphi+\psi:=(\va_{\varphi+\psi}(x))_{x\in X^{\an}}$ defined by
	\begin{align*}
	\forall x\in X^{\an},\quad\va_{\varphi+\psi}(x) = \va_{\varphi}(x) \otimes \va_{\psi}(x),
	\end{align*}
is a metric on $L+M$. Likewise, the family $-\varphi$ defines a metric on $-L$, considering the corresponding dual norm family. If $(\varphi,\psi)\in C^0(L)\times C^0(M)$, resp. $\varphi \in C^0(L)$, then $\varphi + \psi \in C^0(L+M)$, resp. $-\varphi\in C^0(-L)$.
	\item[(iii)] Let $\varphi$ be a metric on $\cO_X$. We identify $\varphi$ with the function $-\ln |1|_{\varphi} : X^{\an} \to \bR$.
\end{itemize}

Thus, for any two metrics $\varphi,\psi$ on a line bundle $L$, $\varphi-\psi$ is considered as a function on $X^{\an}$. If $\varphi$ is a metric on a line bundle $L$, it is said to be \emph{bounded from below}, resp. \emph{bounded from above}, resp. \emph{bounded}, if there exists a continuous metric $\psi$ on $L$ such that $\varphi-\psi$ is bounded from below, resp. bounded from above, resp. bounded. Note that in the case where $X$ is proper and $\varphi$  is bounded from below (or from above), for all $\psi'\in C^0(L)$, the function $\varphi-\psi'$ is bounded from below (or from above). 

A \emph{singular metric} on a line bundle $L$ is a metric of the form $\varphi = \phi + f$, where $\phi\in C^0(L)$ and  $f : X^{\an} \to \bR\cup\{-\infty,+\infty\}$. Namely, for any local section $s$ of $L$, we have
\begin{align*}
|s|_{\varphi}(x) := |s|_{\phi}(x)e^{-f(x)} \in [0,+\infty].
\end{align*}
Most of the time, we consider the case where $f : X^{\an} \to \bR\cup\{-\infty\}$. In that case, for some $x\in X^{\an}$, the "norm" $\va_{\varphi}(x) \equiv +\infty$. 

In this setting, the singular metric $\varphi$ is called \emph{upper semicontinuous}, resp. \emph{lower semicontinuous}, if $f$ is so. In the case where $X^{\an}$ is compact Hausdorff, an usc (upper semicontinuous) singular metric is bounded from above (cf. Lemma \ref{lemma:usc_lsc_functions}).  If $\varphi$ is a singular metric on the line bundle $L$, for all integer $m\geq 1$, we have a singular metric $m\varphi$ on $mL$. If $\varphi,\psi$ are respectively singular metrics on line bundles $L,M$, not reaching $+\infty$, $\varphi + \psi$ then defines a singular metric on $L+M$. Thus, the notion of singular metric makes sense for any $\bQ$-line bundle on $X$ and therefore do all the previously introduced properties of metrics. 

\begin{example}
\label{example:singular_metric_global_section}
We assume that the line bundle $L$ has a non zero global section $s \in H^{0}(X,L)$. We define a singular metric $\psi_{s}$ on $L$ defined by \begin{align*}
\psi_{s} := \varphi + \ln |s|_{\varphi},
\end{align*}
where $\varphi$ is an arbitrary continuous metric on $L$. For any $\varphi\in C^0(L)$, the function $\ln|s|_{\varphi} : X^{\an} \to \bR\cup\{-\infty\}$ reaches $-\infty$ exactly on $\div(s) = \{x\in X : s(x)=0\}$. Note that, on the open set $\{s \neq 0\} \subset X$, $\psi_{s}$ defines a continuous metric. Moreover, the definition of $\psi_{s}$ is independent on the choice of the continuous metric $\varphi$. 
\end{example}

\begin{definition}
\label{def:FS_metric}
Let $\varphi$ be a metric on a line bundle $L$ on $X$. $\varphi$ is called \emph{Fubini-Study} if there exist an integer $m\geq 1$, a family of global sections of $mL$ $(s_i)_{i\in I}$, with $I$ a finite set, without common zeros (thus we assume that $mL$ is globally generated), and a family $(\lambda_i)_{i\in I}\in\bR^{I}$ such that 
\begin{align*}
\varphi = \left(\{\begin{matrix}
\frac{1}{2m}\ln \displaystyle\sum_{i\in I}e^{2\psi_{s_i}+2\lambda_i}\quad \mathrm{if}\; k \;\mathrm{is}\;\mathrm{Archimedean}\\
\frac{1}{m} \max_{i\in I} \{\psi_{s_i} + \lambda_i\} \quad\mathrm{if}\; k \;\mathrm{is}\;\mathrm{non-Archimedean}.
\end{matrix}\right.
\end{align*}
For any additive subgroup $\Gamma \subset \bR$, we denote by $\FS_{\Gamma}(L)$ the set of Fubini-Study metrics $\varphi$ on $L$ written as above with $\lambda_i\in \Gamma$ for all $i\in I$. A metric in $\FS_{\{0\}}(L)$ is said to be \emph{pure}. We denote $\FS(L):= \FS_{\bR}(L)$. Note that for a line bundle $L$, $\FS(L)\neq\emptyset$ implies that $L$ is semiample, namely there exists an integer $m\geq 1$ such that $mL$ is globally generated.
\end{definition}

\begin{definition}
\label{def:DFS_metric}
Let $L$ be an invertible $\cO_X$-module. A metric $\varphi$ on $L$ is called \emph{DFS} if there exist two invertible $\cO_X$-modules $L_1,L_2$ and $(\varphi_1,\varphi_2)\in\FS(L_1)\times\FS(L_2)$ such that $L=L_1-L_2$ and $\varphi=\varphi_1-\varphi_2$. The set of DFS metrics on $L$ is denoted by $\DFS(L)$.
\end{definition}

\begin{remark}
If $L$ is an invertible $\cO_X$-module, $\DFS(L)$ is the analogue of smooth metrics in the framework of complex geometry. Morover, (\cite{BoucksomEriksson}, Proposition 5.4) implies that DFS and Fubini-Study metrics are preserved by pullback, sum, and maxima (in the case where $k$ non-Archimedean).
\end{remark}

\begin{theorem}[\cite{BoucksomEriksson}, Theorem 5.18]
\label{th:BJ18_2.8}
For any invertible $\cO_X$-module $L$, $\DFS(L)$ is dense in $C^0(L)$ (with respect to the uniform convergence).
\end{theorem}

In the case where $k$ Archimedean, a Fubini-Study metric is smooth. In the non-Archimedean case, they are closely related to model metrics (cf. \cite{ChenMori}, \S 2.3).

\begin{theorem}[\cite{BoucksomEriksson}, Theorem 5.14]
We assume that $k$ is non-Archimedean. Let $\varphi$ be a continuous metric on a $\bQ$-line bundle $L$ on $X$. We have the equivalence
\begin{itemize}
	\item[(i)] $\varphi$ is a pure Fubini-Study metric;
	\item[(ii)] $\varphi$ is a model metric induced by a semiample $\bQ$-model $\cL$ of $L$.
\end{itemize}
\end{theorem}

We will make use of the following inequality when studying plurisubharmonic envelopes (cf. \S \ref{subsub:envelopes}).

\begin{theorem}[\cite{BoucksomJonsson18}, Theorem 2.21]
\label{th:Izumi_global}
We assume that $k$ is non-Archimedean. Let $L$ be an invertible $\cO_X$-module, $\varphi\in\DFS(L)$, and $x_1,x_2 \in X^{\qm}$. Then there exists a constant $C=C(L,x_1,x_2,\varphi)>0$ such that 
\begin{align*}
\forall \psi\in\FS(L),\quad |(\psi-\varphi)(x_1)-(\psi-\varphi)(x_2)|\leq C.
\end{align*}
\end{theorem}

Let $\varphi$ be a bounded from above singular metric on a line bundle $L$. Let 
\begin{align*}
\forall s\in H^{0}(X,L),\quad \|s\|_{\varphi} := \displaystyle\sup_{x\in X^{\an}} |s|_{\varphi}(x) \in \bR_{\geq 0},
\end{align*}
it defines a seminorm on $H^{0}(X,L)$. If $X$ is integral and $\varphi$ is non identically zero, then $\|\cdot\|_{\varphi}$ is a norm on $H^{0}(X,L)$.

\begin{lemma}
\label{lemma:sup_norm}
We assume that $X$ is reduced. Let $\varphi$ be a singular metric on an invertible $\cO_X$-module $M$ written as $\varphi = \phi + f$, where $\phi\in C^{0}(M)$ and $f : X^{\an} \to \bR\cup\{+\infty\}$ is a continuous function such that $f\not\equiv + \infty$. Let $R>0$, we denote $\varphi_R := \min\{\phi+R,\varphi\}$, namely for any local section $u$ of $M$, we have 
\begin{align*}
|u|_{\varphi_{R}}(x) = |u|_{\phi}(x)\max \{e^{-R},e^{-f(x)}\}.
\end{align*}
\begin{itemize}
	\item[(1)] $\varphi_R \in C^0(M)$.
	\item[(2)] Let $L$ be another invertible $\cO_X$-module and $\psi\in C^0(L)$. Then for all $R>0$ sufficiently big, we have
	\begin{align*}
\|\cdot\|_{\psi+\varphi_{R}} = \|\cdot\|_{\psi+\varphi}. 
\end{align*}
\end{itemize}
\end{lemma}

\begin{proof}
The first statement is implied by the continuity of $f$. For all $u\in H^{0}(X,L+M)\setminus\{0\}$, for all $R>0$, for all $x\in X^{\an}$, we have 
\begin{align*}
|u|_{\psi+\varphi}(x) &= |u|_{\psi+\phi}(x)e^{-f(x)},\\
|u|_{\psi+\varphi_R}(x) &= |u|_{\psi+\phi}(x)\max\{e^{-R},e^{-f(x)}\}.
\end{align*} 
We define $\alpha : (x\in X^{\an}) \mapsto |u|_{\psi+\phi}(x)$, $\beta : (x\in X^{\an}) \mapsto e^{-f(x)}$ and, for any integer $n\geq 0$, $\beta_n := \max\{\beta,e^{-n}\}$. These are continuous real functions on $X^{\an}$. Note that 
\begin{align}
\label{eq:inequality_sup_norms}
\forall n \geq 0, \quad \displaystyle\max_{x\in X^{\an}} \alpha(x)\cdot \beta_n(x) = \|u\|_{\psi+\varphi_n} \geq \max_{x\in X^{\an}} \alpha(x)\cdot \beta(x) = \|u\|_{\psi+\varphi} >0.
\end{align}

For any integer $n\geq 0$, let $x_n \in X^{\an}$ be such that $\alpha(x_n)\beta_n(x_n)=\|u\|_{\psi+\varphi_n}$.Then the sequence $(\beta(x_n))_{n\geq 0}$ is bounded from below by a positive real number. Indeed, assume that there exists a subsequence $(\beta(x_{n_k}))_{k\geq 0}$ which converges to $0$. As $f$ is continuous and non-negative on $X^{\an}$, we have 
\begin{align*}
\forall k \geq 0,\quad 0\leq \|u\|_{\psi+\varphi_{n_k}} \leq \displaystyle\max_{x\in X^{\an}} f(x) \max\{e^{-n},g(x_{n_k})\}.
\end{align*}
Henceforth, $\lim_{k\to +\infty}\|u\|_{\psi+\varphi_{n_k}}=0$, which contradicts (\ref{eq:inequality_sup_norms}).

Let $m>0$ such that $m\leq \beta(x_n)$ for all $n\geq 0$. Let $n$ be an integer such that $e^{-n} \leq \beta(x_n)$, we have
\begin{align*}
\|u\|_{\psi+\varphi_{n}} = \alpha(x_n)\max\{e^{-n},\beta(x_{n})\} = \alpha(x_n)\beta(x_n) \leq \|u\|_{\psi+\varphi} \leq \|u\|_{\psi+\varphi_{n}}.
\end{align*}
For all $R\geq n$, we have $\|u\|_{\psi+\varphi}\leq \|u\|_{\psi+\varphi_{R}}\leq \|u\|_{\psi+\varphi_{n}} \leq\|u\|_{\psi+\varphi}$, hence the conclusion.
\end{proof}

\begin{remark}
We will make use of Lemma \ref{lemma:sup_norm} in the case where $M=-A$ is the dual of an ample line bundle $A$ and $\varphi = -\psi_{s}$ is the dual of the singular metric defined in Example \ref{example:singular_metric_global_section}, where $s\in H^0(X,A)$.
\end{remark}

\subsubsection{Green functions}
\label{subsub:Green_functions}

In this article, it is convenient to work in the setting of $\bR$-divisors instead of ($\bQ$-)line bundles. We fix an integral scheme $X$ over $\Spec(k)$, and let $\bK$ be either $\bZ$, $\bQ$ or $\bR$. $\Div_{\bK}(X) := \Div(X) \otimes_{\bZ} \bK$ denotes the $\bK$-module of $\bK$\emph{-Cartier divisors} on $X$. Recall that for all $D \in \Div_{\bK}(X)$, for all $x\in X$, there exists an open neighbourhood $U$ of $x$ and a section $f\in (\cM_{X}^{\times} \otimes_{\bZ} \bK)(U)$ such that $D$ is defined by $f$ on $U$. In that case, $f$ is called a  \emph{local equation} of $D$ in $x$. Moreover, there is an injective $\bK$-linear morphism $\div_{\bK} : K(X)^{\times} \to \Div_{\bK}(X)$ whose image is denoted by $\PDIV_{\bK}(X)$. If $D\in \Div_{\bK}(X)$, the set of its $\bK$-global sections is denoted by $H^{0}_{\bK}(D)$ (cf. e.g. \cite{ChenMori}, Definition 2.4.12). By abuse of notation, we let $H^{0}(D) := H^{0}_{\bZ}(D)$.

Recall there exists an isomorphism
\begin{align*}
\cO_X(\cdot) : \Div_{\bZ}(X)/\PDIV_{\bZ}(X) \to \Pic(X),
\end{align*}
inducing an isomorphism $H^0(D) \cong H^0(X,\cO_X(D))$ for all $D\in \Div_{\bZ}(X)$. 

For a Cartier divisor, the analogue of a continuous metric is called a Green function. We start by a preliminary definition. Let $C^0_{\gen}(X)$ be the set of continuous functions on a non-empty Zariski open subset of $X^{\an}$ modulo the equivalence relation 
\begin{align*}
f \sim g \Leftrightarrow f \text{ and } g \text{ coincide on a non-empty Zariski open subset}.
\end{align*}
Let $D\in \Div_{\bZ}(X)$, a \emph{Green function} of $D$ is a class $g\in C^0_{\gen}(X)$ such that, for any local equation $f$ of $D$ over a non-empty Zariski open subset $U$, $g+\ln|f| \in C^0_{\gen}(X)$ extends to a continuous function on $U^{\an}$, namely $g+\ln|f|$ belongs to the image of the canonical homomorphism $C^{0}(U^{\an}) \to C_{\gen}^{0}(X^{\an})$. The set of pairs $(D,g)$ with $D \in \Div_{\bZ}(X)$ and $g$ is Green function of $D$ is denoted by $\widehat{\Div}_{\bZ}(X)$. Let $D$ be a Cartier divisor on $X$, there is correspondence between Green functions of $D$ and continuous metrics on $\cO_X(D)$: if $(D,g)\in\widehat{\Div}_{\bZ}(X)$ and if $f$ is a local equation $D$ on a non-empty Zariski open subset $U\subset X$, $e^{-(g+\ln|f|)}$ extends to a continuous function on $U^{\an}$, defining a continuous metric on $\cO_X(D)_{|U}$. By glueing these local metrics, we obtain a continuous metric on $\cO_{X}(D)$ denoted by $\varphi_g$. Conversely, if $\varphi$ is a continuous metric on $\cO_X(D)$ and if $s_D$ is the rational section of $\cO_X(D)$ defined by $1 \in K(X)^{\times}$, then the class of $-\ln|s_D|_{\varphi}$ in $C^0_{\gen}(X)$ yields a Green function of $D$. If $f\in H^0(D)$, $|f|e^{-g}\in C_{\gen}^{0}(X^{\an})$ extends to a continuous function on $X^{\an}$ denoted by $|f|_g$. We let \begin{align*}
\|f\|_g := \displaystyle\max_{x\in X^{\an}} |f|_g(x).
\end{align*}
Note that, under the isomorphism $H^0(D)\cong H^0(X,\cO_X(D))$, the norm $\|\cdot\|_g$ coincides with the norm $\|\cdot\|_{\varphi_g}$.

Let $\bK$ be either $\bQ$ or $\bR$, we denote by $\widehat{\Div}_{\bK}(X)$ the quotient $\widehat{\Div}_{\bZ}(X) \otimes_{\bZ} \bK$ by the vector subspace spanned by the elements of the form $\lambda (D,g) - (\lambda D,\lambda g)$, where $\lambda\in\bK$ and $(D,g)\in\widehat{\Div}_{\bZ}(X)$. Let $(D,g) \in \widehat{\Div}_{\bK}(X)$, for all $f \in H^0_{\bK}(D)$, $|f|e^{-g}$ extends to a continuous function on $X^{\an}$ and we denote by $\|f\|_g$ its maximum.

\subsubsection{Plurisubharmonic metrics}
\label{subsub:psh_metrics}

Throughout this paragraph, we fix a projective scheme $X$ over $\Spec(k)$.

\begin{definition}
\label{def:psh_metric_unified}
Let $L$ be a semiample $\bQ$-line bundle on $X$ and $\varphi$ be a (possibly singular) metric on $L$. $\varphi$ is called \emph{plurisubharmonic} (\emph{psh} for short) if $\varphi \not\equiv -\infty$ and if $\varphi$ is decreasing limit of a net of Fubini-Study metrics on $L$. The class of psh metrics on $L$ is denoted by $\PSH(L)$.
\end{definition}

\begin{remark}
\begin{itemize}
	\item[(1)] When $k$ is Archimedean, this definition of psh metric is equivalent to the usual definition of a psh metric on a semiample $\bQ$-line bundle on a projective scheme over $\Spec(k)$ (\cite{BoucksomEriksson}, Theorem 7.1).
	\item[(2)] If $L$ is an invertible $\cO_{X}$-module, the class $\PSH(L) \cap C^0(L)$ coincide with the class of \emph{semi-positive} continuous metrics in the sense of (\cite{ChenMori}, \S 2.3).
\end{itemize}
\end{remark}

\begin{lemma}[\cite{BoucksomEriksson}, Lemma 7.3]
Let $L$ be a semiample $\bQ$-line bundle on $X$. Then there exists a projective scheme $Y$ over $\Spec(k)$, a surjective morphism $f : X \to Y$ over $\Spec(k)$ with $f_{\ast}\cO_{X} = \cO_{Y}$ and an ample $\bQ$-line bundle $A$ on $Y$ such that $L = f^{\ast}A$. Furthermore, $f$ and $(Y,A)$ are unique up to isomorphism and $\PSH(L) \cong \PSH(A)$.
\end{lemma}

\begin{proposition}[\cite{BoucksomJonsson21}, Proposition 5.6 et \cite{BoucksomEriksson}, Corollary 7.5]
\label{prop:PSH_metrics}
Using the same notation as in Definition \ref{def:psh_metric_unified}, we have the following:
\begin{itemize}
	\item[(i)] if $\varphi \in \PSH(L)$,  then $\varphi + c \in \PSH(L)$ for all $c\in \bR$;
	\item[(ii)] if $L' \in \Pic(X)_{\bQ}$ is semiample and $\varphi'\in\PSH(L')$, then $\varphi+\varphi'\in\PSH(L+L')$;
	\item[(iii)] $\PSH(L)$ is preserved by finite maximum;
	\item[(iv)] if $(\varphi_i)_i$ is a decreasing net of psh metrics such that $\varphi := \lim_i \varphi_i \not\equiv -\infty$, then $\varphi\in\PSH(L)$;
	\item[(v)] $\PSH(L)$ is closed under the topology of uniform convergence.
	\item[(vi)] for all $t\in\bQ_{>0}$, $\PSH(tL)=t\PSH(L)$.
\end{itemize}
\end{proposition}

\begin{example}[\cite{BoucksomJonsson21}, Proposition 5.8]
Let $L$ be a semiample $\bQ$-line bundle on $X$ and $A$ be a finite set. We fix an integer $m\geq 1$, let $((s_\alpha,\lambda_{\alpha}))_{\alpha\in A}$ be a family such that, for all $\alpha\in A$, $s_{\alpha}\in H^{0}(X,mL)$ and $\lambda_{\alpha}\in\bR$. Then the (possibly singular) metric
\begin{align*}
\varphi := \frac{1}{m}\displaystyle\max_{\alpha\in A} \left(\ln|s_{\alpha}| + \lambda_{\alpha}\right)
\end{align*} 
belongs to $\PSH(L)$. Note that $\varphi$ is continuous iff the $s_{\alpha}$ do not share a common zero.
\end{example}

To mimic the complex case, it is convenient to have an approximation result for psh metrics by \emph{smooth} ones (cf. \cite{CLD12} when $k$ is non-Archimedean). 

\begin{definition}
\label{def:regularisable_psh_metric}
A singular metric $\varphi$ on  a $\bQ$-line bundle $L$ is called \emph{psh-regularisable} if it can be written as a decreasing limit of a net of smooth and semi-positive metrics on $L$.
\end{definition}

\begin{theorem}[\cite{BoucksomEriksson}, Theorem 7.14]
\label{th:BoucksomEriksson_th_7.14}
Let $\varphi \in \PSH(L)$, where $L$ is a semiample $\bQ$-line bundle. Then $\varphi$ is psh-regularisable. 
\end{theorem}

Theorem \ref{th:BoucksomEriksson_th_7.14} allows to study more closely the so-called \emph{pluripolar} sets, namely the sets included in the singularity locus of some psh metric. From now on, we assume that $k$ is non-Archimedean. 

\begin{proposition}
\label{prop:BJ18_5.2}
Let $L$ be an invertible $\cO_X$-module and $\varphi \in \PSH(L)$. Then, for all $x\in X^{\qm}$, we have $|\cdot|_{\varphi}(x) \not\equiv +\infty$.
\end{proposition}

\begin{proof}
Write $\varphi = \lim_{i\in I} \varphi _i$ where $(\varphi_i)_{i\in I}\in \FS(L)^{I}$ is a decreasing net. We fix a reference metric $\psi\in \DFS(L)$. Let $x\in X^{\qm}$, the global Izumi inequality (Theorem \ref{th:Izumi_global}) gives the existence of a constant $C>0$ such that
\begin{align*}
(\varphi_i-\varphi_j)(x) = |(\varphi_i-\varphi_j)(x)|=|(\varphi_i-\psi)(x)-(\varphi_j-\psi)(x)| \leq C,
\end{align*}
for all $i,j\in I$ with $i\leq j$. Hence $(\varphi_i-\varphi)(x) \leq C$ and $|\cdot|_{\varphi}(x) \not\equiv +\infty$.
\end{proof}

\begin{definition}
\label{def:Alexander_Taylor_capacity}
Let $L$ be an invertible $\cO_X$-module such that $\PSH(L)\neq \emptyset$ (in particular $L$ is semiample). Let $\varphi_{0}\in \DFS(L)$ be a reference metric. Let $E\subset X^{\an}$ be any subset, its \emph{Alexander-Taylor capacity} is defined by
\begin{align*}
T_{L,\varphi_0}(E) := \displaystyle\sup_{\varphi\in\PSH(L)} \left(\sup_{X^{\an}} (\varphi-\varphi_0) - \sup_{E}(\varphi-\varphi_0)\right).
\end{align*}
\end{definition}

\begin{remark}
\label{rem:Alexander_Taylor_capacity}
\begin{itemize}
	\item[(1)] Definition \ref{def:Alexander_Taylor_capacity} is a non-Archimedean analogue of the capacity defined in \cite{BT82},\cite{AT84}.
	\item[(2)] Using the same notation as in Definition \ref{def:Alexander_Taylor_capacity}, Proposition \ref{prop:PSH_metrics} $(i)$ implies the equality
\begin{align*}
T_{L,\varphi_0}(E)= -\displaystyle\sup_{\varphi\in\PSH_0(L)} \sup_{E}(\varphi-\varphi_0) \in [0,+\infty],
\end{align*}
where $\PSH_0(L) := \{\varphi\in\PSH(L) : \sup_{X^{\an}}(\varphi-\varphi_0) = 0\}$. 
	\item[(3)] A priori, the construction of the Alexander-Taylor capacity depends on the choice of a reference metric. In the following, we will only be interested in its finiteness. As the difference of two continuous metrics is a bounded function, this finiteness is in fact independent on the choice of the reference metric.
\end{itemize}
\end{remark}

\begin{lemma}
\label{lemma:BoucksomEriksson_4.45}
Let $L$ be an invertible $\cO_X$-module such that $\PSH(L)\neq \emptyset$ (in particular $L$ is semiample). Let $\varphi_{0}\in \DFS(L)$ be a reference metric. For all $x\in X^{\qm}$, we have $T_{L,\varphi_0}(x) := T_{L,\varphi_0}(\{x\})<+\infty$.
\end{lemma}

\begin{proof}
Let $x\in X^{\qm}$, assume that $T_{L,\varphi_0}(x) = +\infty$. Remark \ref{rem:Alexander_Taylor_capacity}, for any integer $m\geq 1$, there exists a metric $\varphi_m \in \PSH(L)$ such that $\sup_{X^{\an}}(\varphi_m-\varphi_0) = 0$ and $(\varphi_m-\varphi_0)(x) \leq -2^{m}$. For all $m\in \bN_{>0}$, we define 
\begin{align*}
\psi_m := \displaystyle\sum_{i=1}^{m}2^{-i}\varphi_i + 2^{-m}\cdot 0.
\end{align*} 
By convexity of $\PSH(L)$, $\psi_m\in \PSH(L)$. $(\psi_m)_{m\geq 1}$ is thus a decreasing sequence in $\PSH(L)$ satisfying $\sup_{X^{\an}}(\psi_m-\varphi_0) = 0$ and $(\psi_m-\varphi_0)(x) \leq -m$ for $m\in\bN_{>0}$. Proposition \ref{prop:PSH_metrics} (iv) then implies that $\psi := \lim_{m\to+\infty}\psi_m \in \PSH(L)$ and satisfies $(\psi-\varphi_0)(x)= -\infty$, which contradicts Proposition \ref{prop:BJ18_5.2}. 
\end{proof}

\subsubsection{Envelopes}
\label{subsub:envelopes}

Throughout this paragraph, we fix a projective scheme $X$ over $\Spec(k)$.

\begin{definition}
\label{def:psh_envelope_bounded_metric}
Let $\varphi$ be a bounded metric on a semiample line bundle $L$. Its \emph{psh envelope}, denoted by $P(\varphi)$, is defined by 
\begin{align*}
P(\varphi) := \sup \{\phi : \phi \in  \PSH(L) \text{ and } \phi\leq\varphi\}.
\end{align*} 
Its \emph{regular psh envelope}, denoted by $Q(\varphi)$, is defined by
\begin{align*}
Q(\varphi) := \sup \{\phi : \phi \in  \PSH(L)\cap C^0(L) \text{ and } \phi\leq\varphi\}.
\end{align*} 
\end{definition}

\begin{proposition}[\cite{Boucksom21}, Proposition 6.3]
\label{prop:Boucksom21_6.3}
Let $\varphi,\varphi_1,\varphi_2$ be bounded metrics on the line bundle $L$. Then 
\begin{itemize}
	\item[(i)] $P(\varphi) \leq \varphi$ with equality iff $\varphi$ is continuous and psh;
	\item[(ii)] for all $m\in\bN$, $P(m\varphi) = mP(\varphi)$;
	\item[(iii)] if $\varphi_1\leq\varphi_2$, then $P(\varphi_1) \leq P(\varphi_2)$;
	\item[(iv)] if $c\in\bR$, then $P(\varphi+c) = P(\varphi) + c$;
	\item[(v)] $d(P(\varphi_1),P(\varphi_2))\leq d(\varphi_1,\varphi_2)$.
\end{itemize}
\end{proposition}

Using the same notation as in Definition \ref{def:psh_envelope_bounded_metric}, we have $Q(\varphi)\leq P(\varphi) \leq \varphi$ and $Q(\varphi) = P(\varphi_{\star})$ (\cite{BoucksomEriksson}, Proposition 7.25).

\begin{definition}
\label{def:continuity_of_envelope}
Let $L$  ba a semiample $\bQ$-line bundle. We say that the \emph{continuity of envelopes} holds for $L$ if for all $\varphi\in C^0(L)$, $P(\varphi)=Q(\varphi)$ is continuous.
\end{definition}

\begin{lemma}\cite{BoucksomEriksson}, Lemma 7.29)
\label{lemme:BoucksomEriksson_lemma_7.29}
For a semiample $\bQ$-line bundle $L$, continuity of envelopes is equivalent to the following: for any uniformly bounded from above family $(\varphi_{\alpha})_{\alpha}$ of psh metrics on $L$, $(\sup_{\alpha} \varphi)^{\star}\in\PSH(L)$.  
\end{lemma}

\begin{conjecture}
\label{conj:BoucksomEriksson_7.30}
Continuity of envelopes holds for any semiample $\bQ$-line bundle on a normal projective variety over $\Spec(k)$.
\end{conjecture}

Conjecture \ref{conj:BoucksomEriksson_7.30} is known to be true when $k$ is Archimedean and in the case where: $X$ is smooth, $L$ is ample and $k$ is non-Archimedean either of equicharacteristic zero, trivially valued, or discretely valued (cf. \cite{BoucksomEriksson}).

In this article, we will make use of envelopes of possibly unbounded metrics. We essentially reproduce the arguments of \cite{BoucksomJonsson21} in the case of an arbitrary complete valued field. 

\begin{definition}
\label{def_psh_envelope_singular_metric}
Let $L$ be an invertible $\cO_X$-module, $\phi\in C^0(L)$ and $f : X^{\an}\to \bR\cup\{-\infty,+\infty\}$. The \emph{psh envelope} of $\varphi := \phi + f$ is defined by
\begin{align*}
P(\varphi):= \sup\{\psi : \psi\in \PSH(L) \text{ and } \psi\leq\varphi \}.
\end{align*}
\end{definition} 

The psh envelope of a bounded metric is in fact a psh metric when the continuity of envelopes holds. We show an analogue result for singular metrics. 

\begin{proposition}
\label{prop:BoucksomJonsson_5.17_général}
We assume that $k$ is non-Archimedean. Let $L$ be an ample invertible $\cO_X$-module, $\phi_0\in C^0(L)$ and $f : X^{\an}\to \bR\cup\{-\infty,+\infty\}$.The following statements are equivalent:
\begin{itemize}
	\item[(i)] let $(\varphi_i)_{i\in I}\in \PSH(L)^{I}$ be uniformly bounded from above, then $(\sup_{i\in I} \varphi_i)^{\star} \in PSH(L)$;  
	\item[(ii)] let $\varphi := \phi_0+f$, then either $P(\varphi) \equiv -\infty$, $P(\varphi)^{\star}\in \PSH(L)$ or $P(\varphi)^{\star} \equiv +\infty$;
	\item[(iii)] for all $\phi\in C^0(L)$, $P(\phi)\in C^0(L)$.
\end{itemize}
\end{proposition}

\begin{proof}
We reproduce the arguments of (\cite{BoucksomJonsson21}, Lemma 5.17). We first assume that $(i)$ is true. Denote $\cF := \{\psi\in \PSH(L) : \psi\leq \varphi\}$. If $\cF = \emptyset$, then $P(\varphi) \equiv -\infty$. If now $\cF \neq \emptyset$, if $\cF$ is uniformly bounded from above, $(i)$ implies $P(\varphi)^{\star}\in \PSH(L)$. In the opposite case, for all $x\in X^{\qm}$, Lemma \ref{lemma:BoucksomEriksson_4.45} yields $T_{L,\varphi_0}(x) <+\infty$  (where $\varphi_0\in\DFS(L)$ is an arbitrary reference metric). Henceforth, we have
\begin{align*}
\displaystyle\sup_{X^{\an}}(\sup\{\psi : \psi \in \cF\}-\varphi_0) = + \infty,
\end{align*}
and the inequality
\begin{align*}
\forall \psi\in \cF,\quad (P(\varphi)-\varphi_0)(x) \geq \displaystyle\sup_{X^{\an}}(\psi-\varphi_0) - T_{L,\varphi_0}(x)
\end{align*}
implies $(P(\varphi)-\varphi_0)(x) = +\infty$. By density of $X^{\qm}$ in $X^{\an}$, we obtain $P(\varphi)^{\star}\equiv +\infty$.

Assume now that $(ii)$ is true. Let $\phi\in C^0(L)$. Then $P(\phi)^{\star}\in \PSH(L)$ is a candidate in the definition of $P(\phi)$, thus $P(\phi)^{\star} \leq P(\phi)$, hence $P(\phi)^{\star} = P(\phi)$. In particular $P(\phi)$ is usc. As $\phi$ is continuous (hence lsc), we obtain $Q(\phi) = P(\phi_{\star}) = P(\phi)$. Furthermore, as a supremum on a family of continuous functions, $Q(\phi)$ is lsc. Whence $P(\phi)$ is continuous. 

Finally, Lemma \ref{lemme:BoucksomEriksson_lemma_7.29} gives $(iii) \Rightarrow (i)$.
\end{proof}

\subsubsection{Volume and Monge-Ampère energy}
\label{subsub:volume_MA}

Throughout this paragraph, we fix a geometrically integral projective scheme $X$ over $\Spec(k)$.

\begin{definition}
\label{def:local_volume}
Let $\varphi,\psi$ be two bounded metrics on a line bundle $L$. The \emph{relative} $\chi$\emph{-volume} of $\varphi$ with respect to $\psi$ is defined by 
\begin{align*}
\vol_{\chi}(L,\varphi,\psi) := \displaystyle\limsup_{n\to+\infty} \frac{(d+1)!}{n^{d+1}}\cdot \ln \frac{\|\cdot\|_{n\psi,\det}}{\|\cdot\|_{n\varphi,\det}}.
\end{align*}
Theorem 4.5 of \cite{ChenMaclean15} yields that the above superior limit is in fact a limit.
\end{definition}

\begin{remark}
In the literature, other normalisations for the relative $\chi$-volume can be found. For instance, in \cite{Boucksom21}, using the same notation as in  Definition \ref{def:local_volume}, a relative volume is defined by 
\begin{align*}
\vol(L,\varphi,\psi) := \displaystyle\limsup_{n\to+\infty} \frac{d!}{n^{d+1}}\cdot \ln \frac{\|\cdot\|_{n\psi,\det}}{\|\cdot\|_{n\varphi,\det}} = \frac{1}{d+1}\vol_{\chi}(L,\varphi,\psi).
\end{align*}
\end{remark}

\begin{proposition}[\cite{Boucksom21}, Proposition 6.4]
\label{prop:Boucksom_6.4}
Let $\varphi$ be a bounded metric on a line bundle $L$.
\begin{itemize}
	\item[(1)] We have the equality $\|\cdot\|_{\varphi} = \|\cdot\|_{P(\varphi)} = \|\cdot\|_{Q(\varphi)}$ of seminorms.
	\item[(2)] For any bounded metric $\psi$ on $L$, we have $\vol_{\chi}(L,\varphi,\psi) = \vol_{\chi}(L,P(\varphi),P(\psi))$.
\end{itemize}
\end{proposition}

Let $(L_1,\varphi_1),...,(L_d,\varphi_d)$ be metrized invertible $\cO_X$-modules such that $\varphi_i\in \PSH(L_i)\cap C^{0}(L_i)$ for all  $i=1,...,d$. Bedford-Taylor theory (in the complex case) and its non-Archimedean counterpart (\cite{CLD12,GK17}) allow to define a probability Radon measure on $X^{\an}$, called the \emph{Monge-Ampère measure} of $(L_1,\varphi_1),...,(L_d,\varphi_d)$, defined by
\begin{align*}
\MA(\varphi_1,...,\varphi_d) := \frac{1}{(L_1\cdots L_d)}\cdot \diff\dc\varphi_1\wedge\cdots\wedge \diff\dc\varphi_d.
\end{align*}
When $(L_i,\varphi_i)=(L,\varphi)$ for all $i=1,...,d$, we denote $\MA(\varphi):=\MA(\varphi_1,...,\varphi_d)$. Historically, the non-Archimedean Monge-Ampère measures were introduced by Chambert-Loir in \cite{ChambertLoir06}.

\begin{remark}
\label{rem:DPSH}
The above construction extends to the case where the metrics are assumed to be \emph{dpsh}, namely a difference of psh metrics (\cite{BoucksomEriksson}, \S 8.1). 
\end{remark}

\begin{definition}
\label{def:Monge-Ampère_energy}
Let $\varphi,\psi$ be two continuous psh metrics on a line bundle $L$. The \emph{Monge-Ampère energy} of $\varphi$ with respect to $\psi$ is defined by
\begin{align*}
\cE(\varphi,\psi) := \frac{1}{d+1} \int_{X^{\an}} (\varphi-\psi)\MA_d(\varphi,\psi),
\end{align*}
where
\begin{align*}
\MA_d(\varphi,\psi) := \displaystyle\sum_{j=0}^{d} (\diff\dc\varphi)^{\wedge j}\wedge(\diff\dc\psi)^{\wedge (d-j)}.
\end{align*}
\end{definition}

If $\varphi,\psi$ are two continuous metrics on a line bundle $L$ and if the continuity of envelopes holds for $L$, we denote 
\begin{align*}
\cE\circ P(\varphi,\psi) := \cE(P(\varphi),P(\psi)).
\end{align*}

\begin{theorem}[\cite{Boucksom21}, Corollary 6.5]
\label{th:Boucksom21_6.5}
Let $L$ be a semiample line bundle, let $\varphi,\psi$ be two continuous metrics on $L$ and assume that $P(\varphi),P(\psi)$ are continuous. Then
\begin{align*}
\vol_{\chi}(L,\varphi,\psi) = (d+1)\cE \circ P(\varphi,\psi).
\end{align*}
\end{theorem}

In this article, we address a differentiability property for the relative $\chi$-volume. The latter is related to a so-called orthogonality property. If $L$ is a semiample invertible $\cO_X$-module, we say that 
\begin{itemize}
	\item[(a)] the \emph{orthogonality property} holds for $\varphi$ if the measure $(\diff\dc P(\varphi))^{\wedge d}$ is supported on $\{\varphi = P(\varphi)\}$;
	\item[(b)] $\cE\circ P$ is \emph{differentiable} in $\varphi$ if, for any $f\in C^0(X^{\an})$, we have
	\begin{align*}
	\left.\frac{\diff}{\diff t}\right|_{t=0} \cE\circ P (\varphi +tf,\varphi_0) = \int_{X^{\an}}f (\diff\dc P(\varphi))^{\wedge d},
	\end{align*}
	where $\varphi_0\in \DFS(L)$ is an arbitrary reference metric.
\end{itemize}
 
\begin{proposition}
\label{prop:orthognality_property}
The following statements are equivalent:
\begin{itemize}
	\item[(i)] for all metric $\varphi \in C^0(L)$, $\cE\circ P$ is differentiable in $\varphi$;
	\item[(ii)] for all metric $\varphi \in C^0(L)\cap\PSH(L)$, $\cE\circ P$ is differentiable in $\varphi$;
	\item[(iii)] the orthogonality property holds for any $\varphi\in C^0(L)$.
\end{itemize}
\end{proposition}

\begin{proof}
When $L$ is ample, it is (\cite{BGM20}, Lemma 3.5). The arguments generalise to the semiample case (\cite{Boucksom21}, Remark 6.8). 
\end{proof}

\begin{theorem}[\cite{Boucksom21}, Theorem 6.7, \cite{BoucksomJonsson21}, Corollary 8.6]
\label{th:Boucksom21_6.7}
Let $\varphi\in C^0(L)$ such that $P(\varphi)\in C^0(L)$, then the orthogonality property holds for $\varphi$ and $E\circ P$ is differentiable in $\varphi$.
\end{theorem}

\subsection{Okounkov bodies and slopes}
\label{sub:Okounkov_bodies_and_slopes}

\subsubsection{Line bundle case}
\label{subsub:Okounkov_line_bundles}

Throughout this paragraph, we fix a complete valued field $(k,\va)$ and we denote by $\cC_k$ the class of finite dimensional vector spaces over $k$ endowed with two norms which are assumed to be ultrametric when $k$ is non-Archimedean. Let $\overline{V}=(V,\|\cdot\|_{\varphi},\|\cdot\|_{\psi})$ be an element in $\cC_k$, its \emph{degree} is defined by
\begin{align*}
\widehat{\deg}(\overline{V}) := -\ln\|s_1\wedge\cdots\wedge s_r\|_{\varphi} + \ln\|s_1\wedge\cdots\wedge s_r\|_{\psi},
\end{align*}
where $(s_1,...,s_k)$ is an arbitrary basis of $V$. This degree is well-defined thanks to the product formula
\begin{align*}
-\ln|a| + \ln|a| = 0, \forall a\in k^{\times}.
\end{align*}
The \emph{slope} of $\overline{V}$ is defined by
\begin{align*}
\widehat{\mu}(\overline{V}) := \frac{\widehat{\deg}(\overline{V})}{\dim_k(V)}.
\end{align*}

The abstract framework in which we consider Okounkov bodies is the following. We fix a graded semigroup $\Gamma \subset \bN^{d+1} = \bN \times \bN^d \subseteq \bN \times \bR^{d}$. For all $n\in\bN$, the homogeneous degree $n$ component of $\Gamma$ is $\Gamma_n := \{\alpha\in\bN^d : (n,\alpha)\in \Gamma\}$. $\Gamma$ is assumed to satisfy the following conditions (cf. \cite{LazarsfeldMustata09} \S 2.1 for further details) : 
\begin{itemize}
	\item[(a)] $\Gamma_0 = \{0\}$;
	\item[(b)] there exists a finite set $B\subset \{1\}\times \bN^d$ such that $\Gamma$ is included in the sub-monoid of $\bN^d$ generated by $B$;
	\item[(c)] $\Gamma$ generates $\bZ^{d+1}$ as a group.
\end{itemize}   
We denote by $\Sigma(\Gamma)\subset \bR^{d+1}\}$ the closed convex cone generated by $\Gamma$ and $\Delta(\Gamma):= \Sigma(\Gamma) \cap \left(\{1\}\times \bR^{d}\right)$ its base: it is the \emph{Okounkov body} of $\Gamma$. $\Delta(\Gamma)$ is a convex body in $\bR^d$ and we have the equality
\begin{align*}
\displaystyle\lim_{n\to+\infty} \frac{|\Gamma_n|}{n^d} = \vol(\Delta(\Gamma)).
\end{align*}
Moreover, we consider a superadditive function $\phi : \Gamma \to \bR$, namely $\phi(\gamma+\gamma')\geq \phi(\gamma) + \phi(\gamma')$ for all $\gamma,\gamma'\in \Gamma$.

\begin{lemma}[\cite{ChenMaclean15}, Lemma 4.1]
\label{lemma:ChenMaclean_4.1}
Assume that $\phi(0)=0$. 
\begin{itemize}
	\item[(1)] For any real number $t$, the set $\Gamma_{\phi}^{t} := \{(n,\alpha)\in\Gamma : \phi(n,\alpha)\geq nt\}$ is a sub-semigroup of $\Gamma$.
	\item[(2)] Let $t$ be a real number satisfying 
	\begin{align*}
	t < \displaystyle \lim_{n\to+\infty}\sup_{\alpha\in\Gamma_n}\frac{\phi(n,\alpha)}{n}.
	\end{align*}
	Then $\Gamma_{\phi}^{t}$ satisfies condtions $(a)-(c)$ above.
\end{itemize}
\end{lemma}

\begin{theorem}[\cite{ChenMaclean15}, Theorem 4.3]
We use the same notation as above. Assume that
\begin{align*}
\theta := \displaystyle \lim_{n\to+\infty}\sup_{\alpha\in\Gamma_n}\frac{\phi(n,\alpha)}{n} < +\infty.
\end{align*}
For all $n\geq 1$, denote by $Z_n := \phi(n,\cdot)$ the uniformly distributed random variable on $\Gamma_n$. The the sequence of random variables $(Z_n/n)_{n\geq 1}$ converges in law to a random variable $Z$ whose probability law is defined by 
\begin{align*}
\bP(Z\geq t) = \frac{\vol(\Delta(\Gamma_{\Phi}^{t}))}{\vol(\Delta(\Gamma))}, \quad \forall t\neq \theta.
\end{align*}
Furthermore, there exists a concave function $G_{\phi} : \Delta(\Gamma) \to \bR \cup \{-\infty\}$ taking finite values in $\Delta(\Gamma)^{\circ}$, such that
\begin{align*}
Z = (G_{\phi})_{\ast}\lambda,
\end{align*}
where $\lambda$ denotes the normalised Lebesgue measure on $\Delta(\Gamma)$. Concretely, for any bounded continuous function $h$, we have
\begin{align*}
\displaystyle\lim_{n\to+\infty} \frac{h(\phi(n,\alpha)/n)}{|\Gamma_n|} = \frac{1}{\vol(\Delta(\Gamma))} \int_{\Delta(\Gamma)^{\circ}} G_{\phi}(x)\diff x.
\end{align*}
\end{theorem}

We now introduce some notation to apply the latter results in a geometric context. Let $X$ be an integral projective scheme over $\Spec(k)$ of dimension $d\geq 1$ and $L$ be a line bundle on $X$. We assume that there exists a regular rational point $p\in X(k)$. Then the Cohen structure theorem yields an isomorphism of rings $\widehat{\cO_{X,p}}\cong k[\![z_1,...,z_d]\!]$. If $V_{\bullet}\subset V_{\bullet}(L)$ is a graded linear series of $L$, after trivialising $L$ in a neighbourhood of $p$, $V_{\bullet}$ can be viewed as a graded subalgebra of $\cO_{X,p}[T]$. We also fix a monomial order on $\bN^d$. The latter yields a $\bN^d$-filtration of $V_n$ and we denote by $(\gr^{(n,\alpha)}(V_{\bullet}))_{\alpha\in\bN^d}$ its graded homogeneous components. 
The $\bN^{d+1}$-graded algebra $\gr(V_{\bullet})$ can then be viewed as a graded subalgebra of $k[z_1,...,z_d,T]$. We define the graded semigroup
\begin{align*}
\Gamma(V_{\bullet}) := \{(n,\alpha)\in\bN^{d+1} : \gr^{(n,\alpha)}(V_{\bullet}) \neq (0)\}.
\end{align*}

We now add analytic data. Using the same notation as above, for all $n\in\bN$, let $\varphi_n,\psi_n$ be norms on $V_n$, which are assumed to be ultrametric when $k$ is non-Archimedean. We assume that the families $(\varphi_n)_{n\geq 0},(\psi_n)_{n\geq 0}$ are submultiplicative. For all $(n,\alpha)\in\Gamma(V_{\bullet})$, $\varphi_n$ and $\psi_n$ induce norms on $\gr^{(n,\alpha)}(V_{\bullet})$ denoted by $\varphi_n^{\alpha}$ and $\psi_n^{\alpha}$. Taking the orthogonal direct sum (with respect to the index $\alpha$) of these norms, we obtain norms $\widehat{\varphi_n},\widehat{\psi_n}$ on $\gr(V_n) := \bigoplus_{\alpha\in\Gamma_n} \gr^{(n,\alpha)}(V_{\bullet})$. For any $n\in\bN$, we consider the auxiliary norm $\eta_n$ on $\gr(V_n)$ defined as follows. For all $\alpha\in\Gamma(V_n)$, we denote by $s_{(n,\alpha)}$ the canonical image of $(n,\alpha)$ in $\gr(V_n)$. We then have
\begin{align*}
\gr(V_n) = \displaystyle\bigoplus_{\alpha\in\Gamma(V_n)} k s_{(n,\alpha)}.
\end{align*}
$\eta_n$ is defined as the unique norm on $\gr(V_n)$ such that $(s_{(n,\alpha)})_{\alpha\in\Gamma(V_n)}$ is an orthonormal basis of $\gr(V_n)$. The function
\begin{align*}
\fonction{\phi}{\Gamma(V_{\bullet})}{\bR}{(n,\alpha)}{\widehat{\deg}(\gr^{(n,\alpha)}(V_{\bullet}),\varphi_n^{\alpha},\eta_n^{\alpha}),}
\end{align*}
is then superadditive. If additionally
\begin{align*}
\exists C>0,\quad \forall n\in\bN_{\geq 1},\quad \displaystyle\inf_{\alpha\in\Gamma(V_n)}\ln\|s_{(n,\alpha)}\|_{\widehat{\varphi}_n} \geq -Cn,
\end{align*}
the previous asymptotic result yields the existence of a concave function $G_{\phi}$, called the \emph{concave transform} and which depends on $(\eta_n)_{n\geq 0}$, defined by
\begin{align*}
\forall x\in\Delta(\Gamma(V_{\bullet})), \quad G_{\phi}(x) = \sup\{t\in\bR : x\in\Delta(\Gamma(V_{\bullet})_{\phi}^{t})\} \in \bR\cup\{-\infty\}. 
\end{align*}

\begin{remark}
\label{remark:concave_transform}
Using the same notation as above, the definition of the concave transform $G_{\phi}$ only depends on the norm family $\varphi := (\varphi_n)_{n\geq 0}$. It is denoted by $G_{\varphi}$. However, the construction of $G_{\varphi}$ depends on the choice of the regular rational point $p\in X(k)$ and of the monomial order on $\bN^d$.
\end{remark}

\begin{theorem}[\cite{ChenMaclean15}, Corollary 4.6]
\label{th:ChenMaclean_4.6}
Let $X$ be an integral projective scheme over $\Spec(k)$ of dimension $d$, $L$ be a line bundle on $X$ endowed with two continuous metrics $\varphi,\psi$. Let $V_{\bullet}$ be a graded linear series of $L$ such that the Okounkov semigroup $\Gamma(V_{\bullet})$ satisfies 
\begin{itemize}
	\item[(a)] $\Gamma(V_{0})=\{0\}$;
	\item[(b)] there exists a finite set $B\subset \{1\}\times\bN^d$ such that $\Gamma(V_{\bullet})$ is included in the sub-monoid of $\bN^{d+1}$ generated by $B$.   
	\item[(c)] $\Gamma(V_{\bullet})$ generates $\bZ^{d+1}$ as a group.
\end{itemize} 
Then the sequence $\left(\widehat{\mu}(V_n,\|\cdot\|_{n\varphi},\|\cdot\|_{n\psi})/n\right)_{n\geq 1}$ converges in $\bR$. Furthermore, if $L$ is big, we denote
\begin{align*}
\vol_{\chi}(L,\varphi,\psi) := \displaystyle(d+1)\vol(L)\lim_{n\to+\infty}  \frac{\widehat{\mu}(H^{0}(X,nL),n\varphi,n\psi))}{n} = -\lim_{n\to+\infty}\frac{(d+1)!}{n^{d+1}}\ln\frac{\|\cdot\|_{n\varphi,\det}}{\|\cdot\|_{n\psi,\det}}.
\end{align*}
\end{theorem}

\begin{remark}
If $L$ is big, Lemma 2.2 of\cite{LazarsfeldMustata09} ensures that the hypothesis of Theorem \ref{th:ChenMaclean_4.6} are satisfied .
\end{remark}

\subsubsection{Concave transform and relative $\chi$-volume of $\bR$-divisors}
\label{subsub:transformée_concave_R_diviseur}

Throughout this parargraph, we recall the construction of Okounkov bodies for $\bR$-divisors and the construction of concave transforms for $\bR$-divisors endowed with a Green functions. We use \cite{LazarsfeldMustata09,Ballay22}.

%Recall (\cite{LazarsfeldMustata09}, Definition 4.3) that for any big $D\in \Div_{\bQ}(X)$ there is a convex body $\Delta(D)$ only depending on the numerical equivalence class $[D] \in N^1_{\bQ}(X)$. 

Let $\bK$ be either $\bQ$ or $\bR$. Let $X$ be an integral projective normal scheme over $\Spec(k)$ of dimension $d$ and assume there exists a regular rational point $p\in X(k)$. We fix a monomial order on $\bN^d$. Let $D\in \Div(X)_{\bK}$ be a big $\bK$-divisor. For all $f\in H^0_{\bK}(D)$, let $f_1^{a_1}\cdots f_l^{a_l}$ be a local equation of $\div(f)+D$ in $p$, where $f_1,...,f_l\in \cO_{X,p}\setminus\{0\}$ and $a_1,...,a_l\in \bR_{>0}$. Let $v(f) := a_1v(f_1) + \cdots + a_l v(f_l) \in \bR_{\geq 0}^d$, it is independent of the choice of the local equation of $\div(f)+D$ in $p$. We now define 
\begin{align*}
\Delta(D) := \overline{\left\{\frac{v(f)}{n} : n\geq 1,\quad f\in H^0_{\bR}(D)\right\}} \subset \bR^d.
\end{align*}
It is a convex body $\bR^d$ called the \emph{Okounkov body} of $D$.

Let $g$ be a Green function of $D$. The algebra $\bigoplus_{n\in \bN} H^0_{\bK}(X,nD)$ is a graded sub-algebra of subfinite type of $k[\![z_1,...,z_d]\!]$. Recall (cf. \S \ref{subsub:Green_functions}) that $g$ defines a norm $\xi_{ng}$ on $H_{\bK}^0(X,nD)$ for all integer $n\geq 0$. Note that when $\bK=\bQ$, this norm coincides with the supremum norm defined by the corresponding continuous metric $\varphi_g$ on the $\bQ$-line bundle $\cO_X(D)$. If $h$ is another Green function for $D$, an adaptation of the results of Chen and Maclean yields the convergence of the sequence
\begin{align*}
\frac{\widehat{\deg}(H_{\bK}^0(X,nD),\xi_{ng},\xi_{nh})}{n^{d+1}/(d+1)!},\quad n\geq 1, \quad H_{\bK}^0(X,nD)\neq (0).
\end{align*}
Its limit is denoted by $\vol_{\chi}(D,g,h)$. Furthermore, we have concave transforms $G_{(D,g)},G_{(D,h)} : \Delta(D) \to \bR\cup\{-\infty\}$ taking real values in $\Delta(D)^{\circ}$. We then have the equality
\begin{align*}
\vol_{\chi}(D,g,h) = (d+1)!\int_{\Delta(D)} G_{(D,g)}-G_{(D,h)}\diff\lambda,
\end{align*}
where $\lambda$ the Lebesgue measure on $\Delta(D)$. 

\begin{remark}
\label{rem:notation_fibrés_en_droites_diviseurs}
Let $L$ be an invertible $\cO_X$ module and $s$ be a non-zero rational section of $L$. By abuse of notation, we denote by $L$ the class $\div(s)$ in $N^{1}_{\bR}(X)$. This class is independent on the choice of $s$. Moreover, if $\varphi \in C^0(L)$, we have the equality $G_{(\div(s),g_{\varphi})} = G_{(L,\varphi)}$ in $\Delta(L)$. We will freely make use of expressions of the form $\Delta(L+tA)$ and $G_{(L+tA,\varphi + t\Phi)}$, where $A$ is an invertible $\cO_X$-module, $\Phi\in C^0(A)$ and $t\in \bR$.
\end{remark}

\begin{remark}
\label{rem:concave_transform_finite_ample_case}
Assume that $D$ is an ample $\bR$-divisor on $X$ equipped with a Green function $g$. Then the concave transform $G_{(D,g)}$ is finite on $\Delta(D)$ (cf. \cite{BoucksomChen}, \S 1.3).
\end{remark}

\section{Local differentiability}
\label{sec:local_differentiability}

In this section, we fix an infinite complete valued field $(k,\va)$ and a geometrically integral normal projective scheme $X$ de dimension $d$ over $\Spec(k)$. We assume that the continuity of envelopes conjecture (Conjecture \ref{conj:BoucksomEriksson_7.30}) holds for $k$.

\subsection{Singular metrics and slices of the Okounkov body}

We study a particular case. Let be $A$ an ample line bundle on $X$, and assume there exists a non-zero global section $s\in H^{0}(X,A)$. Example \ref{example:singular_metric_global_section} yields a singular metric $\psi_{s}$ on $A$. Recall that if $\phi_0$ is a reference continuous metric on $A$, then $\psi_{s} = \phi_{0} + \ln|s|_{\phi_{0}}$. Then $-\psi_{s}$ is a singular metric on $-A$ which can be written $-\psi_{s} = -\phi_0 - \ln|s|_{\phi_0}$ with $-\phi_0\in C^0(-A)$ and $-\ln|s|_{\phi_0} : X^{\an} \to \bR\cup\{+\infty\}$ which is continuous and non identically $+\infty$. For all $R>0$, we denote
\begin{align*}
-\psi_{s,R} := \min\{-\phi+R,-\psi_{s}\}.
\end{align*}
Let $L$ be a semiample and big line bundle on $X$ and let $\varphi\in C^0(L)$. Let $n\in\bN$ such that $nL-A$ est big. Then Lemma \ref{lemma:sup_norm} yields, for all $R>0$, $n\varphi-\psi_{s,R} \in C^0(nL-A)$ and, for all $R>0$ sufficiently big, the equality 
\begin{align}
\label{eq:equality_sup_norms_APCR}
\|\cdot\|_{n\varphi-\psi_{s,R}} = \|\cdot\|_{n\varphi_{0}-\psi_{s}}.
\end{align}

As $X$ is geometrically integral, one may assume that there exists a regular rational point $x\in X(k)$. We fix a monomial order on $\bN^d$ (e.g. the lexicographic order). Then the results of \S \ref{sub:Okounkov_bodies_and_slopes} yield a convex body $\Delta(nL-A)\subset \bR^d$ and, for all $R>0$, a concave transform $G_{n\varphi-\psi_{s,R}} : \Delta(nL-A) \to \bR\cup\{-\infty\}$ (taking real values in $\Delta(nL-A)^{\circ}$). Furthermore, for all $R>0$, the construction of $G_{n\varphi-\psi_{s,R}}$ only depends on $\|\cdot\|_{n\varphi-\psi_{s,R}},p$ and the monomial order. 

\begin{lemma}
\label{lemma:Nystrom_14.2}
There exists $R_0>0$ such that, for all $R\geq R_{0}$, the concave transform $G_{n\varphi-\psi_{s,R}}$ is independent on $R$ and is denoted by $G_{n\varphi-\psi_{s}}$.
\end{lemma}

\begin{proof}
It is a direct consequence of the above discussion and equality (\ref{eq:equality_sup_norms_APCR}).
\end{proof} 

\begin{proposition}
\label{prop:Nystrom_14.3}
Assume that $nL-A$ is semiample. Let $\psi\in C^0(nL-A)$. Then there exists $R_0>0$ such that, for all $R\geq R_0$, we have
\begin{align*}
\cE\circ P(n\varphi-\psi_{s},\psi):=\cE \circ P(n\varphi-\psi_{s,R},\psi) = d!\int_{\Delta(nL-A)^{\circ}}G_{n\varphi-\psi_{s,R}}(x)-G_{\psi}(x)\lambda(\diff x),
\end{align*}
where $\lambda$ denotes the Lebesgue measure on $\bR^d$.
\end{proposition}

\begin{proof}
For all $R>0$, as $P(n\varphi-\psi_{s,R})$ is continuous (recall that we assume the continuity of envelopes), we have
\begin{align*}
    \cE \circ P(n\varphi-\psi_{s,R},\psi) = \frac{1}{d+1}\vol_{\chi}(nL-A,n\varphi-\psi_{s,R},\psi) = d!\int_{\Delta(nL-A)^{\circ}}G_{n\varphi-\psi_{s,R}}(x)-G_{\psi}(x)\lambda(\diff x),
\end{align*}
where the first equality comes from Theorem \ref{th:Boucksom21_6.5} and the second comes from Theorem \ref{th:ChenMaclean_4.6}. Lemma \ref{lemma:Nystrom_14.2} then yields the conclusion.
\end{proof}

Recall that 
\begin{itemize}
	\item[(1)] $L$ is a semiample and big line bundle on $X$ and $\varphi,\psi \in C^0(L)$;
	\item[(2)] $A$ is an ample line bundle on $X$, $s\in H^0(X,A)$, $\phi\in C^0(A)$ and $\psi_{s} = \phi + \ln|s|_{\phi}$ is a singular metric on $A$;
\end{itemize}
Let $u\in C^0(X^{\an})$, we denote
\begin{align*}
\fonction{f}{\bR}{\bR_{\geq 0}}{t}{\vol_{\chi}(L,\varphi + tu,\psi).}
\end{align*}
$f$ is a concave and continuously differentiable (cf. Theorem \ref{th:Boucksom21_6.7}) and we have
\begin{align}
\label{eq:differentiability_local_volume_continuous_direction}
\forall t \geq 0, \quad f'(t) = (d+1)\int_{X^{\an}} u (\diff\dc P(\varphi+tu))^{\wedge d}.
\end{align}

Following \cite{Nystrom13}, we will generalise (\ref{eq:differentiability_local_volume_continuous_direction}) in the case where $u = \phi-\psi_{s}=-\ln|s|_{\phi}$. For $R>0$, recall that $u_R := \phi-\psi_{s,R} = \min\{u, R\}$ is a continuous function on $X^{\an}$. 
\begin{lemma}
\label{lemma:Nystrom_14.1}
Let $t\geq 0$. Then there exists $R_0>0$ such that, for all $R\geq R_0$, we have $P(\varphi+tu_R) = P(\varphi+tu)^{\star}$. 
\end{lemma}

\begin{proof}
For all $R>0$, $\varphi+tu_R \leq \varphi+tu$ implies $P(\varphi+tu_R) \leq P(\varphi+tu)\leq P(\varphi+tu)^{\star}$. Let us show that there exists $R_0>0$ such that, for all $R\geq R_0$, $P(\varphi+tu_R) \geq P(\varphi+tu)^{\star}$. Since $\varphi+tu$ is continuous on the open set $X^{\an}\setminus \{s=0\}^{\an}$, it is locally bounded on this open set. Thus $P(\varphi+tu)^{\ast} \not\equiv +\infty$. Furthermore, $P(\varphi+tu)^{\star} \not\equiv -\infty$ as $P(\varphi+tu_1)\in C^0(L)$ is a candidate in the definition of $P(\varphi+tu)$. Proposition \ref{prop:BoucksomJonsson_5.17_général} yields $P(\varphi+tu)^{\star} \in \PSH(L)$. In particular, $P(\varphi+tu)^{\star} \in \PSH(L)$ is usc hence bounded from above. Whence the existence of $R_1>0$ such that
\begin{align*}
P(\varphi+tu)^{\star} \leq \varphi + tR_1.
\end{align*}
By definition of $P(\varphi+tu)^{\star}$, we have $P(\varphi+tu)^{\star} \leq (\varphi +tu)^{\star} = \varphi+tu$. Hence there exists $R_0>0$ such that 
\begin{align*}
P(\varphi+tu)^{\star} \leq \varphi+tu_{R_0}.
\end{align*}
But $P(\varphi+tu)^{\star}$ is a candidate in the definition of $P(\varphi+tu_{R_0})$, hence $P(\varphi+tu_{R_0}) \geq P(\varphi+tu)^{\star}$. For all $R\geq R_0$, we have $P(\varphi+tu)^{\star} \leq P(\varphi+tu_{R_0}) \leq P(\varphi+tu_{R}) \leq P(\varphi+tu)^{\star}$, hence the conclusion.
\end{proof}

Lemma \ref{lemma:Nystrom_14.1} allows to define, for all $t\geq 0$, a Monge-Ampère measure $(\diff\dc P(\varphi +tu))^{\wedge d} := (\diff\dc P(\varphi +tu_R))^{\wedge d}$ for all $R> 0$ sufficiently big.

Let $n>0$, $L':=L+\frac{1}{n}A$ and $nL'-A=nL$ are semiample and big. Lemma \ref{lemma:Nystrom_14.2} and Proposition \ref{prop:Nystrom_14.3} yield a concave transform $G_{n\varphi + (\phi-\psi_{s})}$ and the equality
\begin{align*}
\cE\circ P(n\varphi+(\phi-\psi_{s}),n\psi):=\cE \circ P(n\varphi+(\phi-\psi_{s,R}),n\psi) = d!\int_{\Delta(nL)^{\circ}}G_{n\varphi+(\phi-\psi_{s})}(x)-G_{n\psi}(x)\lambda(\diff x),
\end{align*}
for all $R>0$ sufficiently big.

\begin{proposition}
\label{prop:differentiability_singular_direction}
For all $t\in \bR_{\geq 0}$, let $f(t) := \vol_{\chi}(L,\varphi+tu,\psi) = (d+1)\cE\circ P(\varphi+tu,\psi)$. Then $f$ is continuously differentiable function and we have 
\begin{align*}
\forall t\geq 0, \quad f'(t) = (d+1)\int_{X^{\an}}u(\diff\dc P(\varphi+tu))^{\wedge d}.
\end{align*}
\end{proposition}

\begin{proof}
Let $t\in\bR_{\geq 0}$ and $k\in \bN$, we denote $f_k(t):= \vol_{\chi}(L,\varphi+tu_k,\psi)$. $(f_k)_{k\geq 0}$ is monotone sequence of concave and continuously differentiable functions converging pointwise to $f$. Theorem \ref{th:Boucksom21_6.7} implies that, for all $k\geq 0$, for all $t\geq 0$, we have
\begin{align*}
f'_k(t) = (d+1)\int_{X^{\an}} u_k (\diff\dc P(\varphi+tu_k))^{\wedge d}.
\end{align*}

\begin{lemma}
\label{lemma:uniform_convergence_envelopes}
Let $t\geq 0$. Let $(t_j)_{j\geq 0}$ be a monotone real sequence of non-negative real numbers with limit $t$. Then the sequence of metrics $(P(\varphi+t_ju))_{j\geq 0}$ converges uniformly to $P(\varphi+tu)$.
\end{lemma}

\begin{proof}
For all $j\geq 0$, let $\psi_j := P(\varphi+t_ju) \in \PSH(L)\cap C^0(L)$ and $\psi := P(\varphi+tu)$. We first assume that the sequence $(t_j)_{j\geq 0}$ is decreasing.  Then $(\psi_j)_{j\geq 0}$ is a decreasing sequence of psh metrics on $L$ and bounded from below by $\psi \not\equiv -\infty$. Therefore $P(\varphi+tu)\leq \rho := \lim_{j\to +\infty} \psi_j \in \PSH(L)$. For all $j\geq 0$, $\psi_j \leq \varphi+t_ju$. Thus $\rho \leq \varphi +tu$ and $\rho \leq P(\varphi+tu)$. Hence $\rho = P(\varphi+tu)$. 

We now consider the case where $(t_j)_{j\geq 0}$ is increasing. Denote $\rho := \lim_{j\to +\infty} \psi_j$. Then $\rho \leq \psi = P(\varphi+tu)$. Let $\epsilon >0$ and $\Psi \in \PSH(L)\cap C^0(L)$ such that $\Psi \leq \varphi +tu$. For all $j\geq 0$, $\psi_j-\Psi$ is lsc. Hence for all $x\in X^{\an}$, there exists $j_0\geq 0$ such that $\{\Psi < \psi_j + \epsilon\}$ is an open neighbourhood of $x$ for all $j\geq j_0$. By compactness of $X^{\an}$, there exists $j_0\geq 0$ such that, for all $j\geq j_0$, $\Psi \leq \psi_j + \epsilon$. Hence $P(\varphi+tu)\leq \rho +\epsilon$ and $P(\varphi+tu)=\rho$. 

In both cases, as $\rho\in C^0(L)$, Dini's lemma yields the uniform convergence of $(\psi_j)_{j\geq 0}$ to $\psi$.
\end{proof}

Lemma \ref{lemma:uniform_convergence_envelopes} and Theorem 8.12 of \cite{BoucksomEriksson} yield the continuity of the map
\begin{align*}
\fonction{g}{\bR_{\geq 0}}{\bR}{t}{(d+1)\int_{X^{\an}}u(\diff\dc P(\varphi+tu))^{\wedge d}.}
\end{align*}
Then $(f'_k)_{k\geq 0}$ is a sequence of real decreasing functions converging simply to a continuous function $g$. Dini's lemma provides the compact convergence of $(f'_k)_{k\geq 0}$ to $g$. Hence $(f_k)_{k\geq 0}$ converges compactly to the continuous function $f$, which is therefore continuously differentiable with derivative $g$. 
\end{proof}

\begin{corollary}
\label{corollary:differentiability_non_relative_singular_direction}
We use the same notation as in Proposition \ref{prop:differentiability_singular_direction}. Then the function
\begin{align*}
f_{\varphi} : (t\in\bR_{\geq 0}) \mapsto (d+1)!\int_{\Delta(L)} G_{\varphi+tu}(x)\lambda(\diff x),
\end{align*}
is continuously differentiable and, for any $t\geq 0$, we have 
\begin{align*}
f_{\varphi}'(t) = (d+1)\int_{X^{\an}}u(\diff\dc P(\varphi+tu))^{\wedge d}.
\end{align*}
\end{corollary}

\begin{proof}
This is a consequence of Proposition \ref{prop:differentiability_singular_direction} combined with the equality
\begin{align*}
f(t) = f_{\varphi}(t) - \int_{\Delta(L)}G_{\psi}(x)\lambda(\diff x).
\end{align*}
\end{proof}

To obtain a differentiability property of the relative $\chi$-volume, we will make use of the behaviour of concave transforms on slices of Okounkov bodies. As $X$ is geometrically integral, one may assume that there exists a regular rational point $p\in X(k)$. In $p$, we set a system of coordinates $(z_1,...,z_d)$ such that $z_1=0$ is the local equation in $p$ of a Cartier divisor $E$ of $X$ whose underlying subscheme is assumed to be integral and normal. We fix a big class $\xi \in N^1(X)_{\bR}$, whose augmented base locus is denoted by $B_{+}(\xi)$, and we assume that $E \nsubseteq B_{+}(\xi)$, . We denote
\begin{align*}
\Delta(\xi)_{t} &:= \Delta(\xi) \cap ({t} \times \bR^{d-1}),\\
\Delta(\xi)_{\geq t} &:= \Delta(\xi) \cap ([t,+\infty[ \times \bR^{d-1}).
\end{align*}
Let $e\in N^1(X)_{\bR}$ be the class of $E$ and we set $\mu(\xi;e) := \sup\{s>0 : \xi - s\cdot e \in \Big(X)\}$. $E \nsubseteq B_{+}(\xi)$ implies $\Delta(\xi)_0 \neq \emptyset$ (cf. $\xi_{|E} \in \Big(E)$). 

\begin{theorem}[\cite{LazarsfeldMustata09}, Theorem 4.26]
\label{th:LazarsfeldMustata_th:4.26}
Using the same notation as above, for all $0\leq t\leq \mu(\xi;e)$, we have
\begin{align*}
\Delta_{\geq t}(\xi) = \Delta(\xi-t\cdot e) +t\cdot e_1,
\end{align*}
where $(e_1,...,e_d)$ denotes the canonical basis of $\bR^d$. Furthermore,
\begin{align*}
\Delta(\xi)_{t} = \Delta_{X|E}(\xi-t\cdot e),
\end{align*}
where $\Delta_{X|E}(\xi-t\cdot e)$ is the Okounkov body of a big $\bR$-divisor on $E$ representing the class $(\xi-t\cdot e)_{|E}$ (Okounkov bodies are constructed with respect to the point $p \in E^{\reg}$).
\end{theorem}

We further assume that $E$ is an irreducible component of a Cartier divisor $\div(s)$, where $s\in H^{0}(X,A)$ is a global section of an ample invertible $\cO_X$-module $A$ such that $s=z_1$ locally in $p$. Let $t\geq 0$ and $\Phi \in C^0(A)$ be a reference metric. We will compare the concave transforms $G_{(L+tA,\varphi+t\Phi)}$ et $G_{(L,\varphi+t(\Phi-\psi_{s}))}$ on $\Delta(L+tA)_{\geq t}$.

\begin{proposition}
\label{prop:Nystrom_14.9}
Using the same notation as above, for all $(a,\alpha) \in \bR\times \bR^{d-1} \cap \Delta(L+tA)_{\geq t}$, we have
\begin{align}
\label{eq:equality_Nystrom_14.9}
G_{(L+tA,\varphi+t\Phi)}(a,\alpha) = G_{(L,\varphi+t(\Phi-\psi_{s}))}(a-t,\alpha).
\end{align}
\end{proposition}

\begin{proof}
First note that, for all $(a,\alpha)\in \Delta(L+tA)_{\geq t}$, Theorem \ref{th:LazarsfeldMustata_th:4.26} yields $(a-t,\alpha)\in \Delta(L)$.

Let $u$ be a non-zero rational section of $L$, denote $D:= \div(u)$ and by $g_{\varphi}$ the Green function $-\ln|u|_{\varphi}$. Then for all $f\in H^0(D)$, $f\cdot 1 \in H^0_{\bR}(D+tE)$ and, for all $x\in X^{\an}$, we have
\begin{align}
\label{eq:equality_sup_norms_Green_functions}
|f\cdot 1|_{g_{\varphi}+tg_{\Phi}}(x) = |f|_{g_{\varphi}}(x) \cdot |1|_{tg_{\Phi}}(x) = |f|_{g\varphi + tg_{\Phi-\psi_{s}}}(x).
\end{align}
Thus, for all $n\geq 1$, the injective morphism $\iota_n : H^0(nD) \to H^0_{\bR}(nD+ntE)$ induces an isometry on its image, where $H^0(nD)$, resp. $H^0_{\bR}(nD+ntE)$, is endowed with the norm $\|\cdot\|_{ng_\varphi+ntg_{\Phi-\psi_{s}}}$, resp. $\|\cdot\|_{g_{n\varphi+ntg_{\Phi}}}$. Explicitly, we have
\begin{align*}
\iota_n(H^0(nD)) = \{f\in H^0_{\bR}(nD+ntE) : \quad v(f) \geq (nt,0,...,0)\}.
\end{align*}
Hence the equivalence
\begin{align*}
\forall f\in H^0(nD), \quad f \in \gr^{(n,v(f))}(V_{\bullet}(D)) \Leftrightarrow \iota(f) \in \gr^{(n,v(f)+nte_1)}(V_{\bullet}(D+tE)))
\end{align*}
Therefore, $\iota_n$ induces an isometry between $\gr^{(n,(na,n\alpha))}(V_{\bullet}(D))$, endowed with the sub-quotient norm $\|\cdot\|_{(n,(na,n\alpha)),D}$ induced by $\|\cdot\|_{ng_\varphi+ntg_{\Phi-\psi_{s}}}$, and $\gr^{(n,(n(a+t),n\alpha))}(V_{\bullet}(D+tE)))$, endowed with the sub-quotient norm $\|\cdot\|_{(n,(n(a+t),n\alpha)),D+tE}$ induced by $\|\cdot\|_{g_{n\varphi+ntg_{\Phi}}}$. Furthermore, $\iota_n$ also induces an isometry between $\gr^{(n,(na,n\alpha))}(V_{\bullet}(D))$, endowed with the canonical norm $\|\cdot\|_{\eta^{(n,(na,n\alpha))}}$ (introduced in \S \ref{sub:Okounkov_bodies_and_slopes}), and $\gr^{(n,(n(a+t),n\alpha))}(V_{\bullet}(D+tE)))$, endowed with the canonical norm $\|\cdot\|_{\eta^{(n,(n(a+t),n\alpha))}}$.

Finally, the equality
\begin{align*}
&\widehat{\deg}(\gr^{(n,(na,n\alpha))}(V_{\bullet}(D)),\|\cdot\|_{(n,(na,n\alpha)),D},\|\cdot\|_{\eta^{(n,(na,n\alpha))}}) \\
= &\widehat{\deg}(\gr^{(n,(n(a+t),n\alpha))}(V_{\bullet}(D+tE))),\|\cdot\|_{(n,(n(a+t),n\alpha)),D+tE},\|\cdot\|_{\eta^{(n,(n(a+t),n\alpha))}}),
\end{align*}
holds for all $n\in \bN_{\geq 1}$ and $(a,\alpha)\in \bN^d$. (\ref{eq:equality_Nystrom_14.9}) is then obtained from the definition of concave transforms.
\end{proof}

We keep the same notation as above. By Theorem \ref{th:LazarsfeldMustata_th:4.26}, we have $\Delta(L)_{0} =  \Delta_{X|E}(L)$. Note that $E\nsubseteq B_+(L)$ implies that $L_{|E}$ is both semiample and big. Moreover, we have an injective morphism
\begin{align*}
\iota : \left((0,\alpha) \in \Delta(L)_{0}\right) \mapsto \alpha \in \Delta(L_{|E}).
\end{align*}
If $\varphi \in C^0(L)$, $P(\varphi)\in \PSH(L)\cap C^0(L)$ and therefore $P(\varphi)_{|E}\in \PSH(L_{|E})\cap C^0(L_{|E})$. We also have the concave transforms $G_{(L_{|E},P(\varphi)_{|E})} : \Delta(L_{|E}) \to \bR \cup \{-\infty\}$ and $G_{(L,\varphi)}:\Delta(L) \to \bR\cup\{-\infty\}$. The latter are related by the following proposition.

\begin{proposition}
\label{prop:Nystrom_11.10}
We use the same notation as above and assume that $L$ is ample. We have the equality
\begin{align*}
G_{(L,\varphi)}(0,\alpha) = G_{(L_{|E},P(\varphi)_{|E})}(\alpha),
\end{align*}
for all $(0,\alpha)\in \Delta(L)_0$. In particular, $G_{(L,\varphi)}$ takes real values in the relative interior of $\Delta(L)_0$ in $\Delta(L_{|E})$.
\end{proposition}

\begin{proof}
By abuse of notation, we denote $G_1 := G_{(L,\varphi)} $ and $G_2 = G_{(L_{|E},P(\varphi)_{|E})} \circ \iota$. By definition of $P(\varphi)$, we have $P(\varphi)_{|E} \leq \varphi_{|E}$. For all $(n,(0,n\alpha)) \in \Gamma(L)$, we denote by $\gr^{(n,(0,n\alpha))}$, resp. $\gr^{(n,\alpha)}$, the $(n,(0,\alpha))$-graded component, resp. $(n,\alpha)$-graded component, given by the Okounkov filtration of the graded linear series $\bigoplus_{m\in\bN} H^0(X,mL)$, resp. $\bigoplus_{m\in\bN} H^0(E,mL_{|E})$. For all $(n,(0,n\alpha)) \in \Gamma(L)$, we denote by $\varphi^{(n,(0,n\alpha))}$, resp. $P(\varphi)_{|E}^{(n,n\alpha)}$, the sub-quotient norm on $\gr^{(n,(0,n\alpha))}$ induced by $\varphi$, resp. the sub-quotient norm on $\gr^{(n,\alpha)}$ induced by $P(\varphi)_{|E}$. Finally, for all $(n,(0,n\alpha)) \in \Gamma(L)$, we denote by $\eta^{(n,(0,n\alpha))}$, resp. $\eta^{(n,n\alpha)}$, the canonical norm on $\gr^{(n,(0,n\alpha))}$, resp. on $\gr^{(n,\alpha)}$.

For any integer $n\geq 1$, we have the inequality $\|\cdot\|_{nP(\varphi)_{|E}}\leq \|\cdot\|_{nP(\varphi)} = \|\cdot\|_{n\varphi}$. Moreover, for all $(n,(0,n\alpha)) \in \Gamma(L)$, $\iota$ induces an isometry between $(\gr^{(n,(0,n\alpha))},\|\cdot\|_{\eta^{(n,(0,n\alpha))}})$ and its image in $(\gr^{(n,\alpha)},\|\cdot\|_{\eta^{(n,n\alpha)}})$. Therefore, we obtain the inequality
\begin{align*}
\widehat{\deg}(\gr^{(n,(0,n\alpha))},\|\cdot\|_{\varphi^{(n,(0,n\alpha))}},\|\cdot\|_{\eta^{(n,(0,n\alpha))}}) \leq \widehat{\deg}(\gr^{(n,n\alpha)},\|\cdot\|_{P(\varphi)_{|E}^{(n,n\alpha)}},\|\cdot\|_{\eta^{(n,n\alpha)}}),
\end{align*}
for all $(n,(0,n\alpha)) \in \Gamma(L)$. Hence $G_1\leq G_2$.

To prove the inverse inequality, we use the following extension property (\cite{Fang22}, Theorems 1.1 and 4.1). Let $\epsilon >0$ and $n\geq 1$. Then for all  $\tilde{t} \in H^0(E,nL_{|E})$, there exists an integer $m_0\geq 0$ such that, for all $m\geq m_0$, there exists $t_m \in H^0(X,nmL)$ such that 
\begin{align}
\label{eq:extension_property}
(t_m)_{|E} = \tilde{t}^m \text{ and } \|t_m\|_{nm\varphi} = \|t_m\|_{nmP(\varphi)} \leq e^{m\epsilon}\|\tilde{t}^m\|_{nmP(\varphi)_{|E}}.
\end{align}
Let $t\in H^0(X,nL)$ and assume that $v(t)$ is of the form $(0,\alpha)\in \bN^d$. Denote $\tilde{t}:=t_{|E} \in H^0(X|E,nL)$, we have $v(t)/n \in \Delta(L)_0$ and $v(\tilde{t})/n\in \Delta(L_{|E})$. Let $u\in \bR$ be such that 
\begin{align*}
nu\leq \widehat{\deg}(\gr^{(n,v(\tilde{t}))},P(\varphi)_{|E}^{(n,v(\tilde{t}))},\eta^{(n,v(\tilde{t}))}).
\end{align*}
Let $m\gg 0$ and $t_m\in H^0(X,nmL)$ be such that both equalities in (\ref{eq:extension_property}) hold. Then
\begin{align*}
m\widehat{\deg}(\gr^{(n,v(\tilde{t}))},P(\varphi)_{|E}^{(n,v(\tilde{t}))},\eta^{(n,v(\tilde{t}))}) &\leq\widehat{\deg}(\gr^{(nm,mv(\tilde{t}))},P(\varphi)_{|E}^{(nm,mv(\tilde{t}))},\eta^{(nm,mv(\tilde{t}))}) \\
& = -\ln\|\tilde{t}^m\|_{P(\varphi_{|E})^{(nm,mv(\tilde{t}))}} + \ln\|\tilde{t}^m\|_{\eta^{(nm,mv(\tilde{t}))}}\\
 & \leq -\ln\|t_m\|_{\varphi^{nm,v(t_m)}}+\ln\|t_m\|_{\eta^{(nm,v(t_m))}} + m\epsilon\\
 & = \widehat{\deg}(\gr^{(nm,v(t_m))},\varphi^{(nm,v(t_m))},\eta^{(nm,v(t_m))}) + m\epsilon.
\end{align*}
Thus $nm(u-\epsilon/n) \leq \widehat{\deg}(\gr^{(nm,v(t_m))},\varphi^{(nm,v(t_m))},\eta^{(nm,v(t_m))})$. Furthermore, the equality  $v(f_m)=(0,mv(\tilde{t}))$ implies
\begin{align*}
\iota\left(\frac{v(t)}{n}\right)=\iota\left(\frac{v(t_m)}{nm}\right) = \frac{v(\tilde{t})}{n} \Rightarrow \frac{v(t)}{n}=\frac{v(t_m)}{nm}.
\end{align*}
Hence finally
\begin{align*}
u - \frac{\epsilon}{m} \leq G_1\left(\frac{v(t)}{n}\right) \Rightarrow G_2\left(\frac{v(t)}{n}\right) \leq G_1\left(\frac{v(t)}{n}\right) + \frac{\epsilon}{n}.
\end{align*}
Note that $\epsilon$ is independent on $\tilde{t}$, whence we have 
\begin{align*}
\forall x \in \displaystyle\bigcup_{n\in\bN_{\geq 1}}\frac{1}{n}(0,v(H^{0}(X|E,nL))) ,\quad  G_2(x) \leq G_1(x).
\end{align*}
By density we obtain the inequality $G_2 \leq G_1$ on $\Delta(L)_0$.
\end{proof}

\begin{corollary}
\label{cor:Nystrom_11.11}
Using the same notation as in Proposition \ref{prop:Nystrom_11.10}, for all $\varphi,\psi\in C^0(L)$, we have
\begin{align*}
\vol_{\chi}(L_{|E},P(\varphi)_{|E},P(\psi)_{|E}) = d! \int_{\Delta(L)_0} \left(G_{(L,\varphi)}-G_{(L,\psi)}\right)(0,\alpha)\diff\alpha. 
\end{align*}
\end{corollary}

\begin{proof}

We first show that in the current context, we have $\iota(\Delta(L)_0) = \Delta(L_{|E})$. For all $n\geq 1$, for all $\tilde{t}\in H^0(E,nL_{|E})$ , we have $v(\tilde{t})/n \in \Delta(L_{|E})$ and there exist $m\geq 1$ and $t\in H^0(X,nmL)$ such that $t_{|E}=\tilde{t}^m$. Then $v(t)/nm \in H^0(X,nmL)\cap \Delta(L)_0$ and
\begin{align*}
\displaystyle\bigcup_{n\geq 1} \frac{1}{n}v(H^0(E,L_{|E})) \subset \iota \left(\bigcup_{n\geq 1} \frac{1}{n}H^0(X,nmL)\cap \Delta(L)_0 \right).
\end{align*}
Hence $\iota(\Delta(L)_0) = \Delta(L_{|E})$ as $\iota$ is closed and continuous. 

Then now the discussion following Theorem \ref{th:ChenMaclean_4.6} and Proposition \ref{prop:Nystrom_11.10} give 
\begin{align*}
\vol_{\chi}(L_{|E},P(\varphi)_{|E},P(\psi)_{|E}) &= d! \int_{\Delta(L_{|E})} \left(G_{(L_{|E},P(\varphi)_{|E})}-G_{L_{|E},P(\psi)_{|E}}\right)(\alpha)\diff\alpha\\
& = d! \int_{\iota(\Delta(L)_0)} \left(G_{(L_{|E},P(\varphi)_{|E})}-G_{L_{|E},P(\psi)_{|E}}\right)(\alpha)\diff\alpha\\
& = d! \int_{\Delta(L)_0} \left(G_{(L,\varphi)}-G_{(L,\psi)}\right)(0,\alpha)\diff\alpha.
\end{align*}
\end{proof}

\subsection{Proof of the local differentiability}
\label{sub:proof_of_the_local_differentiability}

%Ample case
We first introduce some notation. Let $m\geq 1$ be an integer.
\begin{itemize}
	\item[(i)] Let $L=L_0$ be an ample line bundle equipped with a continuous metric $\varphi=\varphi_0\in C^0(L)$.	
	\item[(ii)] For any $i=1,...,m$, let $L_i$ a line bundle on $X$ equipped with a continuous metric $\varphi_i\in C^0(L)$.
\end{itemize}
For all $i=0,...,m$, let $s_i$ be a non-zero rational section of $L_i$, $D_i:=\div(s_i)$ be its corresponding Cartier divisor and $g_i$ the Green function on $D_i$ associated to the metric $\varphi_i$. Denote $D:=D_0$ and $g:=g_0$. By abuse of notation, for all $a\in \bR^{m}$, denote $\overline{L_a}=(L_a,\varphi_a)$
\begin{align*}
\overline(L_a) := (L_a,\varphi_a) := \left(D + \displaystyle\sum_{i=1}^{m} a_i D_i, g+ \sum_{i=1}^{m} a_i g_i\right).
\end{align*}
We also denote $\mathbf{L}:=(\overline{L},\overline{L_1},...,\overline{L_m})$ and $O_{\mathbf{L}} := \{a\in \bR^m : L_a \text{ is ample}\}$. We will study the function
\begin{align*}
\fonction{f^{\mathbf{L}}}{O_{\mathbf{L}}}{\bR}{a}{(d+1)!\int_{\Delta(L_a)}G_{\varphi_a}(x)\lambda(\diff x)}
\end{align*}

\begin{theorem}
\label{th:local_differentiability_general}
The function $f^{\mathbf{L}}$ is continuously differentiable on $O_{\mathbf{L}}$.
\end{theorem}

\begin{proof}
It is enough to prove the continuous differentiability of $f^{\mathbf{L}}$ at $0\in O_{\mathbf{L}}$.  It suffices to prove that $f^{\mathbf{L}}$ admits continuous partial derivatives at $0$, and by symmetry, a continuous partial derivative along the first coordinate at $0$. Denote $\overline{A}=(A,\Phi):=(L_1,\varphi_1)$. Let 
\begin{align*}
U := \{t\in \bR : te_1 \in O_{\mathbf{L}}\},
\end{align*}
it is an open subset of $\bR$. 

We argue by induction on $d=\dim(X)$. We first consider the case $d=0$. Then for all $t\in U$, we have
\begin{align*}
f^{\mathbf{L}}(te_1) = f^{\mathbf{L}}(0) + t(d+1)!\int_{\Delta(A)}G_{\Phi}(x)\lambda(\diff x).
\end{align*}
Hence $f^{\mathbf{L}}$ admits a partial derivative at $0$ and
\begin{align*}
\frac{\partial f^{\mathbf{L}}}{\partial x_1}(a) = (d+1)!\int_{\Delta(A)}G_{\Phi}(x)\lambda(\diff x),
\end{align*}
which varies continuously with respect to $(L,\varphi)$.

We now assume that $d=\dim(X)\geq 1$. We start by a reduction.

\begin{lemma}
\label{lemma:reduction_local_differentiability}
It suffices to prove Theorem \ref{th:local_differentiability_general} in the case where:
\begin{itemize}
	\item[(i)] $A$ is very ample;
	\item[(ii)] there exists a regular rational point $p\in X(k)$;
	\item[(iii)] there exist a non-zero global section $s\in H^0(X,A)$, an irreducible component $E$ of $\div(s)$ and a coordinate system $(z_1,...,z_d)$ of $X$ in $p$ of $\div(s)$ such that $E$ is an effective Cartier divisor whose underlying subscheme is a geometrically integral normal subscheme of $X$ satisfying $E \nsubseteq B_{+}(L)$ and $z_1$ is a local equation of $E$ in $p$.
\end{itemize}  
\end{lemma}

\begin{proof}

%big case
We first show that one may assume that $A$ is very ample. Write $(A,\Phi) = (A_1,\Phi_1)- (A_2,\Phi_2)$, where $A_1,A_2$ are both very ample. Let $\mathbf{L}_{A}:=(\overline{A_1},\overline{A_2},\overline{L_2},...,\overline{L_m})$. For all $t\in U$, both $(t,-t)\in O_{\mathbf{L}_A}$ and $(t\in U) \mapsto f^{\mathbf{L}}(te_1)$ is continuously differentiable at $0$ iff $(t\in U) \mapsto f^{\mathbf{L}_A}(t,-t,0,...,0)$ is so. 
%Thus we assume that $A$ is very ample. 

%Ample case
By homogeneity, we may assume that $A$ is very ample. If $\dim(X)=1$, by a theorem of Bertini (\cite{Flenner99}, Corollary 3.4.14), the set  
\begin{align*}
\{s \in H^0(X,A) : \div(s) \text{ is geometrically reduced and normal}\},
\end{align*}
is a dense Zariski open set. Since $X$ is geometrically integral, we may assume that there exist $s\in H^{0}(X,A)$ and an irreducible component $E$ of $\div(s)$ such that $E=\{p\}$, where $p\in X(k)$ is a regular rational point. Since $p$ is regular, $E$ is a Cartier divisor on $X$. Moreover, $E$ is geometrically integral and normal and $B_{+}(L)=\emptyset$ since $L$ is ample.

Likewise, in the $d=\dim(X)\geq 2$ case, a theorem of Bertini (\cite{Jouanolou83}, Theorème 6.3 and \cite{Flenner99}, Corollary 3.4.14) implies that the set
\begin{align*}
\{s \in H^0(X,A) : \div(s) \text{ is geometrically integral and normal}\},
\end{align*} 
is a dense Zariski open set. Since $B_{+}(L)\subsetneq X$ is a proper closed subscheme, there exists $s\in H^0(X,A)$ such that $E:=\div(s)$ is a geometrically integral normal subscheme of $X$ satisfying $E \nsubseteq B_{+}(L)$. One also may assume that there exists a regular rational point $p\in E(k)\cap E^{\reg}$ which belongs to $X(k)\cap X^{\reg}$. 

In both cases, as $p$ is both a regular point of $X$ and of $E$, there exists a coordinate system $(z_1,...,z_d)$ of $X$ in $p$ such that $z_1$ is a local equation of $E$ in $p$.
\end{proof}

From now on, we use the same notation and hypotheses of Lemma \ref{lemma:reduction_local_differentiability}. In the following, Okounkov bodies are constructed with respect to the filtrations induced by the lexicographic order on $\bN^d$ and the isomorphism $\widehat{\cO_{X,p}} \cong k[\![z_1,...,z_d]\!]$ defined by the choice of the coordinate from Lemma \ref{lemma:reduction_local_differentiability} $(iii)$. Moreover we may assume that $\mathbf{L}=(\overline{L},\overline{A})$

We first prove the right differentiability of $f^{\mathbf{L}}$. For any $0\leq t\in U$, we have 
\begin{align}
\label{eq:differential_local}
f^{\mathbf{L}}(t) &= (d+1)!\int_{\Delta(L+tA)}G_{(L+tA,\varphi+t\Phi)}(p)\diff p \nonumber\\
&= (d+1)!\int_{0}^{t}\int_{\Delta(L+tA)_{r}}G_{(L+tA,\varphi+t\Phi)}(r,\alpha)\diff\alpha\diff r\nonumber\\
+ &(d+1)!\int_{\Delta(L+tA)_{\geq t}} G_{(L+tA,\varphi+t\Phi)}(p)\diff p \nonumber\\
& = (d+1)!\int_{0}^{t}\int_{\Delta(L+tA)_{r}} G_{(L+tA,\varphi+t\Phi)}(r,\alpha)\diff\alpha\diff r\nonumber\\
+ & (d+1)!\int_{\Delta(L)} G_{(L,\varphi+t(\Phi-\psi_{s}))}(q)\diff q,
\end{align}
where we used 
\begin{align*}
\Delta(L+tA) = \displaystyle\bigcup_{0\leq r < t} \Delta(L+tA)_{r} \cup \Delta(L+tA)_{\geq t}
\end{align*}
for the first equality and Proposition \ref{prop:Nystrom_14.9} for the second one.

We first consider the first term of the RHS of (\ref{eq:differential_local}). For all $r,t\in \bR_{\geq 0}$ with $r\leq t$, let
\begin{align*}
f(r,t) = (d+1)!\int_{\Delta(L+tA)_r} G_{(L+tA,\varphi+t\Phi)}(\alpha)\diff\alpha.
\end{align*}
Assume that $r<t$. Then $L+(t-r)A$ is ample and Theorem \ref{th:LazarsfeldMustata_th:4.26} combined with Propositions \ref{prop:Nystrom_14.9} and \ref{prop:Nystrom_11.10} yield
\begin{align*}
f(r,t) &= (d+1)!\int_{\Delta(L+(t-r)A)_0}G_{(L+(t-r)A,\varphi+t\Phi-r\psi_{s})}(0,\alpha)\diff\alpha\\
& = (d+1)!\int_{\Delta(L_{|E}+(t-r)A_{|E})} G_{\left(L_{|E}+(t-r)A_{|E},P(\varphi+t\Phi-r\psi_{s})_{|E}\right)}(\alpha)\diff\alpha.
\end{align*}
By induction hypothesis, $f$ is a continuous function on $\{(r,t)\in(\bR_{\geq 0})^2 : r\leq t\}$. Let
\begin{align*}
\fonction{F}{\bR_{\geq 0}}{\bR,}{t}{\int_{0}^{t}f(r,t)\diff r.}
\end{align*}
Then
\begin{align*}
\left|\frac{F(t)-F(0)}{t}- f(0,0)\right| = \left|\int_{0}^{t} \frac{f(r,t)-f(0,0)}{t}\diff r \right|\leq \displaystyle\sup_{0\leq r \leq t} |f(r,t)-f(0,0)| \underset{t\to 0}{\longrightarrow}0.
\end{align*}
Hence $F$ is right differentiable in $0$ and we have
\begin{align}
\label{eq:limit_RHS_local_differentiability_line_bundle_slice}
F'(0) = f(0,0) = (d+1)!\int_{\Delta(L_{|E})} G_{(L_{|E},P(\varphi)_{|E})}(\alpha)\diff\alpha.
\end{align}

We now study the second term of the RHS of (\ref{eq:differential_local}). For any $t\in U\setminus\{0\}$, let
\begin{align*}
A(t) := \frac{(d+1)!}{t} \int_{\Delta(L)} \left(G_{(L,\varphi+t(\Phi-\psi_{s}))}-G_{(L,\varphi)}\right)(q)\diff q. 
\end{align*}
Then Proposition \ref{prop:Nystrom_14.3} yields
\begin{align*}
\forall t \in U\setminus\{0\},\quad A(t) = \frac{1}{t}\vol_{\chi}(L,\varphi+t(\Phi-\psi_s),\varphi).
\end{align*}
Now Proposition \ref{prop:differentiability_singular_direction} implies
\begin{align}
\label{eq:limit_RHS_local_differentiability_line_bundle_superior_slice}
\displaystyle\lim_{t\to 0^{+}} A(t) = (d+1)\int_{X^{\an}}(\Phi-\psi_{s})(\diff\dc P(\varphi))^{\wedge d}.
\end{align}
Combining (\ref{eq:limit_RHS_local_differentiability_line_bundle_slice}) and (\ref{eq:limit_RHS_local_differentiability_line_bundle_superior_slice}), we obtain that $f^{\mathbf{L}}$ is right differentiable at $0$ and
\begin{align}
\label{eq:local_differential}
\displaystyle\lim_{t\to 0^{+}} \frac{f^{\mathbf{L}}(t)-f^{\mathbf{L}}(0)}{t} = (d+1)!\int_{\Delta(L_{|E})} G_{(L_{|E},P(\varphi)_{|E})}(\alpha)\diff\alpha \nonumber\\
+ (d+1)\int_{X^{\an}}(\Phi-\psi_{s})(\diff\dc P(\varphi))^{\wedge d}.
\end{align}
Finally, for any family $\mathbf{F}$ satisfying the conditions of Theorem \ref{th:local_differentiability_general}, the function $f^{\mathbf{F}}$ admits a right derivative at $0$ in any direction.

We now consider the left differentiability of $f^{\mathbf{L}}$. By homogeneity, we may assume that $L-A$ is ample. Let $t\in U$ such that $-1<-t<0$. Let $\mathbf{L'}:=(L,L-A)$. Then from the equality
\begin{align*}
L-tA = (1-t)\left(L+\frac{t}{1-t}(L-A)\right),
\end{align*}
we obtain 
\begin{align}
\label{eq:left_derivative}
f^{\mathbf{L}}(-te_1) = (1-t)^{d+1}f^{\mathbf{L'}}\left(\frac{t}{1-t}\right).
\end{align}
Then the right differentiability result obtained above combined with the linearity of the right differential imply
\begin{align*}
\displaystyle\lim_{t\to 0^{+}}\frac{f^{\mathbf{L}}(-te_1)-f(0)}{-t} & =  (d+1)f^{\mathbf{L}}(0) - \lim_{t\to 0^{+}}\frac{f^{\mathbf{L'}}(t)-f(0)}{t}\\
& = (d+1)f^{\mathbf{L}}(0) - \lim_{t\to 0^{+}}\frac{f^{(\overline{L},\overline{L})}(t)-f(0)}{t} + \lim_{t\to 0^{+}}\frac{f^{(\overline{L},\overline{A})}(t)-f(0)}{t}\\
& = \lim_{t\to 0^{+}}\frac{f^{\mathbf{L}}(t)-f(0)}{t}.
\end{align*}
Thus $f^{\mathbf{L}}$ is right and left differentiable at $0$ and the derivatives are equal. Therefore we obtain the general differentiability result of Theorem \ref{th:local_differentiability_general}.

To conclude the proof of Theorem \ref{th:local_differentiability_general}, it suffices to prove that $g:= \frac{\partial f^{\mathbf{L}}}{\partial x_1}$ is continuous at $0$. We use the same notation as above: namely $\mathbf{L}=(L,A)$ and $U=\{t\in\bR : L+tA \text{ is ample}\}$. For any $t\in U$, we denote $(L_t,\varphi_t) := (L+tA,\varphi+t\Phi)$. By linearity of the differential, we may assume that the conditions of Lemma \ref{lemma:reduction_local_differentiability} hold for $A$. Recall that
\begin{align}
\label{eq:local_differential_2}
(f^{\mathbf{L}})'(0) = (d+1)!\int_{\Delta(L_{|E})} G_{(L_{|E},P(\varphi)_{|E})}(\alpha)\diff\alpha \nonumber\\
- (d+1)\int_{X^{\an}}\ln|s|_{\Phi}(\diff\dc P(\varphi))^{\wedge d}.
\end{align}

By induction hypothesis, the first term of the RHS of (\ref{eq:local_differential_2}) varies continuously w.r.t $(L,\varphi)$. Now the end of the proof is given by the following lemma.

\begin{lemma}
\label{lemma:continuity_local_differentiability}
The map
\begin{align*}
(t\in U) \mapsto \int_{X^{\an}}\ln|s|_{\Phi}(\diff\dc P(\varphi_t))^{\wedge d}
\end{align*}
is continuous.
\end{lemma}

\begin{proof}
It suffices to show the continuity at $0$. For all $t\geq 0$, let $\tilde{\varphi}_t := P(\varphi+t\Phi)-tP(\Phi) \in C^0(L)$. For all $0\leq s \leq t$, $P(\varphi+s\Phi)+(t-s)P(\Phi)\in \PSH(L+tA)$ and $P(\varphi+s\Phi)+(t-s)P(\Phi) \leq \varphi + t\Phi$, hence $\tilde{\varphi}_s \leq \tilde{\varphi}_t$. Denote $\tilde{\varphi} := \lim_{t\to 0^+}\tilde{\varphi}_t$, it is a bounded  psh metric on $L$ (cf. Remark \ref{rem:DPSH}) $L$. For all $t\geq 0$, $P(\varphi+t\Phi) \in \PSH(L+tA)$ implies the positivity of the currents $\diff\dc\tilde{\varphi}_t + t\diff\dc P(\Phi)$ and $\diff\dc \tilde{\varphi}$, hence $\tilde{\varphi}\in \PSH(L)$. 

Let $(t_n)_{n\geq 0}\in (O_{\mathbf{L}}\cap \bR_{>0})^{\bN}$ be a decreasing sequence of real numbers with limit $0$. Then $(P(\tilde{\varphi}_{t_n}))_{n\geq 0}$ is a decreasing sequence in $\PSH(L)$, bounded from below by $P(\varphi)$, hence converging to $P(\tilde{\varphi})\in \PSH(L)$ (cf. Proposition \ref{prop:PSH_metrics} $(iv)$ and continuity of the operator $P$). Furthermore, $P(\varphi)\leq P(\tilde{\varphi})\leq\varphi$ yields $P(\tilde{\varphi}) = P(\varphi)$. Note that the properties of the Monge-Ampère operator(cf. \cite{BoucksomEriksson}, \S 8.1) yield the weak convergence
\begin{align*}
\MA(P(\tilde{\varphi}_{t_n})) \displaystyle\rightharpoonup_{n\to +\infty} \MA(P(\varphi)).
\end{align*}
As $\tilde{\varphi}\in\PSH(L)$, we obtain $\tilde{\varphi}=P(\tilde{\varphi})=P(\varphi)$.

For all $n\geq 0$, writing $P(\varphi +t_n\Phi) = \tilde{\varphi}_{t_n} + t_n P(\Phi)$, we have an equality of positive Radon measures on $X^{\an}$ of the form
\begin{align*}
\MA(P(\varphi+t_n\Phi)) = \MA(\tilde{\varphi}_{t_n}) + O(t_n).
\end{align*}
Hence the weak convergence $\MA(P(\varphi+t_n\Phi)) \rightharpoonup_{n\to +\infty} \MA(P(\varphi))$ which concludes the proof of the lemma.
\end{proof}
\end{proof}

\begin{corollary}
\label{cor:differentiability_relative_chi_volume}
Let $L,M$ be line bundles on $X$, where $L$ is ample. Let $\varphi,\psi\in C^0(L)$ and $\Phi,\Psi\in C^0(L)$. Then the limit
\begin{align*}
\displaystyle\lim_{\bQ\ni t\to 0} \frac{\vol_{\chi}(L+tM,\varphi+t\Phi,\psi+t\Psi)-\vol_{\chi}(L,\varphi,\psi)}{t}
\end{align*}
exists. Assume that $M$ satisfies the conditions of Lemma \ref{lemma:reduction_local_differentiability}. Then the limits equals
\begin{align*}
(d+1)\vol_{\chi}(L_{|E},P(\varphi_{|E}),P(\psi_{|E})) - (d+1)\int_{X^{\an}}\ln|s|_{\Phi}(\diff\dc P(\varphi))^{\wedge d}.
\end{align*}
\end{corollary}

\begin{proof}
This is a consequence of Theorem \ref{th:local_differentiability_general} combined with the results of \S \ref{subsub:transformée_concave_R_diviseur}.
\end{proof}

\section{Differentiability over an adelic curve}
\label{sec:differentiability_adelic_curve}

\subsection{Adelic curves}
\label{sub:adelic curves}

We first recall several facts of the theory of adelic curves. The references are\cite{ChenMori,ChenMori21,ChenMoriHSP,CM22_equidistribution}.

\subsubsection{Adelic structure on a field}

An \emph{adelic curve} is the data $S=(K,(\Omega,\cA,\nu),(\va_{\omega})_{\omega\in \Omega})$ where $K$ is a field, $(\Omega,\cA,\nu)$ is a measure space and $(\va_{\omega})_{\omega\in \Omega}$ is a family of absolute values on $K$ satisfying the following condition:
\begin{align*}
\forall a \in K^{\times}, \quad (\omega\in\Omega) \mapsto \ln|a|_{\omega}
\end{align*} 
is $\cA$-measurable and $\nu$-integrable. The adelic curve $S$ is called \emph{proper} if the product formula
\begin{align*}
\forall a\in K^{\times},\quad \int_{\Omega}\ln|a|_{\omega} \nu(\diff\omega) = 0 
\end{align*}
holds. From now on, we fix a proper adelic curve $S=(K,(\Omega,\cA,\nu),(\va_{\omega})_{\omega\in \Omega})$. We further assume that $K$ is perfect and that either $K$ is countable or the $\sigma$-algebra $\cA$ is discrete. We denote by $\Omega_{\ar}$, resp. $\Omega_{\um}$, the set of $\omega\in\Omega$ such that $\va_{\omega}$ is Archimedean, resp. non-Archimedean.

\subsubsection{Adelic vector bundles on $S$}\label{subsub:adelic_vector_bundle} Let $V$ be a finite dimensional vector space over $K$. for all $\omega \in \Omega$, let $V_{\omega}:=V\otimes_{K} K_{\omega}$. A \emph{norm family} on $V$ is a family $\xi=(\|\cdot\|_{\omega})_{\omega\in \Omega}$ such that, for all $\omega\in\Omega$, $\|\cdot\|_{\omega}$ is a norm on $V_{\omega}$ (assumed to be ultrametric if $\omega\in\Omega_{\um}$). Let $\mathbf{e}:=(e_i)_{i=1}^{r}$ be a basis of $V$, we denote by $\xi_{\mathbf{e}}=(\|\cdot\|_{\mathbf{e},\omega})_{\omega\in\Omega}$ the norm family on $V$ defined by
\begin{align*}
\forall \omega\in\Omega,\quad \forall (\lambda_1,...,\lambda_r)\in K_{\omega}, \quad \|\lambda_{1}e_{1}+\cdots\lambda_{r}e_{r}\|_{\mathbf{e},\omega} := \left\{\begin{matrix}
\max_{i=1,...,r}|\lambda_{i}|_{\omega},\quad \text{si }\omega\in\Omega_{\um}, \\
|\lambda_{1}|_{\omega}+\cdots+|\lambda_{r}|_{\omega},\quad \text{si }\omega\in\Omega_{\ar}.
\end{matrix}\right.
\end{align*}

An \emph{adelic vector bundle} on $S$ is the data $\overline{V} = (V,\xi)$ where $V$ is a finite dimensional vector space over $K$ and $\xi=(\|\cdot\|_{\omega})_{\omega\in \Omega}$ is a norm family on $V$ satisfying the following conditions:
\begin{itemize}
	\item[(i)] $\xi$ is \emph{strongly dominated}, namely there exists a basis $\mathbf{e} = (e_i)_{i=1}^r$ of $V$ such that the local distance
	\begin{align*}
	\omega\in\Omega\mapsto d_{\omega}(\xi,\xi_{\mathbf{e}}) := \displaystyle\sup_{x \in V\setminus\{0\}} \left|\ln \frac{\|x\|_{\omega}}{\|x\|_{\mathbf{e},\omega}}\right|
	\end{align*}
is $\nu$-dominated;
	\item[(ii)] $\xi$ is \emph{measurable}, namely for all $s\in V$, $\omega\in\Omega \mapsto \|s\|_{\omega}$ is $\cA$-measurable. 
\end{itemize}

%Lorsque $\cA$ est la tribu discrète ou que $K$ est dénombrable, on dispose des constructions algébriques suivantes pour les fibrés vectoriels adéliques sur $S$ produit tensoriel, sous-fibré, quotient, dual, etc.) dans ce contexte (cela utilise fortement l'hypothèse sur $K$ ou $\cA$).

\subsubsection{Arakelov degree, positive degree, slope} Let $\overline{V} = (V,\xi)$ be an adelic vector bundle on $S$, its \emph{Arakelov degree} is defined by 
\begin{align*}
\widehat{\deg}(\overline{V)} := -\int_{\Omega} \ln\|e_1\wedge\cdots\wedge e_r\|_{\omega,\det}\nu(\diff\omega),
\end{align*}
where $(e_i)_{i=1}^{r}$ is an arbitrary basis of $V$ and $\|\cdot\|_{\omega,\det}$ denotes the \emph{determinant norm} on $\det(V)$ defined by
\begin{align*}
\forall \eta\in \det(V), \quad \|\eta\|_{\omega,\det} = \displaystyle\inf_{\eta=s_1\wedge\cdots\wedge s_r} \|s_1\|_{\omega}\cdots\|s_r\|_{\omega}.
\end{align*}
Morally, the Arakelov degree is the analogue of the Euler-Poincaré characteristic of an euclidean lattice.

The \emph{positive degree} of $\overline{V}$ is defined by
\begin{align*}
\widehat{\deg}_{+}(\overline{V}) := \displaystyle\sup_{W\subset V} \widehat{\deg}(\overline{W}) \geq 0.
\end{align*}
It is the counterpart of the number of "small sections" from the classical Arakelov theory (\cite{ChenMori}, \S 4.6). 

If $V\neq \{0\}$, we define its \emph{slope}  by 
\begin{align*}
\widehat{\mu}(\overline{V}) := \frac{\widehat{\deg}(\overline{V})}{\dim(V)},
\end{align*}
and its \emph{maximal slope}, resp. \emph{minimal slope}, by
\begin{align*}
\widehat{\mu}_{\max}(\overline{V}) := \displaystyle\sup_{W\subset V} \widehat\mu(\overline{W}), \quad \text{resp. } \widehat{\mu}_{\min}(\overline{V}) := \inf_{V \twoheadrightarrow W \neq \{0\}} \widehat\mu(\overline{W}),
\end{align*}
where $W$ runs over the set of vector subspaces, resp. of non-zero quotients, of $V$.

\subsubsection{Adelic line bundle on a projective scheme}
\label{subsub:adelic_line_bundle}
Throughout this paragraph, we fix a projective scheme $\pi : X \to \Spec(K)$ over $\Spec(K)$. For all $\omega\in\Omega$, we denote $X_{\omega}:=X\otimes_{K} K_{\omega}$ and by $X_{\omega}^{\an}$ the analytification in the sense of Berkovich of $X_{\omega}$.

Let $E$ be a finite locally free $\cO_{X}$-module. A \emph{metric family} on $E$ is a family $\varphi=(\varphi_{\omega})_{\omega\in\Omega}$ of continuous metrics on $E_{\omega}:=E\otimes_{\cO_{X}} \cO_{X_{\omega}}$. From now on, we fix an invertible $\cO_X$-module $L$ and a metric family $\varphi=(\varphi_{\omega})_{\omega\in\Omega}$ on $L$. Then $-\varphi:=(-\varphi_{\omega})_{\omega\in\Omega}$ is a metric family on $-L$. If $f: Y \to X$ is a morphism of projective scheme over $\Spec(K)$, $f^{\ast}\varphi:=(f_{\omega}^{\ast}\varphi_{\omega})_{\omega\in\Omega}$ is a metric family on $f^{\ast}L$. If $\psi=(\psi_{\omega})_{\omega\in\Omega}$ is a metric family on an invertible $\cO_X$-module $M$, then $\varphi+\psi:=(\varphi_{\omega}+\psi_{\omega})_{\omega\in\Omega}$ is a metric family on $L+M$. Furthermore, the metric family $\varphi$ is called \emph{plurisubharmonic} (\emph{psh} for short) if, for all $\omega\in \Omega$, $\varphi_{\omega}\in \PSH(L_{\omega})$ (cf. Definition \ref{def:psh_metric_unified}).

We mention an example of metric family. Let $V$ be a finite dimensional vector space over $K$ and $\xi=(\|\cdot\|_{\omega})_{\omega\in\Omega}$ be a norm family on $V$. Assume that there exists a projective morphism $f:X\to \bP(V)$ over $\Spec(K)$, then $(f\circ \pi)^{\ast}\xi$ is a metric family on $(f\circ \pi)^{\ast}E$ inducing by quotient a metric family on $L := f^{\ast}\cO_{V}(1)$ called \emph{quotient metric family} induced by $\overline{V}=(V,\xi)$.

The metric family $\varphi$ is called \emph{dominated} if there exist two invertible $\cO_{X}$-modules $L_1,L_2$, respectively endowed with quotient metrics $\varphi_1,\varphi_2$, such that $L=L_1-L_2$ and $\varphi=\varphi_1-\varphi_2$.

The metric family $\varphi$ is called \emph{measurable} if the following conditions are satisfied:
\begin{itemize}
	\item[(i)] for any closed point $P\in X$, the induced norm family $P^{\varphi}\varphi$ on $P^{\ast}L$ is measurable;
	\item[(ii)] let $\omega\in\Omega$ such that $\va_{\omega}$ is trivial. Then, for any valued extension $(L_{\omega},\va_{L,\omega})$ of transcendence degree $1$ of rational exponent of $(K,\va_{\omega})$, for any $K$-morphism $P : \Spec(L_{\omega}) \to X$, the norm family $P^{\ast}\varphi$ is measurable (we refer to \cite{ChenMori}, \S 6.1 for more details). 
\end{itemize}

$\overline{L}=(L,\varphi)$ is an \emph{adelic line bundle} on $X$ if $\varphi$ is both measurable and dominated. We denote by $\cM(L)$ the set of \emph{adelic} metrics on $L$, namely measurable and dominated families of metrics . Assume that $X$ is geometrically reduced and $\overline{L}$ is an adelic line bundle on $X$. For all $\omega\in\Omega$, let $(f_{\omega}(L))_{\omega} = H^{0}(X,L)\otimes_{K} K_{\omega}$ endowed with the supremum norm, denoted by $\|\cdot\|_{\varphi_{\omega}}$. We denote by $f_{\ast}(\varphi)$ the corresponding norm family on $H^0(X,L)$. If $\varphi$ is dominated, then $f_{\ast}(\varphi)$ is strongly dominated (\cite{ChenMori}, Theorem 6.1.13). If $K$ is countable and $\varphi$ is measurable, then $f_{\ast}(\varphi)$ is measurable (\cite{ChenMori}, Theorem 6.1.32).

\subsubsection{Adelic Cartier divisors}

Throughout this paragraph, we fix a geometrically integral normal projective scheme $\pi : X \to \Spec(K)$ over $\Spec(K)$. 

Let $D$ be a Cartier divisor on $X$. A \emph{Green function family} of $D$ is a family $g=(g_{\omega})_{\omega\in\Omega} $ such that, for all $\omega\in\Omega$, $g_{\omega}$ is a Green function of $D_{\omega}$. From \S \ref{subsub:Green_functions}, for all $\omega\in\Omega$, $g_{\omega}$ induces a continuous metric $\cO_{X_{\omega}}(D_{\omega}) := \cO_{X}(D) \otimes_{\cO_X} \cO_{X_{\omega}}$ denoted by $\varphi_{g_{\omega}}$. A pair $(D,g)$ is called an \emph{adelic Cartier divisor} if the metric family $\varphi =(\varphi_{g_{\omega}})_{\omega\in\Omega}$ is adelic. Let $\widehat{\Div}(X)$ be the abelian group of adelic Cartier divisors. We denote by $\widehat{C}^0(X)$ the set of Green function families $g$ of $0$ such that $(0,g)$ is an adelic Cartier divisor. It carries the structure of vector space over $\bR$.

\begin{definition}[\cite{ChenMori}, Definition 6.2.9]
Let $\bK$ be either $\bQ$ or $\bR$. Let $\widehat{\Div}_{\bK}(X)$ be the vector space over $\bK$ defined as the quotient of $\widehat{\Div}(X)\otimes_{\bZ} \bK$ by the vector subspace spanned by elements of the form
\begin{align*}
(0,g_1)\otimes \lambda_1 + \cdots (0,g_n)\otimes \lambda_n - (0,\lambda_1g_1+\cdots\lambda_ng_n),
\end{align*}
where $g_1,...,g_n \in \widehat{C}^0(X)$ and $\lambda_1,...,\lambda_n\in\bK$. The elements of $\widehat{\Div}_{\bK}(X)$ are called $\bK$\emph{-Cartier adelic divisors} on $X$. If $\overline{D}\in \widehat{\Div}_{\bK}(X)$, then we have a height function $h_{\overline{D}} : X(\overline{K}) \to \bR$.
\end{definition}

\begin{example}[\cite{ChenMori}, Example 2.9.11]
Let $s\in K(X)^{\times}$. Then, for all $\omega\in \Omega$, $-\ln|s|_{\omega}$ is a Green function of $\div(s)$. Moreover, the corresponding Green function family is measurable and dominated. Therefore
\begin{align*}
(s\in K(X)^{\times}) \mapsto (\div(s),(-\ln|s|_{\omega})_{\omega\in\Omega}) \in \widehat{\Div}(X)
\end{align*}
is a group homomorphism whose image is the set of \emph{principal adelic Cartier divisors}. This homomorphism induces a $\bK$-linear map $K(X)^{\times} \otimes_{\bZ}\bK \to \widehat{\Div}_{\bK}(X)$ whose image is the set of \emph{principal} $\bK$\emph{-Cartier adelic divisors}.
\end{example}

Let $(D,g)\in\widehat{\Div}_{\bK}(X)$ and let $s\in H^{0}_{\bK}(X,D)$. For all $\omega\in\Omega$, we denote
\begin{align*}
\|s\|_{g_{\omega}} := \displaystyle\sup_{x\in X_{\omega}^{\an}} \{(|s|_{\omega}\exp(-g_{\omega}))(x)\}.
\end{align*}
Then denote by $\xi_{g} := (\|\cdot\|_{g_{\omega}})_{\omega\in\Omega}$ the corresponding norm family on $H^{0}_{\bK}(X,D)$ (if $\bK$ is either $\bQ$ or $\bR$, $X$ is assumed to be normal so that $H^{0}_{\bK}(X,D)$ is a vector subspace of $k(X)$). Theorem 6.2.18 of \cite{ChenMori} then ensures that $\pi_{\ast}(D,g) := (H^{0}_{\bK}(X,D),\xi_{g})$ is an adelic vector bundle on $S$

\subsubsection{Arithmetic intersection product}

Let $X$ be a projective scheme over $\Spec(K)$ of dimension $d\geq 1$. An adelic line bundle $\overline{L}=(L,\varphi)$ on $X$ is called \emph{dpsh} if we can write $(L,\varphi) = (A_1,\varphi_1)-(A_2,\varphi_2)$ where $A_1,A_2$ are ample and $\varphi_1,\varphi_2$ are psh. For $i=0,...,d$, let $\overline{L_i}=(L_i,\varphi_i)$ be a dpsh adelic line bundle on $X$. Then an arithmetic intersection product $(\overline{L_0}\cdot \overline{L_1}\cdots \overline{L_d})_S$ is defined in \cite{ChenMori21}. This arithmetic intersection product is multi-linear with respect to the $\overline{L_i}$. 

\subsubsection{$\chi$-volume, arithmetic Hilbert-Samuel formula}
\label{subsub:chi_volume_arithmetic_Hilbert-Samuel}

Let $\bK$ be either $\bZ$, $\bQ$ or $\bR$. Throughout this paragraph, we fix a geometrically integral normal projective scheme of dimension $d$ $\pi : X \to \Spec(K)$ over $\Spec(K)$ and an adelic $\bK$-Cartier divisor $\overline{D}=(D,g)$ on $X$. The $\chi$\emph{-volume} of $\overline{D}$ is the superior limit
\begin{align*}
\widehat{\vol}_{\chi}(\overline{D})=\displaystyle\limsup_{n\to+\infty} \frac{\widehat{\deg}(\pi_{\ast}(n\overline{D}))}{n^{d+1}(d+1)!}.
\end{align*}
If $\bK=\bZ$ and $D=L$ is semiample and big, (\cite{ChenMoriHSP} Theorem-Definition 3.1.2) yields the convergence (to $\widehat{\vol}_{\chi}(\overline{L})$) of the sequence
\begin{align*}
\left(\frac{\widehat{\deg}(\pi_{\ast}(n\overline{L}))}{n^{d+1}(d+1)!}\right)_{n\in\bN_{>0}}.
\end{align*}
The adelic line bundle $\overline{L}$ is called \emph{relatively ample} if $L$ is ample and $\varphi$ is psh. In that case, (\cite{ChenMoriHSP} Theorem 4.5.1) gives the equality
\begin{align}
\label{eq:HS_arithmétique_rel_ample}
\widehat{\vol}_{\chi}(\overline{L}) = (\overline{L}^{d+1})_{S}.
\end{align}
This equality can be extended to the case where $L$ is only assumed to be semiample (\cite{Luo23}, Theorem 3.7).

In \cite{Luo23}, Luo proved the following continuity result for the $\chi$-volume. 
\begin{theorem}[\cite{Luo23}, Theorem 4.5]
Let $\overline{D}_1=(D_1,g_1),...,\overline{D}_r=(D_r,g_r)$ be adelic $\bQ$-Cartier divisors on $X$, where, for any $i=1,...,r$, $D_i$ is semiample. Then the function $\widehat{\vol}_{\chi}(\cdot)$ is continuous on 
\begin{align*}
C_{\bR}(\overline{D}_1,...,\overline{D}_r) := \left\{\lambda_1\overline{D}_1+\cdots+\lambda_r\overline{D}_r : \lambda_1,...,\lambda_r\in \bR_{\geq 0}\right\}.
\end{align*}
\end{theorem} 

\subsection{Complements on the arithmetic $\chi$-volume}
\label{sub:complement_chi_volume}

Throughout this paragraph, we fix a proper adelic curve $S=(K,(\Omega,\cA,\nu),(\va_{\omega})_{\omega\in\Omega})$ where $K$ is a perfect countable field and a geometrically integral normal projective scheme $X$ of dimension $d$ over $\Spec(K)$.

Let $L$ be an invertible $\cO_X$-module. For all $\omega\in \Omega$, recall that $K_{\omega}$ denotes the completion of $K$ with respect to $|\cdot|_{\omega}$, $X_{\omega} := X \otimes_{K} K_{\omega}$ and $L_{\omega} := L \otimes_{\cO_{X}} \cO_{X_{\omega}}$. Note that, for all $\omega\in\Omega$, $X_{\omega}$ is a geometrically integral normal projective scheme of dimension $d$ over $\Spec(K_{\omega})$. We fix a monomial order $\leq$ on $\bN^d$. We assume that there exists a regular rational point $p\in X(K)$. 

For any integer $n\geq 0$, let $V_n:=H^{0}(X,L^{\otimes n})$. Recall that in \S \ref{sub:Okounkov_bodies_and_slopes}, we have defined a $\bN^d$-filtration on $V_n$ and we denote by $(\gr^{\alpha}(V_n))_{\alpha\in \bN^d}$ its graded homogeneous components. Moreover, the non-zero homogeneous components in the previous decomposition are indexed by a certain finite semi-group $\Gamma(V_n)$. Finally, let $\gr(V_n) := \bigoplus_{\alpha\in \Gamma(V_n)} \gr^{\alpha}(V_n)$. 

We further assume that $\varphi\in \cM(L)$ is an adelic metric family on $L$. For any integer $n\geq 0$, the metric family induces a measurable and dominated norm family on $V_n$ denoted by $\xi_{n\varphi}$. For any $\alpha\in \Gamma(V_n)$, the subquotient norm family induced on $\gr^{\alpha}(V_n)$ yields a measurable and dominated norm family on $\gr^{\alpha}(V_n)$ denoted by $\xi_{n\varphi}^{\alpha}$. Finally, taking the orthogonal direct sum of all adelic vector bundles $(\gr^{\alpha}(V_n),\xi_{n\varphi}^{\alpha})$, for $\alpha$ we obtain an adelic vector bundle denoted by $\overline{\gr(V_n)} := (\gr(V_n), \widehat{\xi}_{n\varphi}=(\|\cdot\|_{\widehat{n\varphi}_{\omega}})_{\omega\in\Omega})$. Moreover, from (\cite{ChenMori}, Propositions 4.3.10 and 4.3.13), we have the following estimate.
\begin{align}
\label{eq:estimate_degree}
\forall n\in \bN, \quad \mid\widehat{\deg}(V_n,\xi_{n\varphi}) -\widehat{\deg}(\gr(V_n), \widehat{\xi}_{n\varphi})\mid \leq \frac{1}{2}\dim_{K}(V_n)\ln\dim_{K}(V_n)\nu(\Omega_{\ar}).
\end{align}

Recall that, for any integer $n\geq 0$, $\overline{\gr(V_n)}$ has a canonical basis $(s_{(n,\alpha)})_{\alpha\in\Gamma(V_n)}$. Let $\omega\in\Omega$. Recall that in \S \ref{sub:Okounkov_bodies_and_slopes}, a concave transform $G_{\varphi_{\omega}}$ has been defined on the Okounkov body $\Delta(L_{\omega})$. Moreover, we have
\begin{align*}
\widehat{\deg}(\gr(V_n)) &= -\int_{\Omega}\ln\|\bigwedge_{\alpha\in\Gamma(V_n)}s_{(n,\alpha)}\|_{\widehat{n\varphi}_{\omega},\det}\nu(\diff \omega)\\
& = \int_{\Omega}\widehat{\deg}_{\omega}(\gr(V_n)_{\omega},\widehat{n\varphi}_{\omega},\eta_{n,\omega})\nu(\diff \omega),
\end{align*}
where, for any $\omega\in \Omega$, $\eta_{n,\omega}$ denotes the monomial norm on $\gr(V_n)_{\omega}$ and $\widehat{\deg}_{\omega}$ denotes the local degree introduced in \S \ref{sub:Okounkov_bodies_and_slopes} (we refer to this paragraph for the details). 

Moreover, for any $\omega\in\Omega$, (\cite{ChenMaclean15}, Remark 4.4) yields the equality
\begin{align*}
\displaystyle\lim_{n\to\infty}\frac{\widehat{\deg}_{\omega}(\gr(V_n)_{\omega},\widehat{n\varphi}_{\omega},\eta_{n,\omega})}{n^{d+1}/(d+1)!} = \int_{\Delta(L_{\omega})} G_{\varphi_{\omega}}(x)\lambda(\diff x).
\end{align*}

Finally, the dominated convergence theorem implies the $\nu$-integrability of the LHS of the above equality and the estimate (\ref{eq:estimate_degree}) end the proof of the following proposition, which can be seen as an analogue of Proposition 2.13 in \cite{CM22_equidistribution} in our context. %TODO jusitifer le tcvd

\begin{proposition}
\label{prop:complement_chi_volume}
With the above notation, we have the equality
\begin{align}
\label{eq:complement_chi_volume}
\widehat{\vol}(L,\varphi) = \int_{\Omega}\int_{\Delta(L_{\omega})} G_{\varphi_{\omega}}(x)\lambda(\diff x)\nu(\diff \omega).
\end{align}
\end{proposition}

Finally, let $\overline{L}=(L,\varphi)$ be an adelic line bundle on $X$, we assume that $L$ is semiample. Then, for any $\omega\in\Omega$, the set $\PSH(L_{\omega})$ is non-empty. We denote by $P(\varphi)$ the metric family $(P(\varphi_{\omega}))_{\omega\in\Omega}$. In general, it is not clear that $P(\overline{L}) := (L,P(\varphi))$ is an adelic line bundle on $X$. Nonetheless, Proposition \ref{prop:Boucksom_6.4} (1) yields 
\begin{align*}
\forall s\in H^{0}(X,L),\quad \forall \omega\in\Omega, \quad \|s\|_{\varphi_{\omega}} =\|s\|_{P(\varphi_{\omega})}. 
\end{align*}
Hence the quantity 
\begin{align*}
\widehat{\vol}_{\chi}(P(\overline{L})) := \widehat{\vol}_{\chi}(\overline{L})
\end{align*}
is well-defined.

\subsection{Proof of Theorem \ref{th:global_differentiability_intro}}

Throughout this paragraph, we fix a proper adelic curve $S=(K,(\Omega,\cA,\nu),(\va_{\omega})_{\omega\in\Omega})$ where $K$ is a perfect countable field and a geometrically integral normal projective scheme $X$ of dimension $d$ over $\Spec(K)$. We assume that, for all $\omega\in\Omega$, the continuity of envelopes (Conjecture \ref{conj:BoucksomEriksson_7.30}) holds on $K_{\omega}$.

We start by introducing some notation. Let $L$ be an invertible $\cO_X$-module. For all $\omega\in \Omega$, recall that $K_{\omega}$ denotes the completion of $K$ with respect $\phi(\omega)$, $X_{\omega} := X \otimes_{K} K_{\omega}$ and $L_{\omega} := L \otimes_{\cO_{X}} \cO_{X_{\omega}}$. Note that, for all $\omega\in\Omega$, $X_{\omega}$ is a geometrically integral normal projective scheme of dimension $d$ over $\Spec(K_{\omega})$. 

%Ample case
Let $m\geq 1$ be an integer. For all $i=0,...,m$, let $\overline{L_i}:=(L_i,\varphi_i)$ where
\begin{itemize}
	\item[(i)] $L_i$ is a line bundle on on $X$ and $L:=L_0$ is ample;
	\item[(ii)] $\varphi_i\in \cM(L_i)$ is an adelic metric family and $\varphi:=\varphi_0$ is a psh metric family.
\end{itemize}
For all $i=0,...,m$, let $s_i$ be a non-zero rational section of $L_i$, $D_i:=\div(s_i)$ be its corresponding Cartier divisor and let $g_i$ be the family of Green functions associated to $\varphi_i$. By abuse of notation, for all $a\in \bR^{m}$, we denote $\overline{L_a}:=(L_a,\varphi_a)$, where
\begin{align*}
L_a := D_0 + \displaystyle\sum_{i=1}^{m} a_i D_i,\quad \varphi_a := g_0 + \sum_{i=1}^{m} a_i g_i.
\end{align*}
Let $\mathbf{L}=(\overline{L},\overline{L_1},...,\overline{L_m})$ and $O_{\mathbf{L}} := \{a\in \bR^m : L_a \text{ is ample}\}$. We will study the function
\begin{align*}
\fonction{\widehat{\vol}_{\chi}^{\mathbf{L}}}{O}{\bR,}{a}{\widehat{\vol}_{\chi}(\overline{L_a}).}
\end{align*}

\begin{theorem}
\label{th:global_differentiability}
For any family $\mathbf{L}$ as above, the function $\widehat{\vol}_{\chi}^{\mathbf{L}}$ is continuously differentiable on $O_{\mathbf{L}}$.
\end{theorem}

\begin{proof}

As in the proof of Theorem \ref{th:local_differentiability_general}, it is enough to prove that $\vol_{\chi}^{\mathbf{L}}$ admits a continuous partial derivatives along the direction $e_1$ at $0$. Let $\overline{A} := \overline{L_{1}} =(A,\Phi)$ and 
\begin{align*}
U := \{t\in \bR : te_1 \in O_{\mathbf{L}}\},
\end{align*}
it is an open subset of $\bR$.

\begin{lemma}
\label{lemma:reduction_global_differentiability}
It suffices to prove Theorem \ref{th:global_differentiability} in the case where:
\begin{itemize}
	\item[(i)] $A$ is very ample;
	\item[(ii)] there exists a regular rational point $p\in X(K)$;
	\item[(iii)] there exist a non-zero global section $s\in H^0(X,A)$, an irreducible component $E$ of $\div(s)$ and a coordinate system $(z_1,...,z_d)$ of $X$ in $p$ of $\div(s)$ such that $E$ is an effective Cartier divisor whose underlying subscheme is a geometrically integral normal subscheme of $X$ satisfying $E \nsubseteq B_{+}(L)$ and $z_1$ is a local equation of $E$ in $p$.
\end{itemize}  
\end{lemma}

\begin{proof}
Note that the proof of Lemma \ref{lemma:reduction_local_differentiability} is purely of algebro-geometric nature. Hence we can reproduce it \emph{mutatis mutandis}. 
\end{proof}

From now on, we use the same notation and hypotheses of Lemma \ref{lemma:reduction_global_differentiability} and assume that $\mathbf{L}=(\overline{L},\overline{A})$. For any $\omega\in \Omega$, let $E_{\omega} := E \otimes_{\Spec(K)} \Spec(K_{\omega})$, $s_\omega:= \pi_{\omega}^{\ast}s$ and $p_{\omega}:=\pi_{\omega}^{\ast}p$ where $\pi_{\omega}: X_{\omega} \to X$. Note that $A_{\omega},s_{\omega},E_{\omega}$ and $p_{\omega}$ satisfy conditions $(i)$-$(iii)$ of Lemma \ref{lemma:reduction_local_differentiability}. Let $\mathbf{L}_{\omega}:=(\overline{L_{1,\omega}},...,\overline{L_{m,\omega}})$ and set
\begin{align*}
\forall t \in U,\quad f^{\mathbf{L}_{\omega}}(t) = (d+1)\int_{\Delta(L_{\omega}+tA_{\omega})}G_{\varphi_{\omega}+t\Phi_{\omega}}(x)\diff x. 
\end{align*} From the proof of Theorem \ref{th:local_differentiability_general}, if we let , we have the equality
\begin{align*}
(f^{\mathbf{L}_{\omega}})'(0) &= (d+1)!\int_{\Delta(L_{|E_{\omega}})} G_{(L_{|E_{\omega}},P(\varphi_{\omega})_{|E_{\omega}})}(\alpha)\diff\alpha - (d+1)\int_{X_{\omega}^{\an}}\ln|s|_{\Phi_{\omega}}(\diff\dc P(\varphi_{\omega}))^{\wedge d} \\
&= (d+1)!\int_{\Delta(L_{|E_{\omega}})} G_{(L_{|E_{\omega}},\varphi_{\omega|E_{\omega}})}(\alpha)\diff\alpha - (d+1)\int_{X_{\omega}^{\an}}\ln|s|_{\Phi_{\omega}}(\diff\dc \varphi_{\omega})^{\wedge d},
\end{align*}
since $\varphi_{\omega}=P(\varphi_{\omega})$.

Let $t\in U$. Then conditions (i)-(iii) of Lemma \ref{lemma:reduction_local_differentiability} are satisfied by $s_{\omega},p_{\omega}$ at $L_t=L+tA$. Hence we obtain
\begin{align}
\label{eq:local_partial_derivative_segment}
(f^{\mathbf{L}_{\omega}})'(t) = (d+1)!\int_{\Delta(L_{|E_{\omega}}+tA_{|E_{\omega}})} G_{(L_{|E_{\omega}+tA_{|E_{\omega}})},P(\varphi_{\omega}+t\Phi_{\omega})_{|E_{\omega}})}(\alpha)\diff\alpha \\
- (d+1)\int_{X_{\omega}^{\an}}\ln|s_{\omega}|_{\Phi_{\omega}}(\diff\dc P(\varphi_{\omega}+t\Phi_{\omega}))^{\wedge d}\nonumber.
\end{align}

Fix $t_0\in U$ and let $(t_n)_{n\geq 1}\in [0,t_0]^{\bN}$ be a decreasing sequence with limit $0$. For all $n\geq 1$, for all $\omega\in \Omega$, denote 
\begin{align*}
A_{n}(\omega) := \frac{f^{\mathbf{L}_{\omega}}(t_n)-f^{\mathbf{L}_{\omega}}(0)}{t_n}.
\end{align*}
Theorem \ref{th:local_differentiability_general} yields the pointwise convergence of the sequence $(A_n)_{n\geq 1}$ to the function
\begin{align*}
A : (\omega\in \Omega) \mapsto (f^{\mathbf{L}_{\omega}})'(0).
\end{align*}
For all $n\geq 1$, for all $\omega\in\Omega$, the mean value theorem yields
\begin{align*}
|A_n(\omega)| \leq \displaystyle\sup_{t\in [0,t_0]} (f^{\mathbf{L}_{\omega}})'(t).
\end{align*}

We now state the fundamental domination result.

\begin{lemma}
\label{lemma:domination_condition}
The function
\begin{align*}
B : (\omega\in\Omega) \mapsto \displaystyle\sup_{t\in [0,t_0]} (f^{\mathbf{L}_{\omega}})'(t)
\end{align*}
is $\nu$-integrable.
\end{lemma}

\begin{proof}
Let $t\in [0,t_0]$ and $\omega\in\Omega$. Using the same notation as above, let 
\begin{align*}
B_{1,t}(\omega) &:= \int_{\Delta(L_{|E_{\omega}}+tA_{|E_{\omega}})} G_{(L_{|E_{\omega}+tA_{|E_{\omega}})},P(\varphi_{\omega}+t\Phi_{\omega})_{|E_{\omega}})}(\alpha)\diff\alpha,\\
B_{2,t}(\omega) &:= \int_{X_{\omega}^{\an}}\ln|s_{\omega}|_{\Phi_{\omega}}(\diff\dc P(\varphi_{\omega}+t\Phi_{\omega}))^{\wedge d}..
\end{align*} 
Equality (\ref{eq:local_partial_derivative_segment}) yields
\begin{align*}
\forall t\in [0,t_0],\quad \left|(f^{\mathbf{L}_{\omega}})'(t)\right| \leq \frac{1}{(d+1)!}|B_{1,t}(\omega)|+\frac{1}{d+1}|B_{2,t}(\omega)|.
\end{align*} 
Hence it is enough to show that, for all $i=1,2$, the function
\begin{align*}
B_i : (\omega\in\Omega) \mapsto \displaystyle\sup_{t\in[0,t_0]} |B_{i,t}(\omega)|,
\end{align*}
is $\nu$-integrable.

We first consider $B_1$. Let $t\in[0,t_0]$ and $\omega\in \Omega$. By homogeneity, we may assume that $L-A$ is ample and effective. Then Remark \ref{rem:concave_transform_finite_ample_case} gives the finiteness of the concave transforms $G_{(L_{\omega}+tA_{\omega},\varphi_{\omega}+t\Phi_{\omega})}$, $G_{(tL_{\omega}-tA_{\omega},t\varphi_{\omega}-t\Phi_{\omega})}$, $G_{(L_{\omega},\varphi_{\omega})}$ and $G_{(tA_{\omega},t\Phi_{\omega})}$. Moreover, the hypothesis on $A$ allows to use (\cite{Wilms22}, Theorem 1.2), which yields the equalities
\begin{align}
\label{eq:equality_Okounkov_bodies}
(1+t)\Delta(L_{\omega}) = \Delta(L_{\omega}+tA_{\omega}) + t\Delta(L_{\omega}-A_{\omega}) \quad\text{and}\quad \Delta(L_{\omega}+tA_{\omega}) = \Delta(L_{\omega}) + t\Delta(A_{\omega}).
\end{align}
Now (\cite{ChenMori}, Remark 6.3.22) implies 
\begin{align*}
&\forall (x,y) \in \Delta(L_{\omega}+tA_{\omega})\times\Delta(tL_{\omega}-tA_{\omega}),\\
& G_{(L_{\omega}+tA_{\omega},\varphi_{\omega}+t\Phi_{\omega})}(x) + G_{(tL_{\omega}-tA_{\omega},t\varphi_{\omega}-t\Phi_{\omega})}(y) \leq G_{\left((1+t)L_{\omega},(1+t)\varphi_{\omega}\right)}(x+y),\\
&\forall (\alpha,\beta)\in\Delta(L_{\omega})\times\Delta(tA_{\omega}),\\
&G_{(L_{\omega},\varphi_{\omega})}(\alpha) + G_{(tA_{\omega},t\Phi_{\omega})}(\beta) \leq G_{(L_{\omega}+tA_{\omega},\varphi_{\omega}+t\Phi_{\omega})}(\alpha+\beta).
\end{align*}
Combining these inequalities  with (\ref{eq:equality_Okounkov_bodies}) and Proposition \ref{prop:Nystrom_14.9}, we obtain the inequality
\begin{align*}
A_{1,t}(\omega) := \int_{\Delta(L_{|E_{\omega}})} G_{(L_{|E_{\omega}},P(\varphi_{\omega})_{|E_{\omega}})}(\alpha)\diff\alpha + t\int_{\Delta(A_{|E_{\omega}})} G_{(A_{|E_{\omega}},P(\Phi_{\omega})_{|E_{\omega}})}(\beta)\diff\beta \leq B_{1,t}(\omega)\\
 \leq (1+t)\int_{\Delta(L_{|E_{\omega}})} G_{(L_{|E_{\omega}},P(\varphi_{\omega})_{|E_{\omega}})}(\alpha)\diff\alpha -t\int_{\Delta(L_{|E_{\omega}}-A_{|E_{\omega}})} G_{(L_{|E_{\omega}}-A_{|E_{\omega}},P(\varphi_{\omega}-\Phi_{\omega})_{|E_{\omega}})}(\beta)\diff\beta =: A_{2,t}(\omega).
\end{align*}
Now Proposition \ref{prop:complement_chi_volume} implies that $(\omega\in \Omega) \sup_{t\in[0,t_0]} |A_{1,t}(\omega)|$ and $(\omega\in \Omega) \sup_{t\in[0,t_0]} |A_{2,t}(\omega)|$ are both $\nu$-integrable. Henceforth $B_1$ is $\nu$-integrable.

We now prove the integrability of $B_2$. For all $t\in [0,t_0]$, for all $\omega \in \Omega$, we have
\begin{align*}
|B_{2,t}(\omega)| \leq |\ln\|s\|_{\Psi_{\omega}}| \vol(L_{\omega}+tA_{\omega}) =  |\ln\|s\|_{\Psi_{\omega}}| \vol(L+tA)\leq C|\ln\|s\|_{\Psi_{\omega}}|,
\end{align*}
where $C>0$  is a constant independent on $\omega$ given by the continuity of the map $(t\in [0,t_0])\mapsto \vol(L+tA)$. As $\Psi \in \cM(A)$, the function $(\omega\in\Omega) \mapsto \ln\|s\|_{\Psi_{\omega}}$ is $\nu$-integrable (\cite{ChenMori}, Proposition 6.2.12). Hence $B_2$ is $\nu$-integrable.
\end{proof}

Combining Lemma \ref{lemma:domination_condition}, Proposition \ref{prop:complement_chi_volume} and the dominated convergence theorem we obtain the equality
\begin{align*}
\frac{\partial \widehat{\vol}_{\chi}^{\mathbf{L}}}{\partial x_1}(0) := \displaystyle\lim_{n\to+\infty} \frac{\widehat{\vol}_{\chi}^{\mathbf{L}}(t_ne_1)-\widehat{\vol}_{\chi}^{\mathbf{L}}(0)}{t_n} = \int_{\Omega} \frac{f^{\mathbf{L}_{\omega}}(t_n)-f^{\mathbf{L}_{\omega}}(0)}{t_n}\nu(\diff\omega) = \int_{\Omega}(f^{\mathbf{L}_{\omega}})'(0) \nu(\diff\omega).
\end{align*}

The above equality yields the right differentiability of $\widehat{\vol}_{\chi}^{\mathbf{L}}$ along the direction $e_1$. To show the left differentiability, one can combine (\ref{eq:left_derivative}) for all $\omega\in\Omega$ with the domination result of Lemma \ref{lemma:domination_condition}.

Ultimately, the function $\widehat{\vol}_{\chi}^{\mathbf{L}}$ admits a partial derivative along the first coordinate in any $a\in O_{\mathbf{L}}$. The function $\frac{\partial \widehat{\vol}_{\chi}^{\mathbf{L}}}{\partial x_1}$ is continuous on $O_{\mathbf{L}}$ (e.g. by applying Lemma \ref{lemma:domination_condition} along with the dominated convergence theorem). Hence $\widehat{\vol}_{\chi}^{\mathbf{L}}$ is continuously differentiable on $O_{\mathbf{L}}$.
\end{proof}

\begin{corollary}
\label{corollary:formula_global_derivative}
Let $\overline{L}=(L,\varphi)$ be a relatively ample line bundle on $X$. Let $\overline{M}=(M,\Phi)$ be any adelic line bundle on $X$. Then we have
\begin{align}
\label{eq:formula_global_derivative}
\displaystyle\lim_{t\to 0} \frac{\widehat{\vol}_{\chi}(\overline{L}+t\overline{M})-\widehat{\vol}_{\chi}(\overline{L})}{t} = (d+1)\left(\overline{L}^{d}\cdot \overline{M}\right)_{S}.
\end{align}
\end{corollary}

\begin{proof}
Using the same notation and hypotheses as in the proof of Theorem \ref{th:global_differentiability}, and we further assume that $\varphi$ is a psh metric family. Then we have
\begin{align*}
\frac{\partial \widehat{\vol}_{\chi}^{\mathbf{L}}}{\partial x_1}(0) &= (d+1)!\int_{\Omega}\left(\int_{\Delta(L_{|E_{\omega}})} G_{(L_{|E_{\omega}},\varphi_{\omega|E_{\omega}})}(\alpha)\diff\alpha\right)\nu(\diff\omega) - (d+1)\int_{\Omega}\left(\int_{X_{\omega}^{\an}}\ln|s_{\omega}|_{\Phi_{\omega}}(\diff\dc \varphi_{\omega})^{\wedge d}\right)\nu(\diff\omega)\\
&=  (d+1)\widehat{\vol}_{\chi}(L_{|E},\varphi_{|E}) - (d+1)\int_{\Omega}\left(\int_{X_{\omega}^{\an}}\ln|s_{\omega}|_{\Phi_{\omega}}(\diff\dc \varphi_{\omega})^{\wedge d}\right)\nu(\diff\omega)\\
& = (d+1)\left((L_{|E},\varphi_{|E})^{d}\right)_S - (d+1)\int_{\Omega}\left(\int_{X_{\omega}^{\an}}\ln|s_{\omega}|_{\Phi_{\omega}}(\diff\dc \varphi_{\omega})^{\wedge d}\right)\nu(\diff\omega)\\
& = (d+1)\left((L,\varphi)^{d}\cdot (A,\Phi)\right)_S,
\end{align*}
where: the first equality comes from Proposition \ref{prop:complement_chi_volume}, the second equality is given by the Hilbert-Samuel formula (cf. $(L_{|E},\varphi_{|E})$ is relatively ample) and the first one is given by definition of the arithmetic intersection product on $S$. Using the multilinearity of the arithmetic intersection intersection product, we obtain the desired explicit formula.
\end{proof}

\section{Logarithmic equidistribution over adelic curves}

The differentiability studied in \S \ref{sec:differentiability_adelic_curve} allows to prove a logarithmic equidistribution result. We generalise the results obtained in \cite{CLT09} in the framework of adelic curves.

Throughout this paragraph, we fix a proper adelic curve $S=(K,(\Omega,\cA,\nu),(\va_{\omega})_{\omega\in\Omega})$ where $K$ is a perfect countable field and a geometrically integral normal projective scheme $X$ of dimension $d$ over $\Spec(K)$. 

\subsection{Complements on maximal asymptotic slopes}
\label{sub:complement_maximal_asymptotic_slopes}

Let $(D,g)$ be an adelic $\bR$-Cartier divisor on $X$. Recall that the \emph{asymptotic maximal slope} of $(D,g)$, denoted by, $\widehat{\mu}_{\max}^{\asy}(D,g)$, is defined by
\begin{align*}
\widehat{\mu}_{\max}^{\asy}(D,g) := \displaystyle\limsup_{n\to+\infty}\frac{\widehat{\mu}_{\max}(H^{0}_{\bR}(X,nD),\xi_{ng})}{n}.
\end{align*}

\begin{lemma}
\label{lemma:CM23_prop_8.2.1_divisor}
Let $(D_1,g_1)$ and $(D_2,g_2)$ be adelic $\bR$-Cartier divisors on $E$ such that both $H^{0}_{\bR}(X,D_1)$ and $H^{0}_{\bR}(X,D_2)$ are non-zero. Then we have the inequality
\begin{align*}
\widehat{\mu}_{\max}((H^{0}_{\bR}(X,D_1),\xi_{g_1})\otimes(H^{0}_{\bR}(X,D_2),\xi_{g_2})) \geq \widehat{\mu}_{\max}((H^{0}_{\bR}(X,D_1),\xi_{g_1})) + \widehat{\mu}_{\max}((H^{0}_{\bR}(X,D_2),\xi_{g_2})) \\- \frac{3}{2}\nu(\Omega_{\ar})\ln\left(\dim_{K}\left(H^{0}_{\bR}(X,D_1)\right)\cdot\dim_{K}\left(H^{0}_{\bR}(X,D_2)\right)\right).
\end{align*}
\end{lemma}

\begin{proof}
The proof goes along the same lines of the one of (\cite{CM22_equidistribution}, Proposition 8.2.1). %TODO l'écrire? biblio
\end{proof}

\begin{proposition}
\label{prop:asymptotic_maximal_slope_R_divisor}
Assume that $D$ is big.
\begin{itemize}
	\item[(1)] Then the sequence $(\widehat{\mu}_{\max}(H^{0}_{\bR}(X,nD),\xi_{ng})/n)_{n\in \bN_{>0}}$ converges to $\widehat{\mu}_{\max}^{\asy}(D,g)$.
	\item[(2)] Then $\widehat{\mu}_{\max}^{\asy}(D,g)$ is equal to $\sup_{x\in \Delta(D)^{\circ}} G_{(D,g)}(x)$.
	\item[(3)] Let $a\in \bR_{\geq 0}$. Then $\widehat{\mu}_{\max}^{\asy}(aD,ag)=a\widehat{\mu}_{\max}^{\asy}(D,g)$.
	\item[(4)] Let $(D',g')$ be another adelic $\bR$-Cartier divisor on $X$ such that $D'$ is big. Then the following inequality holds:
	\begin{align*}
	\widehat{\mu}_{\max}^{\asy}(D+D',g+g') \geq \widehat{\mu}_{\max}^{\asy}(D,g) + \widehat{\mu}_{\max}^{\asy}(D',g').
	\end{align*}
\end{itemize}
\end{proposition}

\begin{proof}
(1) is a consequence of Lemma \ref{lemma:CM23_prop_8.2.1_divisor} as in the proof of (\cite{CM22_equidistribution}, Corollary 8.2.2).

(2) is exactly the statement of (\cite{ChenMori}, Lemma 6.4.17).

(3) Follows form (2) combined with the homogeneity of the concave transform.

Finally, (4) follows from Lemma \ref{lemma:CM23_prop_8.2.1_divisor} similarly as in the proof of (\cite{CM22_equidistribution}, Proposition 8.3.2).
\end{proof}

\subsection{Logarithmic equidistribution of closed points}
\label{sub:logarithmic_equidistribution_points}

Let $\overline{L}=(L,\varphi)$ be a relatively ample adelic line bundle on $X$. Let $(x_n)_{n\geq 0}$ be a sequence of closed points of $X$. We further assume that
\begin{itemize}
	\item[(i)] $(x_n)_{n\geq 0}$ is \emph{generic}, namely, for any closed subscheme $Z \subsetneq X$, the set $\{n\in\bN : x_n \in Z\}$ is finite;
	\item[(ii)] $(x_n)_{n\geq 0}$ is \emph{small}, namely the sequence of normalised heights $(h_{\overline{L}}(x_n))_{n\geq 0}$ converges to $h_{\overline{L}}(X)$.
\end{itemize}
For any $\omega\in\Omega$, for any $n\in \bN$, let $\delta_{\overline{L},x_n,\omega}$ denote the probability counting measure on the finite set $\{x_n\}^{\an}_{\omega}$. We also denote by $\delta_{\overline{L},X,\omega}$ the probability measure $(1/\vol(L))(\diff\dc\varphi)^{\wedge d}$.
Finally, let $\overline{M}=(M,\Phi)$ be an arbitrary adelic line bundle on $X$. Recall that the arithmetic intersection product $(\overline{L}^d \cdot \overline{M})_S$ is well defined.

The goal of this section is to prove Theorem \ref{th:logarithmic_equidistribution_intro}.

\begin{theorem}
\label{th:logarithmic equidistribution_points}
We assume that $M$ admits a non-zero global section $s$. Denote $D:=\div(s)$.
\begin{itemize}
	\item[(1)] We have
	\begin{align*}
	h_{\overline{L}}(D) \geq h_{\overline{L}}(X).
	\end{align*}
	\item[(2)] We further assume that $h_{\overline{L}}(D) = h_{\overline{L}}(X)$. Then we have 
	\begin{align}
	\label{eq:logarithmic_equidistribution_points}
	\displaystyle\lim_{n\to+\infty} \int_{\Omega}\left(\int_{X_{\omega}^{\an}} (-\ln|s|_{\Phi_{\omega}})\delta_{\overline{L},x_n,\omega}\right)\nu(\diff\omega) = \int_{\Omega}\left(\int_{X_{\omega}^{\an}} (-\ln|s|_{\Phi_{\omega}})\delta_{\overline{L},X,\omega}\right)\nu(\diff\omega).
	\end{align}
\end{itemize}
\end{theorem}

We start by recalling the following fact. 
\begin{lemma}
\label{lemma:pente_max_point}
Let $(L,\varphi)$ be an adelic line bundle on $X$ such that $L$ is pseudo-effective. Then, for any closed point $x\in X(\overline{K})$, we have
\begin{align*}
\widehat{\mu}_{\max}^{\asy}(\overline{L}_{|{x}}) = \widehat{\vol}_{\chi}(\overline{L}_{|x})=(\overline{L}_{|x})_{S}.
\end{align*}
\end{lemma}

The following limit result is the analogue of (\cite{CLT09}, Lemma 6.1) over adelic curves.
\begin{proposition}
\label{prop:CLT09_lemme_6.1_points}
With the above notation, we have
\begin{align*}
\displaystyle\lim_{n\to +\infty} h_{\overline{M}}(x_n) = \frac{\left(\overline{L}^d\cdot \overline{M}\right)_S}{\vol(L)} - \frac{d\widehat{\vol}_{\chi}(\overline{L})(L^{d-1}\cdot M)}{(d+1)\vol(L)^2} = \frac{\left(\overline{L}^d\cdot \overline{M}\right)_S}{\vol(L)} - \frac{dh_{\overline{L}}(X)(L^{d-1}\cdot M)}{\vol(L)}.
\end{align*}
\end{proposition}

\begin{proof}
Let $n\geq 0$ be an integer. We define a map 
\begin{align*}
\fonction{\psi_n}{U}{\bR}{t}{\widehat{\vol}_{\chi}(\overline{L}_{|{x_n}}+t\overline{M}_{|x_n}),}
\end{align*}
where $U$ is an open convex neighbourhood of $0$ in $\bR$ such that, for any $t\in U$, $L+tM$ is ample. Then Proposition \ref{prop:asymptotic_maximal_slope_R_divisor} (3) implies that, for any $n\geq 0$, $\psi_n$ is concave.

\begin{comment}
%TODO le lemme suivant est intuile dans le cas des points puisque on a déjà la propriété pour la pente maximale asymptotique.
\begin{lemma}
$\psi_n$ is concave. 
\end{lemma}

\begin{proof}
Let $t_1,t_2\in U$ and $\lambda\in [0,1]$. Then (\cite{Luo23}, Proposition 3.11) yields
\begin{align*}
\psi_n(\lambda t_1+(1-\lambda t_2)) &= \frac{\widehat{\vol}_{\chi}(\overline{L}_{|{Y_n}}+(\lambda t_1+(1-\lambda)t_2)\overline{M}_{|Y_n})}{(\dim(Y_n)+1)\vol(L_{|Y_n}+(\lambda t_1+(1-\lambda)t_2)M_{|Y_n})}\\ 
&\geq \frac{\widehat{\vol}_{\chi}(\lambda\overline{L}_{|{Y_n}}+\lambda t_1\overline{M}_{|Y_n})}{(\dim(Y_n)+1)\vol(\lambda L_{|Y_n}+\lambda t_1M_{|Y_n})} + \frac{\widehat{\vol}_{\chi}(\lambda\overline{L}_{|{Y_n}}+(1-\lambda) t_2\overline{M}_{|Y_n})}{(\dim(Y_n)+1)\vol(\lambda L_{|Y_n}+(1-\lambda) t_2M_{|Y_n})} \\
&= \lambda \psi_n(t_1) + (1-\lambda)\psi_n(t_2),
\end{align*}
where the last equality comes from the homogeneity of the $\chi$-volume and the (usual) volume.
\end{proof}
\end{comment}
We now define 
\begin{align*}
\forall t\in U,\quad \Psi(t) := \displaystyle\liminf_{n\to+\infty}\psi_n(t).
\end{align*} 
The above lemma implies that $\Psi$ is a concave function. From the Hilbert-Samuel formula for $\overline{L}$ and since the sequence $(x_n)_{n\geq 0}$ is small, we have 
\begin{align*}
\Psi(0) = \displaystyle\liminf_{n\to+\infty}\left(\overline{L}_{|x_n}\right)_S = \liminf_{n\to+\infty}h_{\overline{L}}(x_n) = h_{\overline{L}}(X) = \frac{\widehat{\vol}_{\chi}(\overline{L})}{(d+1)\deg_{L}(X)}.
\end{align*}

Now Proposition 8.10.1 in \cite{CM22_equidistribution} yields 
\begin{align*}
\forall t\in U, \quad \Psi(t) = \displaystyle\liminf_{n\to+\infty}\widehat{\mu}_{\max}^{\asy}(\overline{L}_{|{x_n}} + t\overline{M}_{|x_n}) \geq \widehat{\mu}_{\max}^{\asy}(\overline{L}+t\overline{M}) \geq \frac{\widehat{\vol}_{\chi}(\overline{L}+t\overline{M})}{(d+1)\vol(L+tM)}.
\end{align*}

For any $t\in U$, let 
\begin{align*}
f(t) = \frac{\widehat{\vol}_{\chi}(\overline{L}+t\overline{M})}{(d+1)\vol(L+tM)}.
\end{align*}
Then Corollary \ref{corollary:formula_global_derivative} implies that $f$ and the $\Psi_n$ are both differentiable in $O$, and we have
\begin{align*}
f'(0) = \frac{\left(\overline{L}^d\cdot \overline{M}\right)_S}{\vol(L)} - \frac{d\widehat{\vol}_{\chi}(\overline{L})(L^{d-1}\cdot M)}{(d+1)\vol(L)^2},\\
\forall n\geq 0, \quad \Psi'_n(0) = h_{\overline{M}}(x_n).
\end{align*}
Therefore (\cite{BermanBoucksom08}, Lemma 7.6) ensures that 
\begin{align*}
\displaystyle\lim_{n\to +\infty} h_{\overline{M}}(x_n) = \frac{\left(\overline{L}^d\cdot \overline{M}\right)_S}{\vol(L)} - \frac{d\widehat{\vol}_{\chi}(\overline{L})(L^{d-1}\cdot M)}{(d+1)\vol(L)^2}.
\end{align*}
\end{proof}

\begin{proposition}
\label{prop:CLT09_lemme_6.2}
Let $\Omega'\in \cA$. Assume that $M$ admits a global section $s$ which forms a regular meromorphic section of $M$ over $X$. Then we have
\begin{align*}
\displaystyle\liminf_{n\to + \infty} \int_{X_{\Omega'}^{\an}} (-\ln |s|_{\Phi}(x))\delta_{\overline{L},x_n,\Omega'}(\diff x) \geq \int_{X_{\Omega'}^{\an}}(-\ln|s|_{\Phi}(x))\delta_{\overline{L},X,\Omega'}(\diff x).
\end{align*}
Moreover, 
\begin{align*}
\displaystyle\liminf_{n\to+\infty} h_{\overline{M}}(x_n) \geq  \int_{X_{\Omega}^{\an}}(-\ln|s|_{\Phi}(x))\delta_{\overline{L},X,\Omega}(\diff x).
\end{align*}
\end{proposition}

\begin{proof}
We adapt the proof of (\cite{CLT09}, Lemme 6.2). Let $R\in \bR$. Then (\cite{CM22_equidistribution}, Corollary 8.11.4) yields
\begin{align*}
\displaystyle\lim_{n\to+\infty} \int_{\Omega'}\left(\int_{X_{\omega}^{\an}}\min(R,-\ln|s|_{\Phi_{\omega}})\delta_{\overline{L},x_n,\omega}\right)\nu(\diff\omega) = \int_{\Omega'}\left(\int_{X_{\omega}^{\an}}\min(R,-\ln|s|_{\Phi_{\omega}})\delta_{\overline{L},X,\omega}\right)\nu(\diff\omega).
\end{align*}
Thus, for any $R\in \bR$, the inequality
\begin{align*}
\displaystyle\liminf_{n\to + \infty}  \int_{\Omega'}\left(\int_{X_{\omega}^{\an}}-\ln|s|_{\Phi_{\omega}}\delta_{\overline{L},x_n,\omega}\right)\nu(\diff\omega) \geq \int_{\Omega'}\left(\int_{X_{\omega}^{\an}}\min(R,-\ln|s|_{\Phi_{\omega}})\delta_{\overline{L},X,\omega}\right)\nu(\diff\omega)
\end{align*}
holds. Since $(\omega\in\Omega)\mapsto \int_{X_{\omega}^{\an}}-\ln|s|_{\Phi_{\omega}}\delta_{\overline{L},X,\omega}$ is $\nu$-integrable, the dominated convergence theorem implies the first assertion by taking the limit $R\to+\infty$.

The second assertion is a consequence of the first one with $\Omega'=\Omega$ combined with the fact that , for any $n\in\bN$, we have
\begin{align*}
h_{\overline{M}}(x_n) =  \int_{\Omega}\left(\int_{X_{\omega}^{\an}}-\ln|s|_{\Phi_{\omega}}\delta_{\overline{L},x_n,\omega}\right)\nu(\diff\omega).
\end{align*}
\end{proof}

In the following, for any $\Omega'\in \cA$, we denote
\begin{align*}
\epsilon_{\Omega'} := \displaystyle\liminf_{n\to + \infty}  \int_{\Omega'}\left(\int_{X_{\omega}^{\an}}-\ln|s|_{\Phi_{\omega}}\delta_{\overline{L},x_n,\omega}\right)\nu(\diff\omega) - \int_{\Omega'}\left(\int_{X_{\omega}^{\an}}-\ln|s|_{\Phi_{\omega}}\delta_{\overline{L},X,\omega}\right)\nu(\diff\omega).
\end{align*}
Note that Proposition \ref{prop:CLT09_lemme_6.2} can be reformulated as, for any $\Omega'\in\cA$, $\epsilon_{\Omega'}\geq 0$. We are now ready to give the proof of Theorem \ref{th:logarithmic equidistribution_points}.

\begin{proof}

\begin{lemma}
\label{lemme:réduction_hauteur_nulle} 
To prove Theorem \ref{th:logarithmic equidistribution_points}, we may assume that $h_{\overline{L}}(X) = 0$. 
\end{lemma}

\begin{proof}
For any $C\in\bR$, let $\varphi+C$ be the metric family on $L$ defined as follows: for any $\omega\in\Omega$, $(\varphi+C)_{\omega}$ is $\varphi_{\omega}$ if $\omega\in\Omega_{\um}$ and $\varphi_{\omega} + C$ if $\omega\in\Omega_{\ar}$. We denote $\overline{L}(C):=(L,\varphi+C)$. Let $C := -h_{\overline{L}}(X)/\nu(\Omega_{\ar})$. 
By a direct computation, we have
\begin{align*}
\forall n\in \bN,\quad \int_{\Omega}\left(\int_{X_{\omega}^{\an}}-\ln|s|_{\Phi_{\omega}}\delta_{\overline{L}(C),x_n,\omega}\right)\nu(\diff\omega) - \int_{\Omega}\left(\int_{X_{\omega}^{\an}}-\ln|s|_{\Phi_{\omega}}\delta_{\overline{L}(C),X,\omega}\right)\nu(\diff\omega)\\
= \int_{\Omega}\left(\int_{X_{\omega}^{\an}}-\ln|s|_{\Phi_{\omega}}\delta_{\overline{L},x_n,\omega}\right)\nu(\diff\omega) - \int_{\Omega}\left(\int_{X_{\omega}^{\an}}-\ln|s|_{\Phi_{\omega}}\delta_{\overline{L},X,\omega}\right)\nu(\diff\omega).
\end{align*}
Moreover, 
\begin{align*}
h_{\overline{L}(C)}(X) = \frac{(\overline{L}(C)^{d+1})_{S}}{(d+1)\deg_{L}(X)} = \frac{(\overline{L}^{d+1})_{S}}{(d+1)\deg_{L}(X)} + \frac{(\overline{\cO_X}(C)\cdot\overline{L}^{d})_{S}}{\deg_{L}(X)} = h_{\overline{L}}(X) + C\nu(\Omega_{\ar}) = 0.
\end{align*}
Hence the lemma.
\end{proof}

We assume that $h_{\overline{L}}(X)=0$. Let $D=\div(s)$, By definition of the arithmetic intersection product, we have  
\begin{align*}
\frac{(\overline{L}_{|D}^{d})_{S}}{\deg_{L}(X)} &= \frac{(\overline{L}^d\cdot \overline{M})_{S}}{\deg_{L}(X)} - \int_{\Omega}\left(\int_{X_{\omega}^{\an}}-\ln|s|_{\Phi_{\omega}}\delta_{\overline{L},X,\omega}\right)\nu(\diff\omega)\\
& = \displaystyle\lim_{n\to+\infty}h_{\overline{M}}(x_n)-\int_{\Omega}\left(\int_{X_{\omega}^{\an}}-\ln|s|_{\Phi_{\omega}}\delta_{\overline{L},X,\omega}\right)\nu(\diff\omega)\\
&\geq \epsilon_{\Omega},
\end{align*}
where we have used Proposition \ref{prop:CLT09_lemme_6.1_points} in the second equality and Proposition \ref{prop:CLT09_lemme_6.2} in the last inequality. 
We deduce $h_{\overline{L}}(D)\geq 0 = h_{\overline{L}}(X)$. In the $h_{\overline{L}}(D) = h_{\overline{L}}(X)=0$ case, we obtain $\epsilon_{\Omega}=0$. 
\end{proof}

\bibliographystyle{alpha}
\bibliography{biblio}

\newcommand{\etalchar}[1]{$^{#1}$}
\begin{thebibliography}{BGGJ{\etalchar{+}}20}

\bibitem[AT84]{AT84}
Herbert~J. Alexander and B.A. Taylor.
\newblock Comparison of two capacities in ?n.
\newblock {\em Mathematische Zeitschrift}, 186:407--417, 1984.

\bibitem[Bal22]{Ballay22}
Fran\c{c}ois Balla\"{y}.
\newblock Arithmetic okounkov bodies and positivity of adelic cartier divisors,
  2022.

\bibitem[BB08]{BermanBoucksom08}
Robert~J. Berman and S\'{e}bastien Boucksom.
\newblock Growth of balls of holomorphic sections and energy at equilibrium.
\newblock {\em Inventiones mathematicae}, 181:337--394, 2008.

\bibitem[BC11]{BoucksomChen}
S{\'{e}}bastien Boucksom and Huayi Chen.
\newblock Okounkov bodies of filtered linear series.
\newblock {\em Compositio Mathematica}, 147(4):1205--1229, may 2011.

\bibitem[BE21]{BoucksomEriksson}
S{\'e}bastien Boucksom and Dennis Eriksson.
\newblock {Spaces of norms, determinant of cohomology and Fekete points in
  non-Archimedean geometry}.
\newblock {\em {Advances in Mathematics}}, 378:107501, February 2021.

\bibitem[BFJ09]{BFJ09}
S\'{e}bastien Boucksom, Charles Favre, and Mattias Jonsson.
\newblock Differentiability of volumes of divisors and a problem of teissier.
\newblock {\em Journal of algebraic geometry}, 18(2):279--308, 2009.

\bibitem[BGGJ{\etalchar{+}}20]{BJ20}
Jos{\'{e}}~Ignacio Burgos~Gil, Walter Gubler, Philipp Jell, Klaus
  K{\"{u}}nnemann, Florent Martin, and Robert Lazarsfeld.
\newblock {Differentiability of non-archimedean volumes and non-archimedean
  Monge-Amp{\`{e}}re equations (with an appendix by Robert Lazarsfeld)}.
\newblock 2020.

\bibitem[BGM20]{BGM20}
S\'{e}bastien Boucksom, Walter Gubler, and Florent Martin.
\newblock Differentiability of relative volumes over an arbitrary
  non-archimedean field, 2020.

\bibitem[BGM21]{Boucksom21}
S{\'{e}}bastien Boucksom, Walter Gubler, and Florent Martin.
\newblock Non-archimedean volumes of metrized nef line bundles.
\newblock {\em {\'{E}}pijournal de G{\'{e}}om{\'{e}}trie Alg{\'{e}}brique},
  Volume 5, oct 2021.

\bibitem[BJ18]{BoucksomJonsson18}
S{\'e}bastien Boucksom and Mattias Jonsson.
\newblock {Singular semipositive metrics on line bundles on varieties over
  trivially valued fields}.
\newblock working paper or preprint, February 2018.

\bibitem[BJ21]{BoucksomJonsson21}
S{\'e}bastien Boucksom and Mattias Jonsson.
\newblock {Global pluripotential theory over a trivially valued field}.
\newblock working paper or preprint, November 2021.

\bibitem[Bou71]{BouTG}
Nicolas Bourbaki.
\newblock {\em Topologie G\'en\'erale}.
\newblock Springer-Verlag Berlin Heidelberg, 1971.

\bibitem[BT82]{BT82}
Eric Bedford and B.~A. Taylor.
\newblock {A new capacity for plurisubharmonic functions}.
\newblock {\em Acta Mathematica}, 149(1), 1982.

\bibitem[Che11]{Chen11}
Huayi Chen.
\newblock Differentiability of the arithmetic volume function.
\newblock {\em Journal of the London Mathematical Society}, 84(2):365--384,
  2011.

\bibitem[CL06]{ChambertLoir06}
Antoine Chambert-Loir.
\newblock {Mesures et {\'{e}}quidistribution sur les espaces de Berkovich}.
\newblock 2006(595):215--235, 2006.

\bibitem[CLD12]{CLD12}
Antoine Chambert-Loir and Antoine Ducros.
\newblock Formes diff\'erentielles r\'eelles et courants sur les espaces de
  berkovich, 2012.

\bibitem[CLT09]{CLT09}
Antoine Chambert-Loir and Amaury Thuillier.
\newblock Mesures de mahler et \'{e}quidistribution logarithmique.
\newblock {\em Annales de l'institut Fourier}, 59(3):977--1014, 2009.

\bibitem[CM15]{ChenMaclean15}
Huayi Chen and Catriona Maclean.
\newblock {Distribution of logarithmic spectra of the equilibrium energy}.
\newblock {\em {manuscripta mathematica}}, 146(3-4):365--394, 2015.
\newblock The english version of ''R{\'e}partition des spectres logarithmiques
  de l'{\'e}nergie {\`a} l'{\'e}quilibre''.

\bibitem[CM19]{ChenMori}
Huayi Chen and Atushi Moriwaki.
\newblock {\em Arakelov geometry over adelic curves}, volume 2258 of {\em
  Lectures Notes in Mathematics}.
\newblock Springer Singapore, 2019.

\bibitem[CM21]{ChenMori21}
Huayi Chen and Atsushi Moriwaki.
\newblock Arithmetic intersection theory over adelic curves, 2021.

\bibitem[CM22a]{CM22_equidistribution}
Huayi Chen and Atsushi Moriwaki.
\newblock Equidistribution theorem over an adelic curve, 2022.

\bibitem[CM22b]{ChenMoriHSP}
Huayi Chen and Atsushi Moriwaki.
\newblock Hilbert-samuel formula and positivity over adelic curves, 2022.

\bibitem[Fan22]{Fang22}
Yanbo Fang.
\newblock Metrised ample line bundles in {non-Archimedean} geometry.
\newblock {\em Annales de l'Institut Fourier}, 2022.

\bibitem[Fle99]{Flenner99}
Hubert Flenner.
\newblock {\em Joins and intersections / H. Flenner, L. O'Carroll, W. Vogel}.
\newblock Springer monographs in mathematics. Springer, Berlin, 1999.

\bibitem[GK17]{GK17}
Walter Gubler and Klaus K\"unnemann.
\newblock A tropical approach to nonarchimedean arakelov geometry.
\newblock {\em Algebra and Number Theory}, 11(1):77--180, Jan 2017.

\bibitem[GS90]{GilletSoule90}
Henri Gillet and Christophe Soul{\'{e}}.
\newblock {Arithmetic intersection theory}.
\newblock {\em Publications Math{\'{e}}matiques de l'Institut des Hautes
  Scientifiques}, 72(1), 1990.

\bibitem[Iko13]{Ikoma13}
Hideaki Ikoma.
\newblock Boundedness of the successive minima on arithmetic varieties.
\newblock {\em Journal of algebraic geometry}, 22(2):249--302, 2013.

\bibitem[Iko18]{Ikoma18}
Hideaki Ikoma.
\newblock Differentiability of the arithmetic volume function along the base
  conditions, 2018.

\bibitem[Jou83]{Jouanolou83}
Jean-Pierre Jouanolou.
\newblock {\em {Th{\'{e}}or{\`{e}}mes de Bertini et applications / Jean-Pierre
  Jouanolou}}.
\newblock Progress in mathematics. Birkh{\"{a}}user, Boston, 1983.

\bibitem[KK12]{KK12}
Kiumars Kaveh and A.~G. Khovanskii.
\newblock {Newton-Okounkov bodies, semigroups of integral points, graded
  algebras and intersection theory}.
\newblock {\em Annals of Mathematics}, 176(2), 2012.

\bibitem[Laz04]{LazarsfeldI04}
Robert Lazarsfeld.
\newblock {\em {Positivity in Algebraic Geometry I}}.
\newblock 2004.

\bibitem[LM09]{LazarsfeldMustata09}
Robert Lazarsfeld and Musta\textcommabelow{t}\u{a} Mircea.
\newblock Convex bodies associated to linear series.
\newblock {\em Annales scientifiques de l'\'Ecole Normale Sup\'erieure}, Ser.
  4, 42(5):783--835, 2009.

\bibitem[Luo23]{Luo23}
Wenbin Luo.
\newblock Continuous extension and birational invariance of $\chi$-volume over
  an adelic curve, 2023.

\bibitem[Mor06]{Moriwaki06}
Atsushi Moriwaki.
\newblock Continuity of volumes on arithmetic varieties.
\newblock {\em Journal of Algebraic Geometry}, 18:407--457, 2006.

\bibitem[Nys14]{Nystrom13}
David~Witt Nystr\"om.
\newblock Transforming metrics on a line bundle to the okounkov body.
\newblock {\em Annales scientifiques de l'\'{E}cole Normale Sup\'{e}rieure},
  47(6):1111--1161, 2014.

\bibitem[Oko96]{Okounkov96}
Andrei Okounkov.
\newblock {Brunn-Minkowski inequality for multiplicities}.
\newblock {\em Inventiones Mathematicae}, 125(3), 1996.

\bibitem[Poi13]{Poineau13b}
J\'er\^ome Poineau.
\newblock Les espaces de {Berkovich} sont ang\'eliques.
\newblock {\em Bulletin de la Soci\'et\'e Math\'ematique de France},
  141(2):267--297, 2013.

\bibitem[Rum00]{Rumely00}
Robert~S. Rumely.
\newblock {\em Existence of the sectional capacity / Robert Rumely, Chi Fong
  Lau, Robert Varley}.
\newblock Memoirs of the American Mathematical Society. American mathematical
  society, Providence (R.I.), 2000.

\bibitem[SUZ97]{SUZ97}
L~Szpiro, E~Ullmo, and S~Zhang.
\newblock Equir\'epartition des petits points.
\newblock {\em Inventiones mathematicae}, 127(2):337--, 1997.

\bibitem[Wil22]{Wilms22}
Robert Wilms.
\newblock On the additivity of newton-okounkov bodies, 2022.

\bibitem[Yua08]{Yuan08}
Xinyi Yuan.
\newblock Big line bundles over arithmetic varieties.
\newblock {\em Invent. Math.}, (173):603--649, 2008.

\end{thebibliography}

\end{document}